\documentclass[11pt,letterpaper]{amsart}
\usepackage{amssymb,amsthm,amsmath}
\usepackage[T1]{fontenc}
\usepackage{lmodern}
\usepackage{bm}%\bm{..} for bold math
\usepackage{esint}%\fint for dashed integrals
\usepackage[hidelinks]{hyperref} %\hypersetup{colorlinks = false, linkbordercolor = myred, citebordercolor = mygreen}
\usepackage{xcolor} \definecolor{myred}{RGB}{204 0 0} \definecolor{mygreen}{RGB}{0 153 76}

\usepackage{tikz}
\usepackage{graphicx}
\usepackage{subcaption}
\usepackage{enumitem}%to set indentation in lists
\usepackage[noabbrev]{cleveref}

\usepackage[backend=biber,style=alphabetic,sorting=nyt,doi=true,isbn=false,url=false]{biblatex}%Imports biblatex package
\renewbibmacro{in:}{}% Suppress "In:" before journal titles

\AtEveryBibitem{%
	\clearfield{number}%
	\clearfield{pages}%
	\clearfield{volume}%
	\clearfield{issn}%
	\clearfield{fjournal}%
	\clearfield{language}%
	\clearfield{keywords}%
	\clearfield{month}%
	\clearfield{note}%
	\clearfield{zbl}%
	\clearfield{zbMATH}%
	\clearfield{url}% 
}
\newtheorem{theorem}{Theorem}[section]

\newtheorem{lemma}[theorem]{Lemma}
\newtheorem{proposition}[theorem]{Proposition}

\theoremstyle{definition}

\newtheorem{remark}[theorem]{Remark}
\newtheorem{example}[theorem]{Example}

\AddToHook{env/conjecture/begin}{\crefalias{theorem}{conjecture}}
\AddToHook{env/lemma/begin}{\crefalias{theorem}{lemma}}
\AddToHook{env/proposition/begin}{\crefalias{theorem}{proposition}}
\AddToHook{env/corollary/begin}{\crefalias{theorem}{corollary}}
\AddToHook{env/definition/begin}{\crefalias{theorem}{definition}}
\AddToHook{env/remark/begin}{\crefalias{theorem}{remark}}
\AddToHook{env/example/begin}{\crefalias{theorem}{example}}

\numberwithin{equation}{section}

\newcommand{\even}{\mathrm{even}}
\newcommand{\odd}{\mathrm{odd}}
\newcommand*{\de}{\partial}

\newcommand*{\R}{\mathbb{R}}

\renewcommand*{\S}{\mathbb{S}}
\newcommand*{\eps}{\epsilon}
\newcommand*{\ind}{\mathcal{X}}
\newcommand*{\weak}{\rightharpoonup}

\newcommand*{\loc}{{\text{\upshape{loc}}}}
\newcommand*{\lap}{\Delta}

\newcommand{\scalar}[1]{\left<#1\right>}
\DeclareMathOperator{\supp}{spt}
\DeclareMathOperator{\divergence}{div}

\DeclareMathOperator{\osc}{osc}

\DeclareMathOperator{\dist}{dist}

\DeclareMathOperator{\tr}{tr}
\DeclareMathOperator{\reg}{Reg}

\DeclareMathOperator{\id}{id}

\author{Giacomo Colombo}
\address{ETH Z\"urich, Switzerland}
\email{gcolom@ethz.ch}
\author{Federico Franceschini}
\address{Stanford University, USA}
\email{ffede@stanford.edu}
\date{\today}
\title[Semiconvexity and curvature bounds]{Mean curvature bounds for the obstacle problem}
\begin{document}
	\begin{abstract}
		We prove optimal $(n-1)$-semiconvexity estimates for solutions to the classical obstacle problem, possibly with a smooth right-hand side. We deduce that the mean curvature of the free boundary is universally bounded above in all dimensions. Existing examples show that full curvature bounds fail in dimensions three and higher. Noticeably, these results are novel even for a constant right-hand side in the $n=2$ case. 
	\end{abstract}
	
	\maketitle
	\tableofcontents
	
	\section{Introduction and Main Results}
	\subsection{}
	A function $u\in C^{1,1}(B_1)$ solves the obstacle problem if 
	\begin{equation}\label{eq:obstacle}
		\Delta u=f\chi_{\{u>0\}}\quad \text{ and }\quad u\ge0\qquad\text{in }B_1\subset\R^n,
	\end{equation}
	for some positive smooth function $f$. Such functions can be constructed by minimizing the convex functional
	\[
	v\mapsto\int_{B_1}\tfrac{1}{2}|\nabla v(x)|^2 +f(x)\max\{v(x);0\}\, dx,
	\]
	among functions $v$ taking some prescribed nonnegative boundary datum on $\de B_1$.
	This is perhaps the most classical free boundary problem, and it arises in a variety of situations; see \cite{DuvatLions,Rodrigues,Friedman,PSU2012,Serfaty}.
	
	Mathematically, the main challenge of this problem is the study of the \textit{free boundary} $\de\{u>0\}$, which has been the subject of extensive research in the past decades. The modern theory started with the seminal paper \cite{Caffarelli77}, where it is proven that the free boundary splits into a regular part (where the free boundary is a smooth hypersurface) and a singular part (where the contact set $\{u=0\}$ has zero density). Although the singular set can be pretty wild (see \cite{Schaeffer}), a lot of progress has recently been made towards understanding its local structure \cite{Weiss,CSVObstacle,FS19,SY23}.
	
	However, little is known about the behavior of the free boundary as a whole, even assuming it is a priori smooth. For example, on the one hand, being a free boundary of a \textit{global} solution of the obstacle problem is an extremely rigid condition \cite{GlobalSolutionsObstacle,2dglobalsolutions}; on the other hand, \textit{any} analytic hypersurface is a free boundary of a suitable solution of the obstacle problem in a sufficiently small ball.
	It is natural to ask ``what happens in between'':
	\begin{gather*}
		\text{\itshape How restrictive is it for a hypersurface to be the free boundary} \\
		\text{\itshape in $B_1$ of a solution to \eqref{eq:obstacle} defined in $B_R$ with $R \gg 1$?}
	\end{gather*}
	A priori curvature estimates are a natural way of answering this question quantitatively. Note that any curvature bound is necessarily one-sided, since it is known by explicit examples \cite[Figure 2]{Schaeffer} that the set $\{u=0\}$ can have outer cusps, even in 2D, so there is no hope to bound the curvature of $\de \{u>0\}$ from below.
	In \cite{SchaefferCurvatureBound}, Schaeffer indeed proved some curvature bounds in 2 dimensions, under some restrictive topological assumptions on the free boundary that could be verified only after an unquantified rescaling, thus making his bounds not universal. 
	Additionally, the following example of Figalli and Serra shows that an upper bound on the full second fundamental form is not possible if $n\ge3$:
	\begin{example}\label{example:anomalous}
		There is a solution $u^*$ of \eqref{eq:obstacle} with $f\equiv1$ in $B_1\subset \R^3$ such that
		\[
		\{u^*>0\}=\big\{x\in B_1 : {\textstyle\sqrt{x_2^2+x_3^2}}> \varphi(|x_1|) \big\},
		\]
		for some positive increasing function $\varphi\in C^\infty((0,1]),$ with $\varphi(0+)=\varphi'(0+)=0,$ such that 
		\[
		[\varphi']_{C^{\eps}((0,1])}=+\infty \quad\text{ for all }\eps>0,
		\]
		see \cite[Equation (A.3)]{FS19}.
		So, not only the curvature of $\de\{u^*>0\}$ along $\{x_3=0\}$ is unbounded, but, as $x\to0$, it explodes faster than $+|x|^{\eps-1}$. 
	\end{example}
	Notice that in this example, the other normal curvature of $\de\{u^*>0\}$ is instead exploding to $-\infty$ faster than $-|x|^{\eps-2}$. So the sum of the two is bounded above. Indeed, J. Spruck suggested that in higher dimensions one could hope to bound the mean curvature (see \cite[Section 4]{SchaefferCurvatureBound}). This is the main result of our work.
	
	\begin{theorem}\label{theorem:viscosity bound}
		Let $u$ solve \eqref{eq:obstacle} for some positive right-hand side $f\in C^{1,\alpha}(B_1)$. Then $\de \{u>0\}$ has mean curvature bounded above in $B_{1/2}$ by a constant $C$ that depends only on 
		\[
		n,\qquad \inf_{B_{1}} f>0,\qquad 
		\alpha>0,\qquad \|f\|_{C^{1,\alpha}(B_1)}.
		\]
	\end{theorem}
	
	Where $\de\{u>t\}$ is a regular surface, our convention is that the mean curvature is $-\divergence(\nabla u /|\nabla u|)$, thus, if $\{u>t\}$ is convex, then the mean curvature is positive.
	
	In general, our bound can be intended in the viscosity sense: for each $z\in \de\{u>0\}\cap B_{1/2}$ with the property that
	\[
	\exists\, E\,\text{ open with }\de E\in C^2\text{ and }\rho>0 \text{ s.t. } z\in \de E\text{ and } \{u>0\}\cap B_\rho(z) \subset E,
	\]
	it must hold that, letting $\nu_E$ be the outer unit normal,
	\begin{equation}\label{eq:viscositybound}
		H_{\de E}(z):=\divergence\nu_E\le C_0.
	\end{equation}
	
	When $n=2$, \Cref{theorem:viscosity bound} gives an a priori upper bound on the full curvature of the free boundary in the viscosity sense, independently of the solution. We emphasize that, in this sense, the result is already new when $n=2$ and $f\equiv1$. 
	It is also, to our knowledge, the first a-priori estimate on the free boundary in higher dimensions.
	
	\begin{remark}[Geometric interpretation]\label{rmk:geometric interpretation}
		The fact that the constant $C_0$ in \Cref{theorem:viscosity bound} is universal implies that, if we zoom in a small ball $B_{r}(x_\circ)\subset B_{1/2}$, we see a free boundary with mean curvature $\le C_0 r$. This is saying that, the picture of any free boundary in $B_1$, irrespectively of the boundary datum, does not contain in $B_{1/2}$ balls of size $\le 1/C_0$ in which the positivity set has an ``outer bump'' of definite size. 
		Changing perspective, if instead $u$ solves \eqref{eq:obstacle} in $B_R$, with $R\gg1$, \Cref{theorem:viscosity bound} gives that $\de\{u>0\} \cap B_{1/2}$ has mean curvature bounded above by $C_0/R$, by some universal $C_0$.
	\end{remark}
	
	\begin{remark}[Analytic interpretation]
		Even though there is a well developed regularity theory for free boundaries, as a geometric object a free boundary does not satisfy a local PDE by itself. Yet, from this point of view, \Cref{theorem:viscosity bound} gives a one-sided second order constraint satisfied by any free boundary.
	\end{remark}
	
	\begin{remark}[No exterior ball property]\label{rmk:exterior ball}
		Using complex-analytic tools, Gustafsson and Sakai \cite{GustSak03,GustSak04} showed that when $n=2$ and $f\equiv 1$ the contact set $\{u=0\}$ has the exterior $r$-ball property\footnote{\,Meaning that, there is $r>0$ such that
			\[\forall z\in\de\{u>0\}\cap B_{1/2},\exists y\in B_{1/2} : B_r(y)\subset\{u>0\}\text{ and }\de B_r(y)\cap \de\{u>0\} =\{z\}.
			\] } for a non-universal $r=r(u)>0.$ This does not follow from (nor imply) \Cref{theorem:viscosity bound}, because this phenomenon is in fact due to analytic rigidity: in \Cref{ex:no ball property} below we construct a $C^\infty$ obstacle in 2D and a solution $u_0$ such that, for each $\rho>0$, the set $\{u_0=0\}\cap B_{1/2}$ has a point which cannot be touched from the outside by a ball of radius $\rho$. 
		
		When $n\ge3$, the exterior sphere property already fails when $f\equiv1$, due to the anomalous solutions in \Cref{example:anomalous}.
	\end{remark}
	
	\begin{remark}[On the regularity of $f$]
		Close to regular points, the free boundary enjoys $C^{k+1,\alpha}$ estimates whenever $f\in C^{k,\alpha}$ (\cite{DSS15}), so one sees that the \Cref{theorem:viscosity bound} is essentially optimal also at the level of regularity of $f$, since $C^2$ estimates of the free boundary should not be expected if $\nabla f$ is not Dini continuous.
	\end{remark}
	We will derive \Cref{theorem:viscosity bound} from a semiconvexity estimate, let us outline the short heuristic that connects the two results. 
	Given a unit vector $e\in\mathbb S^{n-1}$ we write $\Delta_{e^\perp} u := \Delta u - D_{ee}u$. Then the mean curvature of the boundary of a super-level set at a nondegenerate point is
	\[
	-\divergence\left(\frac{\nabla u}{|\nabla u|}\right)=-\frac{\Delta_{\nu^\perp}u}{|\nabla u|},\qquad \nu:=\nabla u/|\nabla u|.
	\]
	Since close to the free boundary one can at best expect that $|\nabla u(x)|\gtrsim \delta(x),$ where 
	$$\delta(x):= \dist(x,\de\{u>0\}),$$
	an upper bound on the curvature asks for a $(n-1)$-semiconvexity estimate of the form
	\[
	\Delta_{e^\perp} u(x) \gtrsim -\delta(x),
	\]
	for all unit vectors $e$. This motivates the following bound:
	\begin{theorem}\label{theorem:linear decay semiconvexity}
		Let $u$ solve \eqref{eq:obstacle} in $B_1$ for some positive right hand side $f$ and assume $0\in\partial\{u>0\}$. Then
		\begin{equation}\label{eq:linrate}
			\Delta_{e^\perp}u(x) \ge -C\, \dist(x,\de\{u>0\}) \quad \text{ for all $x\in B_{1/2}$ and all $e\in\S^{n-1}$,}   
		\end{equation}
		for some $C = C(n,\min_{B_1} f,\alpha,\|f\|_{C^{1,\alpha}(B_1)})>0$.
	\end{theorem}
	Of course the linear rate cannot be improved, as one can see already at a regular free boundary point. The same \Cref{example:anomalous} shows that one cannot hope to bound the full negative part of $D^2 u$ with a linear (nor any power-type) rate in all dimensions $n\ge 3$. This is consistent with the best result available in this direction, which is Caffarelli's original semiconvexity estimate, see
	\cite[Theorem 1]{Caffarelli77}
	\begin{equation}\label{eq:caffarellilog}
		D^2u(x)\ge -C(n) |\log\delta(x)|^{-\alpha(n)}\id \quad \text{ in }B_{1/2}.
	\end{equation}
	
	In the case $n=2$, this bound was later sharpened by Rivier\'e and him \cite[Theorem 1]{CaffAnnals} to
	\[
	D^2u(x)\ge -C_1 \exp\big\{{-C_2 |\log\delta(x)|^{1/2}}\big\}\id \quad \text{ in } B_{1/2}
	\]
	and our \Cref{theorem:linear decay semiconvexity} finds the optimal semiconvexity rate (see also \Cref{theorem:2D 1})
	\[
	D^2u(x)\ge -C\, \delta(x)\id\quad \text{ in } B_{1/2}.
	\]
	
	We emphasize that our arguments draw ideas, build upon and combine nontrivially many recent approaches in the theories of the obstacle problem and thin obstacle problem.
	Specifically we use the frequency formula approach of \cite{FS19,FROS24}, the improvement of semiconvexity approach of \cite{SY23,SavinYufine}, the (log)-epiperimetric inequalities of \cite{Weiss,CSVObstacle}, the spectral gaps of \cite{FranceschiniSavin25} and the linearisation around integer frequency points of \cite{SavinYuinteger}.

	\subsection{About the proof of \Cref{theorem:linear decay semiconvexity}}
	
	To fix the ideas, the reader can, in a first read, consider the case $n=2$ and $f\equiv1$, as this already contains most of the ideas of the general case\footnote{\,The main difference between 2D and nD is that the epiperimetric inequality in 2D (\Cref{lem:apply projection epiperimetric}) is replaced by a log-epiperimetric in nD (\Cref{lem:log epi}).}.
	The starting point to establish a semiconvexity estimate is to quantify the rate of approximation of a solution $u$ of \eqref{eq:obstacle} by another \textit{convex} solution $v$. Indeed, under these assumptions,
	\[
	D^2u\ge -C\, \|u-v\|_{L^2(\de B_1)}\id\quad \text{in }B_{1/2}.
	\]
	Thus, our goal is to show that, when $0\in\de\{u>0\}$ and $u$ is not mean convex in $B_r$, then
	\begin{equation}\label{eq:cubicapprox}
		|u-v^{(r)}|\le C_0 r^3\quad\text{ in $B_r$, for all }r\in(0,1/2),
	\end{equation}
	for some universal $C_0$ and suitable explicit solutions $v^{(r)}$. The typical profiles used as $v^{(r)}$ would be, up to rotation, the homogeneous cones:
	\[
	\tfrac12 x\cdot Ax,\quad  A\in\mathrm{Sym}(n),\, A\ge0,\,\tr A=1,\qquad\tfrac12 (x_1)_+^2.
	\]
	The former are singular profiles, the latter is the regular one.
	
	When $u$ is well approximated by a regular cone, $\eps$-regularity is triggered and \eqref{eq:cubicapprox} follows. The difficulties arise when $u$ is well-approximated, at a certain scale $\rho$, by a singular cone $p$, and then transitions to a regular profile at some smaller scale $\theta\rho\ll\rho$ (not the other way around---because of Weiss monotonicity formula). In this case, $\eps$-regularity estimates degenerate and one loses the uniform constant in \eqref{eq:cubicapprox}. The situation bears analogies with \cite{MinimalSurfacesLinearization} in the framework of minimal surfaces or \cite{TwoPhaseBernoulli} in the Bernoulli problem.
	
	A key observation is that the Almgren's frequency $\phi$ of the difference $u-p$ (where $p$ is a singular blow-up profile) is effectively monotone in $(r,1)$ as long as
	\[
	\phi(r,u-p)\ge 2,
	\]
	and that the cubic approximation \eqref{eq:cubicapprox} holds, by classical formulas, as long as
	\[
	\phi(r,u-p)\ge 3.
	\]
	On the other hand, in \Cref{prop:low frequency}, we prove that, if $\phi$ is smaller than $3-\delta_\circ$, one can adopt similar ideas to \cite{SY23} to prove a linear decay of semiconvexity all the way down to the regular cone regime. 
	
	Then we are left with proving \eqref{eq:cubicapprox} in the transition range $[\theta r,r]$ where $\theta$ is such that
	\[
	3-\delta_\circ=\phi(\theta r,u-p)\le \phi(r,u-p)=3,
	\]
	for some fixed universal $\delta_\circ>0$. In this regime we prove \Cref{prop:extended monneau}, a highly precise refinement of Monneau monotonicity, which gives
	\begin{equation}
		\sup_{B_{\theta r}}|u-p| \le C_\circ \theta^3 \sup_{B_{ r}}|u-p|, \label{eq:point}
	\end{equation}
	for a universal $C_\circ$ and for all $\theta>0,$ no matter how small. An important ingredient to establish \eqref{eq:point} is a new epiperimetric inequality in the spirit of \cite{Weiss,CSVObstacle,SavinYuinteger}. 
	% (we work with negative energy gaps).
	The point of \eqref{eq:point} is that, while the smallness of $\theta$ itself cannot be quantified, it can weaken the overall cubic decay of $u-p$ from scale $r$ to scale $\theta r$ just by the finite constant $C_\circ$.
	
	\subsection{Structure of the paper}
	In \Cref{sec:preliminary} we gather our toolbox of formulas and bounds for general right-hand sides and dimension. In \Cref{sec:model case}, to better illustrate the ideas of the paper and the main contributions, we prove the main result in the model case of 2D solutions with constant right hand side: \Cref{theorem:2D 1}, which follows combining the two key ingredients \Cref{prop:low frequency} and \Cref{prop:extended monneau}.
	We then adapt the arguments to the general case in \Cref{sec:preliminary rhs,sec:theorem rhs}.
	
	\subsection{A note on the paper's length}
	The core argument for $f \equiv 1$ and arbitrary dimension $n$ is relatively short, and would make a 25 pages paper. The additional length is due to the technical work needed to handle general right hand sides $f \in C^{1,\alpha}$. We chose not to omit these adaptations because they involve delicate adjustments at some points, which are worth recording in full detail rather than leaving them to the reader. At the same time, presenting only the general case from the start would likely obscure the main ideas. This motivates our two-step presentation.

	\subsection{Use of AI}
	We wish to make explicit that the use of AI (ChatGPT-5.6 Sol and Gemini-3.1 Pro) in this paper is minor and secondary to the main novelties of the work.
	
	The core conceptual arguments were entirely human-generated. 
	
	In particular, our original proof for $n=2$ and $f\equiv 1$, as well as the realization that the same strategy worked for the \textit{mean} curvature in higher dimension, were in no way AI-helped.
	
	Besides language editing, AI was used for the following tasks:
	\begin{itemize}
		\item To speed up the computations of the truncated frequency formulas when $f\not \equiv 1$. These formulas are essentially known from \cite{FROS24}. Their adaptation to a smooth $f$ was carried out by the second author in \cite{CinftyRectifiable}, but with slightly different statements and assumptions.
		\item To find the barrier in the proof of \Cref{lem:caffarelli rhs holder}. We delegated this search to the AI, and it found a candidate function, which we then rigorously verified. This technical lemma is exclusively needed to generalize our proof to $f\in C^{1,\alpha}$ rather than $f\in C^{1,1}$, which is a secondary improvement relative to our main result.
		\item To speed up the iterative drafting of the log-epiperimetric inequality \Cref{lem:log epi}. This statement (needed only if $n\ge 3$) was known in the literature in slightly different energy regimes (\cite{cubicEpiperimetric,CSV20,SavinYuinteger}) and needed to be adapted to our setting. We claim more insight in the way we use it, rather than in the statement itself which is probably folklore among some experts.
	\end{itemize}
	
	% \subsection{Use of AI}
	% Beside minor editing tasks, the authors acknowledge some use of AI models exclusively in the following: to speed up computations for the proof of \Cref{lem:log epi}, to prove the technical \Cref{lem:caffarelli rhs holder} and for adapting classical monotonicity formulas to the case with right hand side.
	% The paper has entirely been written by the authors and AI computations have been re-done and written by hand. 
	
	\subsection{Acknowledgments} 
	The authors are thankful to Alessio Figalli and Joaquim Serra for suggesting this problem in 2D to us and for providing continuous encouragement.
	
	They are also thankful to Federico Glaudo and Xavier Fernand\'ez-Real for generously sharing their notes \cite{GFR22}, which contained an unpublished proof of the suboptimal estimate $-D^2u\le C_\alpha\dist^\alpha$ for $\alpha<1$ in 2D, and to Matteo Carducci, for helpful discussions on the correct statement in higher dimensions for \Cref{lem:apply projection epiperimetric}.
	
	During the preparation of this work Federico Franceschini has been partly supported by the ERC grant No. 948029 StableIF.

	\section{Notation and preliminary results}\label{sec:preliminary}
	\subsection{Notation and meaning of ``universal''}
	Throughout the paper $u$ will solve \eqref{eq:obstacle} with right-hand side $f\in C^{1,\alpha}(B_1)$ for some $\alpha>0$ and $f\ge \mu>0$.
	
	We will call \emph{universal} positive constants that depend only on upper bounds on the quantities
	\[
	n,\quad 1/\mu,\quad 1/\alpha, \quad \|f\|_{C^{1,\alpha}(B_1)},
	\]
	and we will denote them by $C$ when large and $c$ when possibly very small. If needed for later reference within a proof, they will be numbered as $C_1,$ $C_2$, $c_1$, etc.
	
	We introduce the notation:
	\begin{equation}\label{eq:notation Taylor}
		f(x)=f_1(x)+f_2(x),\quad f(0) + \nabla f(0)\cdot x=:f_1(x).
	\end{equation}
	So that for all $r\in(0,1)$ it holds 
	\[
	\|f_2(r\cdot)\|_{C^{1,\alpha}(B_1)}\le Cr^{1+\alpha}.
	\]
	with $C$ universal.

	Unless otherwise specified, given a function $w$ and $r>0$ we usually write
	\begin{equation}
		w_r :=w(r\cdot),\qquad \tilde w_r :=\frac{w_r}{\|w_r\|_{L^2(\de B_1)}}.\label{eq:notation tilde w}
	\end{equation}
	We also use Almgren's frequency function:
	\begin{equation}\label{eq:def frequency}
		\phi(r,w) := \frac{D(r,w)}{H(r,w)} = D(1,\tilde w_r),
	\end{equation}
	where
	\[
	H(r,w) := \int_{\de B_1} w_r^2,\quad D(r,w) := \int_{B_1} |\nabla w_r|^2,
	\]
	and Weiss' energy functionals 
	\[
	W_\lambda(r,w) := r^{-2\lambda}(D(r,w)-\lambda H(r,w)),\qquad \lambda\ge 0.
	\]

	Given $c\in\R$ we write
	\[
	D^2u(x)\ge c
	\]
	meaning that
	\[
	D_{ee}u(x)\ge c\quad \text{for all }e\in\mathbb S^{n-1}.
	\]
	For $e\in\mathbb S^{n-1}$ we write the $(n-1)$ laplacian in the plane $e^\perp$ as
	\[
	\Delta_{e^\perp} u(x) := \Delta u(x) - D_{ee}u(x).
	\]
	
	We denote the set of quadratic polynomials solving \eqref{eq:obstacle} with $f\equiv 1$ by
	\[
	\mathcal P_2:=\{p(x) = \tfrac12 x\cdot Ax,\,A\in\mathrm{Sym}(n),\,A\ge0,\,\tr A=1\}.
	\]
	With a slight abuse of notation, for $k\in\{0,\dots,n-1\}$ we write
	\begin{equation}\label{eq:stratification polynomials}
		p\in\Sigma_{k}\qquad\iff \qquad p\in\mathcal P_2\quad\text{and}\quad\dim\{p=0\}=k.
	\end{equation}

	\subsection{Some facts on the obstacle problem}
	We collect well known results and some new properties about solutions of \eqref{eq:obstacle}. As general references for the obstacle problem we refer to \cite{PSU2012,XaviXaviLocal}.
	
	\begin{lemma}\label{lem:comparison principle}
		Let $\Omega\subset\R^n$ be a bounded domain, let $u$ solve \eqref{eq:obstacle} in $\overline\Omega$. Suppose that $v,V\colon\overline{\Omega}\to\R$ satisfy
		\[\begin{cases}
			\Delta V \le f \ind_{\{V>0\}} &\text{in }\Omega,\\
			V\ge0                       &\text{in }\Omega,\\
			V\ge u                      &\text{on }\de\Omega,
		\end{cases}\qquad
		\begin{cases}
			\Delta v\ge f                &\text{in }\Omega,\\
			v\le u                      &\text{on }\de\Omega.
		\end{cases}\]
		Then $v\le u\le V$ in $\overline{\Omega}$.
	\end{lemma}
	
	\begin{lemma}\label{lem:local bounds rhs}
		If $u$ and $U$ solve \eqref{eq:obstacle} with right-hand sides  $f$ and $F$ respectively, then
		\[
		\|\lap (u-U)\|_{L^1(B_{1/2})}+\|u-U\|_{L^\infty(B_{1/2})}\le C(\|u-U\|_{L^2(\de B_1)} + \|f-F\|_{L^\infty(B_1)}),
		\]
		for some $C=C(n)>0$.
	\end{lemma}
	\begin{proof}
		One checks that
		\[
		\Delta|u-U|\ge - \|f-F\|_{L^\infty(B_1)}\quad\text{in }B_1,
		\]
		so the $L^\infty$ bound follows from the mean value formula.
		
		The bound on the laplacian follows using as comparisons $v^+$ (resp. $v^-$) solutions of \eqref{eq:obstacle} in $B_{2/3}$ with boundary data $\max\{u,U\}$ (resp. $\min\{u,U\}$) and with right hand side $\min\{f,F\}$ (resp. $\max\{f,F\}$). Indeed, with this choice
		\[
		v^-\le u,U\le v^+\quad\text{and}\quad\Delta(v^+-v^-) \ge -|f-F|\qquad\text{in }B_{2/3}.
		\]
		Choosing a nonnegative cutoff $\xi\in C^\infty_c(B_{2/3})$ with $\xi=1$ on $B_{1/2}$ we can bound
		\[\begin{split}
			\|\Delta (v^+-v^-) &+ |f-F|\|_{L^1(B_{1/2})}\le \int_{B_{2/3}} \xi \Delta(v^+-v_-)+\|f-F\|_{L^1(B_{2/3})}\\
			&=\int_{B_{2/3}}\Delta\xi(v^+-v_-)+\|f-F\|_{L^1(B_{2/3})}\\
			&\le C\|v^+-v_-\|_{L^\infty(B_{2/3})}+\|f-F\|_{L^1(B_{2/3})}\\
			&\le C(\|u-U\|_{L^\infty(\de B_{2/3})}+\|f-F\|_{L^1(B_{2/3})}),
		\end{split}\]
		so the bound for $\|\Delta(v^+-v^-)\|_{L^1}$ follows by the triangular inequality.
		On the other hand, using $f\ind_{\{u>0\}}\le f\ind_{\{v^+>0\}}$ and $F\ind_{\{U>0\}}\ge F\ind_{\{v^->0\}}$ we get
		\begin{align*}
			\lap (u-U)& = f\ind_{\{u>0\}}-F\ind_{\{U>0\}}\\
			&\le \min\{f,F\}\ind_{\{v^+>0\}}-\max\{f,F\}\ind_{\{v_>0\}} +2\|f-F\|_{L^\infty(B_1)}\\
			&=\lap(v^+-v^-) +2\|f-F\|_{L^\infty(B_1)},
		\end{align*}
		and we conclude by symmetry between $u$ and $U$.
	\end{proof}
	The following result summarizes some well-known properties of solutions to \eqref{eq:obstacle} around a free boundary point.
	\begin{lemma}
		If $u$ solves \eqref{eq:obstacle} in $B_1$ and $0\in\de \{u>0\},$ then
		\begin{equation}\label{eq:C11 regularity}
			cr^2 \le \sup_{B_r} u \le Cr^2\quad\text{for }r\in(0,1/2),\qquad \|u\|_{C^{1,1}(B_{1/2})} \le C,
		\end{equation}
		with $C,$ $c>0$ universal. Furthermore for a universal modulus of continuity $\omega$, for $r\in(0,1/2)$,
		\begin{equation}\label{eq:semiconvexity_weak}
			D^2 u \ge -\omega(r)\quad\text{in }B_{r}.
		\end{equation}
	\end{lemma}
	In particular, the family of solutions of \eqref{eq:obstacle} with $0\in\de\{u>0\}$ is sequentially closed with respect to the $C^1_\loc$ convergence.
	In addition, it follows that in the proof of \Cref{theorem:linear decay semiconvexity} we may assume
	\begin{equation}\label{eq:universal upper bound}
		0\le u\le C\quad\text{in }B_1
	\end{equation}
	for a universal constant $C$.
	
	Caffarelli's dichotomy (see also \cite{Blank} for a low regularity case) says that 
	\[
	\de\{u>0\} = \Sigma(u)\cup \reg(u)
	\]
	So, if $0\in\de\{u>0\},$ then either:
	\begin{itemize}
		\item[a)] $0\in\mathrm{Reg}(u)$, so there is a modulus of continuity $\omega$ (non-universal) such that
		\[
		u(x) = \tfrac{f(0)}{2}(x\cdot\nu)_+^2 + \omega(r)r^2\quad\text{in }B_r
		\]
		and $\de\{u>0\}$ is smooth (see \Cref{lem:bounds regular points}), or
		\item[b)] $0\in\Sigma(u)$, so the contact set $\{u=0\}$ has zero density at $0$ and there is a modulus of continuity (universal) such that
		\[
		u(x) = f(0)p_0(x) + \omega(r)r^2\quad\text{in }B_r,
		\]
		for some unique $p_0\in\mathcal P_2$. Similarly, we denote by $p_z$ the blow-up of $u$ at $z\in\Sigma(u)$.
	\end{itemize}
	We split singular points according to their dimension: for $0\le k\le n-1$ we denote
	\[
	\Sigma_k(u) := \{x\,:\, \dim\{p_x=0\} = k\},
	\]
	where $p_x$ is the blow-up of $u$ at $x$. Recalling the notation \eqref{eq:stratification polynomials} we have $p_x\in\Sigma_{k}$ if and only if $x\in\Sigma_{k}$.
	
	\subsection{Semiconvexity estimates}
	
	An important point of this work is the ``convexity inheritance'' mechanism (\cite{SY23}): if our solution $u$ is approximated in $L^\infty$ by another solution $U$, which is strictly convex in the $e$ direction, then $u$ must be convex in the $e$ direction as well.
	% The core of this work consists in approximating a solution $u$ of \eqref{eq:obstacle} with convex and strictly ``mean convex'' solutions as in \eqref{eq:cubicapprox}. In order to deduce \Cref{theorem:linear decay semiconvexity} we need to prove semiconvexity bounds and to inherit convexity from estimates on the distance to convex and strictly convex solutions. 
	
	We make this precise in the next Lemmas, which are adapted from similar estimates in \cite{FS19} and \cite{SY23}, in particular \Cref{lem:semiconvexity} should be compared with \cite[Equation 2.17]{FS19}, while \Cref{lem:inherit convexity} should be compared with \cite[Lemma 3.4]{SY23}. The key observation behind these results, proved in \cite{Caffarelli77}, is the following: if $e\in\mathbb S^{n-1}$ and $u$ solves \eqref{eq:obstacle} with $f\equiv1$ then $\Delta(D_{ee}u)_-\ge0$ in $B_1$.
	
	Recall the notation in \eqref{eq:notation Taylor}.
	\begin{lemma}[Convexity estimates]\label{lem:semiconvexity}
		Let $u$ and $U$ solve \eqref{eq:obstacle} for some $f$ and $F$ respectively, with $U_{ee}\ge-\gamma$ in $B_1$, for some $e\in\mathbb S^{n-1}$ and $\gamma\ge 0$. Then
		\[
		\inf_{B_{1/2}}D_{ee}u\ge -C(\|u-U\|_{L^2(\de B_1)}+ \|f-f_1\|_{C^{1/4}(B_1)}+\|f-F\|_{L^\infty(B_1)} +\gamma).
		\]
		The constant $C=C(n)>0$ is dimensional.
	\end{lemma}
	\begin{proof}
		Take a corrector $\Phi$ solving
		\[
		\Delta\Phi = f-f_1\quad\text{in }B_{7/8},\quad\Phi = 0\quad\text{in }\de B_{7/8}.
		\]
		There is a universal $C>0$ such that
		\[
		|\Phi_{ee}(x)| \le C \|f-f_1\|_{C^{1/4}(B_1)} =:\tfrac12 A,\quad \forall x\in B_{7/8}.
		\]
		The function
		\[
		w(x):=
		\begin{cases}
			(\Phi_{ee}(x)-A -u_{ee}(x))_+&\text{ in } \{u>0\}\\
			0&\text{ otherwise}
		\end{cases}
		\]
		is continuous and subharmonic in $B_{7/8}$. 
		
		Indeed in $\{u>0\}$ we have
		\[
		\Delta w(x) = \lap \Phi_{ee} -\lap u_{ee} =f_{ee}-f_{ee}=0,
		\]
		and by \cref{eq:semiconvexity_weak}:
		\[
		\limsup_{u(x)\downarrow 0} \Phi_{ee}(x)-A-u_{ee}(x)\le -\tfrac12  A<0.
		\]
		Then the Weak Harnack inequality and $L^1\to L^{1,\infty}$ elliptic estimates for the laplacian give:
		\begin{align*}
			\|w\|_{L^\infty(B_{1/2})}&\le C \|w\|_{L^{1,\infty}(B_{3/4})} \le C\|(-u_{ee})_+\|_{L^{1,\infty}(B_{3/4})}\\
			&\le C\|(U_{ee}+\gamma-u_{ee})_+\|_{L^{1,\infty}(B_{3/4})}\\
			&\le C\|(U_{ee}-u_{ee})_+\|_{L^{1,\infty}(B_{3/4})} +C\gamma\\
			&\le  C\|(u-U)_{ee}\|_{L^{1,\infty}(B_{3/4})} +C\gamma\\
			&\le  C\|\lap (u-U)\|_{L^{1}(B_{7/8})}+C\|u-U\|_{L^{1}(B_{7/8})} +C\gamma.
		\end{align*}
		Now, in $B_{1/2}$,
		\[
		\|w\|_{L^\infty(B_{1/2})}\ge -u_{ee}(x)-2A,
		\]
		and in \Cref{lem:local bounds rhs} we proved
		\[
		\|\Delta(u-U)\|_{L^1(B_{7/8})} \le C(\|u-U\|_{L^2(\de B_1)} + \|f-F\|_{L^\infty(B_1)}),
		\]
		so the conclusion follows.
	\end{proof}
	
	We get to the precise form of the mechanism described above.
	
	\begin{lemma}[Convexity Inheritance]
		\label{lem:inherit convexity}
		There are universal constants $\sigma_\circ,$ $\rho_1$ small and $A_\circ$ large with the following property. 
		
		If $u,U$ solve \eqref{eq:obstacle} with right hand sides $f,$ and $F$ respectively and satisfy, for some $\eps>0$ and some $e\in\mathbb S^{n-1}$:
		\begin{enumerate}[label=(\roman*)]
			\item $D_{ee}U\ge-\eps$ in $B_1$ and $D_{ee}U\ge A_\circ \eps$ in $\{\dist(\cdot,\{u=0\})\ge\sigma_\circ\}$;
			\item $|u-U|\le\eps$ in $B_1$;
			\item $\|f-f_1\|_{C^{1,\alpha}(B_1)}\le \eps$;
			\item $\|f-F\|_{C^{1/4}(B_1)}\le \eps$.
		\end{enumerate}
		Then $D_{ee}u\ge0$ in $B_{\rho_1}$.
	\end{lemma}
	\begin{remark}\label{oss:inherit mean convexity}
		\Cref{lem:inherit convexity} holds true, by replacing $D_{ee}U$ with $\Delta_{e^\perp}U$, where $e$ is a unit vector, in assumption (i). The conclusion in this case is $\Delta_{e^\perp} u\ge0$. The proof is the same, but one considers $h:=(u_{\xi_1\xi_1}+\ldots+u_{\xi_{n-1}\xi_{n-1}})/\eps$, where $\{\xi_1,\ldots,\xi_{n-1}\}$ is an orthonormal basis of $e^\perp$. The point is that also this $h$ is a supersolution.
	\end{remark}
	The proof uses the following result, which is classical for $F=0$ ({\cite[Lemma 11]{C98}}) and the original proof works for $F\in C^1$. We postpone to \Cref{app:caffarelli} the proof of the general $F\in C^\alpha$ case, which is more involved.
	
	\begin{lemma}\label{lem:caffarelli rhs holder}
		There are positive constants $A_1$ large and $\sigma_1,$ $\rho_1$ small with the following property.
		
		Let $u$ solve \eqref{eq:obstacle} in $B_1$ for some $f\ge\mu>0$.
		Let $h\colon \{u>0\}\cap B_1\to\R$ satisfy
		\[\begin{cases}
			h\ge0   &\text{on }\de\{u>0\},\\
			\Delta h\le \divergence F   &\text{in }\{u>0\},\\
			h\ge -1         &\text{in }\{u>0\},\\
			h\ge A_1        &\text{in }\{\dist(\cdot,\de\{u>0\})>\sigma_1\},
		\end{cases}\]
		for some $F\in C^\alpha(B_1;\R^n),$ with $\alpha>0$.
		Then
		\[
		h\ge 0\quad\text{ in }\{u>0\}\cap B_{\rho_1}.
		\]
		The constants $A_1$, $\sigma_1$ and $\rho_1$ depend on
		\begin{equation}\label{eq:dependencies_caffarelli}
			n,\quad\mu,\quad\alpha,\quad \|F\|_{C^\alpha(B_1)},\quad \|f\|_{C^\alpha(B_1)}.
		\end{equation}
	\end{lemma}
	\begin{proof}[Proof of \Cref{lem:inherit convexity} given \Cref{lem:caffarelli rhs holder}]
		We check that $h:=D_{ee}u/\eps$ satisfies the assumptions of \Cref{lem:caffarelli rhs holder}.
		
		By \eqref{eq:semiconvexity_weak}, we have  $(u_{ee})_-\to 0,$ as $x\to \de\{u>0\}$. Furthermore, in $\{u>0\}$,
		\[
		\lap h=\divergence V,\quad V:=\eps^{-1}(f_{e}-f_e(0))\, e.
		\]
		Notice that, by (iii),
		\[
		\|V\|_{C^\alpha(B_1)}\le\eps^{-1}\|\nabla f-\nabla f_1\|_{C^\alpha(B_1)}\le1.
		\]
		By the previous \Cref{lem:semiconvexity} with $\gamma=\eps$ we have
		\[
		\eps h\ge -C(\|u-U\|_{L^2(\de B_1)}+ \|f-f_1\|_{C^{1,\alpha}(B_1)}+\|f-F\|_{C^{1/4}(B_1)} +\eps)\ge -C\eps.
		\]
		As $D_{ee}U=0$ a.e. on $\{U=0\}$, assumption (i) implies $U>0$ a.e. on $\{\dist(\cdot,\{u=0\})\ge\sigma_\circ\}$. If $B_{\sigma}(z)\subset \{\dist(\cdot,\{u=0\})\ge\sigma_\circ\}$, we may use classical elliptic estimates on $u-U$ to find
		\[
		(u-U)_{ee}\ge -C(\sigma)(\|u-U\|_{L^2(\de B_1)}+ \|f-F\|_{C^{1/4}(B_1)})\ge -C(\sigma)\eps\quad\text{ in }B_{\sigma}(z),
		\]
		and thus
		\[
		u_{ee}(z)\ge -C(\sigma)\eps +U_{ee}(z)\ge -C(\sigma)\eps+A_\circ\eps.
		\]
		We choose $\sigma_\circ=\tfrac14\sigma_1$ and $A_\circ =2+C(\sigma_\circ)A_1$ where $\sigma_1,A_1$ are the constants of \Cref{lem:caffarelli rhs holder}. Applying this Lemma we find
		\[
		u_{ee}=h\ge0\text{ in }B_{\rho_1}.\qedhere
		\]
	\end{proof}
	
	\subsection{Regular points}
	Around regular points, the free boundary enjoys $C^{k+1,\alpha}$ estimates whenever $f\in C^{k,\alpha}$, see \cite[Theorem 6.17]{PSU2012}. In particular:
	\begin{lemma}\label{lem:bounds regular points}
		There are universal $C,$ and $\delta_0$ such that if
		\[
		|u-\tfrac{f(0)}2(x_n)_+^2|<\delta<\delta_0\quad\text{in }B_1
		\]
		then, writing $x=(x',x_n)$,
		\[
		\{u=0\}\cap B_{1/2} = \{x_n \le \Gamma(x')\}\cap B_{1/2}
		\]
		for some $\Gamma\colon B_{1/2}\cap\{x_n=0\}\to\R$ with 
		\[\|\Gamma\|_{C^{2,\alpha}}\le C,\quad |\Gamma(0)|\le C\sqrt{\delta}.\]
	\end{lemma}
	
	We will make use of the following consequence of regularity properties of the free boundary close to a half-plane solution.
	\begin{lemma}\label{lemma:regular case}
		There are $\delta_1,$  $C_1$ universal such that if $u$ solves \eqref{eq:obstacle} with $u(0)=0$ and $$|u-\tfrac{f(0)}{2}(x_n)_+^2|<\delta \text{ in }B_1\subset\R^n$$ for some $\delta<\delta_1$ then
		\[
		D^2u(x) \ge - C(\delta +\|f-f(0)\|_{C^{1/4}(B_1)})|x|\quad \text{ in }B_{1/4}.
		\]
	\end{lemma}
	\begin{proof}
		By \Cref{lem:bounds regular points}, if $\delta_1$ is universally small we can write $B_{1/2}\cap\{u=0\}= \{x_n\le \Gamma(x')\}$ where, up to a rotation,
		\[
		\Gamma(0)\ge|\nabla \Gamma (0)|=0,\qquad \|\Gamma\|_{C^{2,\alpha}}\le C_1.
		\]
		Notice that \Cref{lem:semiconvexity} applied with $U:=f(0)(x_n)^2_+/2$ and $\gamma=0$, gives
		\[
		-u_{ee}\le C(\delta +\|f-f(0)\|_{C^{1/4}(B_1)}) =:A_2 \text{ in }B_{1/2}.
		\]
		Let $h$ solve
		\[\begin{cases}
			\Delta h=-\frac{f_{ee}}{A_2}  &\text{ in }B_{1/2}\cap\{x_n>\Gamma(x')\},\\
			h=1         &\text{ in }\de B_{1/2}\cap\{x_n>\Gamma(x')\},\\
			h=0         &\text{ on }B_{1/2}\cap\{x_n=\Gamma(x')\},
		\end{cases}\]
		then elliptic regularity and the smoothness of $\Gamma$ yields $\|h\|_{C^{1,1/4}(B_{1/4})}\le C(1+\|f_e/A_2\|_{C^{1/4}})\le C$ so 
		\[
		h(x)\le C|x|\quad \text{in }B_{1/4},
		\]
		for $C$ universal.

		For all $t>0$ the continuous function:
		\[
		w^t(x):=\begin{cases}
			(A_2h(x)-u_{ee}(x)-t)_+&\text{ in } u>0,\\
			0&\text{ otherwise },
		\end{cases}
		\]
		is subharmonic and nonnegative in $B_{1/2}$; so, by the $C^2$ smoothness of $\Gamma$, we have for all $x\in B_{1/4}\cap\{u>0\}$
		\[
		\begin{split}
			-u_{ee}(x) \le w^t(x) + A_2 h(x)+t\le CA_2\dist(x,\Gamma)+t\le C A_2|x|+t.
		\end{split}
		\]
		We conclude letting $t\downarrow 0$.
	\end{proof}

	\subsection{The Thin Obstacle Problem}
	
	The Thin Obstacle Problem arises in the second-order analysis of solutions of the obstacle problem at the singular stratum $\Sigma_{n-1}$, see \cite{FS19}. We recall here some basic, but important, facts.
	
	A function $v$ solves the thin obstacle problem in $B_1\subset \R^n$ if
	\begin{equation}\label{eq:thin obstacle}\begin{cases}
			\lap v\le 0   &\text{in }B_1,\\
			\lap v=0      &\text{in }B_1\setminus\{x_n=0,v=0\},\\
			v\ge0           &\text{in }B_1\cap\{x_n=0\}.
	\end{cases}\end{equation}
	We write in polar coordinates in 2D
	\begin{equation}\label{eq:1.5 homogeneous solution}
		\psi_{3/2}(r,\theta) := r^{3/2}\cos(3\theta/2),\quad x_{n-1}=r\cos\theta,\quad x_{n}=r\sin\theta. 
	\end{equation}
	We will make use of the following properties of solutions of \eqref{eq:thin obstacle}. In particular, it will be crucial for our arguments to work in dimension $n\ge 3$ the recent observation of \cite{FranceschiniSavin25}, which excludes $\lambda$-homogeneous solutions of \eqref{eq:thin obstacle} for any $\lambda\in(2,3)$.
	\begin{proposition}\label{prop:thin obstacle}
		Let $v$ be a nontrivial solution of \eqref{eq:thin obstacle} with 
		\[
		v(0)=0\quad\text{ and }\quad \phi(v)<3.
		\] 
		Then $\phi(0^+,v)\in\{1,3/2,2\}$ and, as $x\to0,$ one of the following happens:
		\begin{itemize}
			\item[(a)] For some constants $a_\pm$, with $a_++a_-\le0$ and not both $0$ it holds
			\[v(x) = a_+(x_n)_+ + a_- (x_n)_- + O(|x|^2);\]
			\item[(b)] For some $a>0$ and $R\in O(n)$ satisfying $Re_n=e_n,$ it holds
			\[v(x) = a\psi_{3/2}(Rx) + O(|x|^{2}),\] 
			where $\psi_{3/2}$ is given by \eqref{eq:1.5 homogeneous solution};
			\item[(c)] For some symmetric matrix $A$ satisfying $A\not\equiv0,$ $\tr A=0$ and $e\cdot Ae\ge0$ for all $e\perp e_n,$ it holds
			\[
			v(x) = x\cdot Ax + o(|x|^2).
			\]
		\end{itemize}
	\end{proposition}
	\begin{proof}
		The frequency is monotone for solutions of \eqref{eq:thin obstacle}, so $\phi(0^+,v)$ is well defined and $\phi(0^+,v)<3$. Then by \cite{PSU2012,FranceschiniSavin25} we must have $\phi(0^+,v)\in\{1,3/2,2\}$. A blow-up analysis at these points allows to conclude. We refer to \cite{PSU2012,XaviSignorini} for more details.
	\end{proof}
	We also recall the comparison principle for solutions of \eqref{eq:thin obstacle}, see \cite{PSU2012}.
	\begin{lemma}\label{lem:comparison thin obstacle}
		Let $v,V$ solve \eqref{eq:thin obstacle} in $B_1$ with $v\le V$ on $\de B_1$. Then $v\le V$ on $\overline{B_1}$.
	\end{lemma}

	\subsection{Estimates at singular points in $\bm{\Sigma_{n-1}}$}
	We collect useful results at singular points adapted from \cite{FS19}. In order to deal with the right-hand side $f$, we need to use as blow-up models more precise functions compared to the family $\tfrac{f(0)}{2}(x\cdot \nu)^2$.
	
	Given $\nu\in\mathbb S^{n-1}$ set
	\begin{equation}\label{eq:def_pnu_RHS}
		p_{\nu}:= \tfrac12 {f(0)} (x\cdot\nu)^2(1 + (x\cdot v_\nu))
	\end{equation}
	where $v_\nu\in\R^n$ is uniquely determined by the requirement that
	\[
	\lap p_\nu = f(0) + \nabla f(0)\cdot x=f_1.
	\]
	A computation shows:
	\[
	v_\nu := \big(\nabla f(0) - \tfrac23 (\nabla f(0)\cdot\nu)\nu\big)/f(0).
	\]
	Explicitly, when $\nu=e_n$,
	\begin{equation}\label{eq:rhs-singular-ansatz}
		p_{e_n}(x):=\tfrac12
		{f(0)}x_n^2
		+\tfrac 12\sum_{i<n}\partial_i f(0)\,x_i x_n^2
		+\tfrac16\partial_n f(0)\,x_n^3 .
	\end{equation}
	We do not display the dependence of $p_\nu$ from $f$, since it will always be clear from the context.
	Note that, provided
	\begin{equation}\label{eq:up to rescaling}
		|\nabla f(0)|\le cf(0)
	\end{equation}
	for a universal constant $c>0$ (which holds up to a universal rescaling), we have also
	\[
	p_\nu(x)\ge0\text{ for all }|x|\le1.
	\]
	If $u$ solves \eqref{eq:obstacle}, a direct computation shows the crucial properties
	\begin{equation}\label{eq:pde u-p_nu}
		\Delta (u-p_\nu)_r = O(r^{3+\alpha}) - r^2 f_r\chi_{\{u_r=0\}}\le O(r^{3+\alpha}),\qquad (u-p_\nu)_r\Delta(u-p_\nu)_r \ge -O(r^{3+\alpha})|w_r|.
	\end{equation}
	We remark that $p_\nu$ fails to be convex by a quadratic reminder, indeed under assumption \eqref{eq:up to rescaling}
	\begin{equation}
		\label{eq:p_k almost convex}
		D_{ee}p_\nu(x) \ge \tfrac 12 f(0)(e\cdot\nu)^2-C(n)\frac{|\pi_{\nu^\perp}\nabla f(0)|^2}{f(0)}(x\cdot\nu)^2\quad \forall e\in\S^{n-1}.
	\end{equation}
	Note that
	\begin{equation}
		D_{ee}p \equiv0\quad\text{for any }e\perp\nu\qquad\text{and}\qquad |\nabla p(x)|\ge cf(0)|x\cdot\nu|\quad\text{in }B_1.\label{eq:p_k almost convex direction}
	\end{equation}
	
	We are ready to prove a useful Lipschitz bound:
	\begin{lemma}\label{lem:lip estimate rhs}
		If $u$ is a solution of \eqref{eq:obstacle} for some $f$ satisfying \eqref{eq:up to rescaling} and $\nu\in\S^{n-1}$, then
		\begin{equation}\label{eq:lip estimates}
			\|\nabla(u-p_\nu)_r\|_{L^\infty(B_1)}\le C(\|(u-p_\nu)_r\|_{L^2(\de B_2)}+r^{3+\alpha})
		\end{equation}
		and
		\begin{equation}
			\{u_r=0\}\cap B_1\subset\{|x\cdot\nu|\le Cr^{-2}( \|(u-p_\nu)_r\|_{L^2(\de B_2)}+r^{3+\alpha})\}.\label{eq:bound contact set}
		\end{equation}
		for some universal $C$ and all $r\in(0,\tfrac34)$.
	\end{lemma}
	\begin{remark}
		If $f$ is affine there is no $+r^{3+\alpha}$ error on the right-hand side of both estimates.
	\end{remark}
	\begin{proof}
		We may assume $\nu=e_n$. We use \Cref{lem:semiconvexity} on $r^{-2}u_r$ and $U:=r^{-2}p_r$ with $e:=e_i$ for $i<n$, so that $D_{ii}U=0$  because of \eqref{eq:p_k almost convex direction}. We find that in $B_{3/2}$
		\[
		\begin{split}
			D_{ii}(u-p)(r\cdot)&=D_{ii}u(r\cdot)\ge -C\big(r^{-2}\|(u-p)_r\|_{L^2(\partial B_{2})}+\|(f_2)_r\|_{C^{1/4}(B_{2})} \big)\\
			&\ge -Cr^{-2}\|(u-p)_r\|_{L^2(\partial B_2)}-Cr^{1+\alpha},
		\end{split}
		\]
		while the equation implies in the $e_n$  direction
		\[
		\begin{split}
			D_{nn}(u-p)(r\cdot)&= f_r\ind_{\{u_r>0\}}-(f_1)_r -\sum_{i<n}D_{ii}u(r\cdot) \\
			&\le -f_r\ind_{\{u_r=0\}}+\|(f_2)_r\|_{L^\infty(B_2)} -\sum_{i<n}D_{ii}u(r\cdot)\\
			&\le Cr^{-2}\|(u-p)_r\|_{L^2(\partial B_2)}+Cr^{1+\alpha},
		\end{split} 
		\]
		in $B_{3/2}$. Then Lipschitz bounds on convex functions imply
		\[
		r^{-2}\|\nabla(u-p)_r\|_{L^\infty(B_1)}\le C \big(r^{-2}\|(u-p)_r\|_{L^2(\de B_2)}+ r^{1+\alpha}\big).
		\]
		% We conclude as \Cref{lem:L2 Linfty} and a covering argument implies $\|u-\tfrac12x_n^2\|_{L^\infty(B_{3/2})}\le C\|u-\tfrac12x_n^2\|_{L^2(\de B_2)}$.
		As a consequence of this estimate we can trap the contact set around the zero set of $p$. Indeed since $\nabla u=0$ on $\{u=0\}$ we find that
		\[
		\{u_r=0\}\cap B_1 \ \subset\ \{r^{-1}|\nabla p(r\cdot)|<C\big(r^{-2}\|(u-p)_r\|_{L^2(\de B_2)}+ r^{1+\alpha}\big)\}.
		\]
		We conclude as $r^{-1}|\nabla p(r\cdot)|\ge \tfrac\mu8 |x_n|$ if $r$ is universally small (see \eqref{eq:p_k almost convex direction}).
	\end{proof}
	
	% \begin{lemma}\label{lem:bound contact set}
		%     There is $C(n)>0$ such that
		%     \[
		%         \{u=0\}\cap B_1\subset\{|x_n|\le C\|u-\tfrac12x_n^2\|_{L^2(\de B_2)}\}.
		%     \]
		% \end{lemma}
	% \begin{proof}
		%     \Cref{lem:lipschitz estimates} yields
		%     \[
		%         \|\nabla(u-\tfrac12x_n^2)\|_{L^\infty(B_1)}\le C\|u-\tfrac12x_n^2\|_{L^2(\de B_2)}.
		%     \]
		
		% \end{proof}

	\section{Proof in 2D with a Constant Right Hand Side}\label{sec:model case}
	
	\subsection{} To better illustrate the ideas, we consider in this section the simplified case where 
	\[
	n=2\quad \text{ and }\quad f\equiv1,
	\]
	that is we consider $u$ solving
	\begin{equation}\label{eq:classical obstacle}
		\Delta u = \ind_{\{u>0\}}\quad\text{and}\quad u\ge0\qquad\text{in }B_1\subset\R^2.
	\end{equation}
	In this case, \Cref{theorem:linear decay semiconvexity} becomes
	\begin{theorem}\label{theorem:2D 1}
		If $u$ solves \eqref{eq:classical obstacle}, then
		\[
		D^2u(x) \ge -C\, \dist(x,\de\{u>0\})\id \quad \text{ for all }x\in B_{1/2}.
		\]
		for a universal $C>0$.
	\end{theorem}
	
	% And we deduce the following
	% \begin{theorem}\label{theorem:2D 2}
		%     If $u$ solves \eqref{eq:obstacle} with $n=2$ and $f\equiv1$, then $\de\{u>0\}$ has curvature bounded above by a universal constant.
		% \end{theorem}
	Here ``universal'' means it does not depend on anything. The rest of this section is devoted to the proof of \Cref{theorem:2D 1}.
	
	\subsection{Monotonicity formulae}\label{sec:monotonicity}
	In this section we note that some monotonicity formulae proven in \cite{FS19} at singular points can be extended to all free boundary points ``up to certain scales''. We keep $n$ general since it it does not change anything.
	
	We start with the classical Weiss' energy \cite{Weiss}, setting
	\[
	W_{\mathrm{obst}}(r,u) := \frac1{r^{n+2}}\int_{B_r}(|\nabla u|^2 +2 u )-\frac1{r^{n+3}}\int_{\de B_r} u^2.
	\]
	There is a dimensional constant $E_\circ$ such that $ W_{\mathrm{obst}}(r,p) = E_\circ$ for all $r>0$ and $p\in\mathcal P_2$. In addition, if $u$ solves \eqref{eq:classical obstacle} then
	\begin{equation}\label{eq:weiss monotonicity}
		\frac{d}{dr} W_{\mathrm{obst}}(r,u) \ge0.
	\end{equation}
	Finally, if $u$ solves \eqref{eq:classical obstacle} and $p\in\mathcal P_2$, writing $w = u-p$ and recalling \eqref{eq:def frequency}, it is proven in \cite[Lemma 2.3]{FS19} that
	\begin{equation}\label{eq:weiss vs frequency}
		W_{\mathrm{obst}}(r,u) - E_\circ = \frac1{r^{n+3}}\|w\|^2_{L^2(\de B_r)}(\phi(r,w)-2).
	\end{equation}
	\begin{lemma}\label{lem:frequency formula}
		Suppose $u$ solve \eqref{eq:classical obstacle} in $B_1$ and $p\in\mathcal P_2$. Set $w=u-p$. Assume also that $\phi(w_s)\ge2$ for some $s\in(0,1)$. Then $\phi(w_r)\ge2$ for all $r\in(s,1)$ and
		\[
		\frac{d}{dr}\phi(w_r) \ge \frac{2}{r}\left(\int_{B_1}\tilde w_r\Delta (\tilde w_r)\right)^2\quad\forall r\in(s,1).
		\]
	\end{lemma}
	\begin{proof}
		The same proof of \cite[Proposition 3.4]{FROS24} applies here, as the only step where it is used that $0$ is singular is to imply that $\phi(w_r)\ge2$ for all $r\ge s$. In our case, this follows by assumption and \cref{eq:weiss monotonicity,eq:weiss vs frequency}.
	\end{proof}
	Given $\lambda\ge 0$ we define the height functionals
	\[
	H(w) := \int_{\de B_1}w^2,\quad H_\lambda(r,w) := H(r^{-\lambda}w_r),
	\]
	and the frequency gap
	\[
	E_\lambda(r,w) := \phi(r,w)-\lambda.
	\]
	\begin{lemma}\label{lem:monneau}
		Suppose that $w=u-p$ where $u$ solves \eqref{eq:classical obstacle} in $B_1$ and $p\in\mathcal P_2$. Then for any $r,\lambda>0$
		\begin{equation}\label{eq:derivative monneau}
			\frac{d}{dr} \log H_\lambda(r,w)= \frac2rE_\lambda(r,w) + \frac2r\int_{B_1}\tilde w_r\Delta(\tilde w_r)\ge \frac 2rE_\lambda(r,w).
		\end{equation}
		In particular, if $\phi(s,w)\ge\lambda\ge2$ for some $s\in(0,1),$ then $\log H_\lambda(r,w)$ is increasing for $r\ge s$.
	\end{lemma}
	\begin{proof}
		Identity \eqref{eq:derivative monneau} is proven for example in \cite[Lemma 2.6]{FS19}. If $\phi(s,u-p)\ge \lambda\ge 2$ for some $s\in(0,1)$ then $\phi(r,u-p)\ge\lambda$ for any $r\in(s,1)$, thanks to \Cref{lem:frequency formula}, so the right hand side in \eqref{eq:derivative monneau} is nonnegative for $r\in(s,1)$.
	\end{proof}
	We have also the following doubling property, which follows from an upper bound on the frequency.
	\begin{lemma}\label{lem:doubling}
		If $\phi(1,u-p)\ge2$ and $\phi(2,u-p)\le10,$ then
		\[
		\|u-p\|_{L^2(\de B_2)}\le C\|u-p\|_{L^2(\de B_1)},
		\]
		for a universal $C>0$.
	\end{lemma}
	\begin{proof}
		The same proof of \cite[Lemma 3.6]{FROS24} applies here, as it relies only on \cite[Equations (3.6) and (3.7)]{FROS24}, which hold also in our case, thanks to \Cref{lem:frequency formula,lem:monneau}.
	\end{proof}
	Now we compute the derivative of the frequency gap.
	\begin{lemma}\label{lem:derivative weiss}
		Given $w\in C^{1,1}$ we have
		\[
		\frac{d}{dr} E_\lambda(r,w) \ge\frac{n-2+2\lambda}{r}(E_\lambda(z_r)-E_\lambda(w_r))-E_\lambda(r,w)\frac{d}{dr} \log H_\lambda(r,w),
		\]
		where $z_r$ is the $\lambda$-homogeneous extension of $\tilde w_r|_{\de B_1}$.
		
		In particular, if $w = u-p$ where $u$ solves \eqref{eq:classical obstacle} and $p\in\mathcal P_2$ with $E_{\lambda}(r,w)\le0$ then
		\[
		\frac{d}{dr} E_\lambda(r,w) \ge\frac{n-2+2\lambda}{r}(E_\lambda(z_r)-E_\lambda(w_r))-\frac2r E_\lambda(r,w)^2.
		\]
	\end{lemma}
	\begin{proof}
		For $\lambda,r>0$ set
		\[
		W_\lambda(r,w) = E_\lambda(r,w) H_\lambda(r,w),
		\]
		so that
		\[
		\frac{d}{dr} E_\lambda(r,w) = \frac{1}{H_\lambda(r,w)}\frac{d}{dr}W_\lambda(r,w) - E_\lambda(r,w)\frac{d}{dr}\log H_\lambda(r,w).
		\]
		Recalling \Cref{lem:monneau}, the result follows if we prove
		\[
		\frac{d}{dr}W_\lambda(r,w) \ge \frac{n-2+2\lambda}{r}(E_\lambda(z_r)-E_\lambda(r,w))H_{\lambda}(r,w).
		\]
		But this is a consequence of \cite[Equation (2.1)]{CSV20}.
	\end{proof}
	
	We also include a frequency gap at singular points.
	Recall that $0\in\Sigma_{n-1}(u)$ if and only if $u(x) = \tfrac{1}{2}(x\cdot e )^2 + o(r^2)$ in $B_r$ for some $e\in\mathbb S^{n-1}$.
	\begin{lemma}\label{lem:frequency gap}
		Let $u$ be a solution of \eqref{eq:classical obstacle} in $B_1\subset\R^n$ with $0\in\Sigma_{n-1}(u)$ and blow-up $p^*$. Then either $u\equiv p^*$ or $\phi(r,u-p^*)\ge3$ for all $r\in(0,1)$.
	\end{lemma}
	\begin{proof}
		By \cite[Proposition 2.10(b)]{FS19} we know $\lambda_*:=\phi(0+,u-p^*)\ge2+\alpha_\circ$ for some $\alpha_\circ>0$, and there exists a nontrivial $\lambda_*$-homogeneous solution of \eqref{eq:thin obstacle}. By \cite[Theorem 2]{FranceschiniSavin25}, we must have in fact $\lambda_*\ge3$. Then \Cref{lem:frequency formula} implies the assertion.
	\end{proof}

	\subsection{Paraboloid Solution in 2D}\label{sec:paraboloid}
	In the proof of \Cref{lem:low frequency close poly}, in Case (b), we will make use of a special global solution of \eqref{eq:classical obstacle}. This is for brevity only, in the general case we bypass this (cf. \Cref{lem:low frequency close poly f}). 
	
	There exists a global solution $P$ of the obstacle problem in $\R^2$ such that
	\begin{equation}\label{eq:blow down paraboloids}
		R^{-2}P(R\cdot) = \tfrac12x_2^2 + R^{-1/2}\psi_{3/2}(x) + o(R^{-1/2}),\quad\{P=0\}=\{x_1\le -\bar ax_2^2\},
	\end{equation}
	where $\bar a>0,$ $R\to+\infty$ and $\psi_{3/2}$ is given by \eqref{eq:1.5 homogeneous solution} (see for example Propositions 2.11 and 7.3 in \cite{2dglobalsolutions}). 
	
	The contact set of this solution is a parabola with vertex in the origin, hence the name.
	
	We will use that $P$ is strictly convex in $\R^2$, more precisely:
	\begin{lemma}
		Given $\sigma>0,$ there are positive constants $R_\sigma$ large and $c_\sigma$ small such that
		\begin{equation}\label{eq:strict convexity paraboloids}
			D^2 (R^{-2}P(R\cdot)) > c_\sigma R^{-1/2}\quad \text{in }\{\dist(\cdot,\{x_2=0,x_1\le-\sigma/2\})>\sigma\}\cap B_1
		\end{equation}
		for $R>R_\sigma$.
	\end{lemma}
	\begin{proof}
		By \eqref{eq:blow down paraboloids}, away from $\{x_2=0,x_1\le0\}$ we have, for $R$ large,
		\[
		D^2 (R^{-2}P(R\cdot)) = D^2(\tfrac12x_2^2) + R^{-1/2}D^2\psi_{3/2}+o(R^{-1/2}).
		\]
		As $D_{11}\psi_{3/2}=\tfrac34 r^{-1/2} \cos(\theta/2)$ we have 
		\[
		D_{11}\psi_{3/2} \ge c_\sigma>0\quad  \text{ in } \{\dist(\cdot,\{x_2=0,x_1\le-\sigma/2\})>\sigma\}\cap B_1.
		\]
		Thus in the same region, for $R$ large, we have
		\[
		D^2(R^{-2} P(R\cdot))  = \begin{pmatrix}
			R^{-1/2}D_{11}\psi_{3/2}+o\big(R^{-1/2}\big)&O\big(R^{-1/2}\big)\\ O\big(R^{-1/2}\big) & 1+O\big(R^{-1/2}\big)
		\end{pmatrix}>c_\sigma'
		\,R^{-1/2}
		\]
		as we wanted.
	\end{proof}
	
	\subsection{Semiconvexity decay assuming frequency $\bm{<3}$}
	
	Recall that we write $p\in\Sigma_1$ if $p(x)=\tfrac12(x\cdot\nu)^2$ for some $\nu\in\mathbb S^1$.
	Given $u\colon \overline B_1\to\R$ we set
	\begin{equation}\label{eq:distance from Sigma1}
		\dist_{\Sigma_{1}}(u) = \inf_{p\in\Sigma_{1}}\|u-p\|_{L^2(\de B_1)}.
	\end{equation}
	
	\begin{proposition}\label{prop:low frequency}
		Given $\delta>0,$ there is $C_\delta>0$ large such that, if $u$ solves \eqref{eq:classical obstacle} in $B_1\subset\R^2$ with $0\in\de\{u>0\}$ and
		\[
		\phi(1,u-p)\le 3-\delta\quad\forall p\in\Sigma_1,
		\]
		then
		\[
		D^2u(x)\ge -C_\delta \dist_{\Sigma_1}(u)|x|\quad \text{in }B_{1/2}.
		\]
	\end{proposition}
	Here is a summary of the proof. When $u$ is a 2D solution, one of the following will happen: either $u$ is close to a regular point, in which case we conclude by $\eps$-regularity; or $u$ is close to a singular point in the smallest stratum (i.e., a strictly convex $p$), in which case $u$ inherits convexity; or to a singular point in the maximal stratum (close to $p\in\Sigma_1$). This last case is the most delicate, to understand it we employ ideas in the spirit of \cite{SY23} and crucially exploit the separation $\delta>0$ from the frequency level $3$.
	
	It may be helpful to the reader to start from the proof of \Cref{prop:low frequency} giving momentarily for granted \Cref{lem:low frequency close poly}.
	
	\begin{lemma}\label{lem:low frequency close poly}
		Given $\delta>0$ there are $\eps_\delta$ and $C_\delta'$ such that if $u$ solves \eqref{eq:classical obstacle} in $B_1\subset\R^2$ with $0\in\de\{u>0\}$ and
		\begin{enumerate}[label=\roman*)]
			\item $\dist_{\Sigma_1}(u)<\eps_\delta$;
			\item $\phi(1,u-p)\le 3-\delta$ for all $p\in\Sigma_1$.
		\end{enumerate}
		Then
		\[
		D^2 u(x)\ge -C'_\delta\dist_{\Sigma_1}(u)|x|\quad \text{in }B_{1/2}.
		\]
	\end{lemma}
	\begin{proof}
		Assume by contradiction that there exist $u_k$ solving \eqref{eq:classical obstacle} with $0\in\de\{u_k>0\}$ and $\eps_k:=\dist_{\Sigma_1}(u_k)\to0$ but
		\begin{equation}\label{eq:contradiction low frequency close poly}
			D_{e_ke_k}u_k(x_k)<-k\eps_k|x_k|
		\end{equation}
		for some $x_k\in B_{1/2}$ and $e_k\in\mathbb S^1$.
		Notice that \Cref{lem:semiconvexity} implies
		\[
		k \eps_k |x_k|<-D_{e_ke_k}u_k(x_k)\le C\eps_k,
		\]
		and in particular 
		\[
		|x_k|\to 0\text{ as } k\to +\infty.
		\]
		Up to a rotation we can assume that
		\[
		\eps_k=\dist_{\Sigma_1}(u_k) = \|u_k-\tfrac12x_2^2\|_{L^2(\de B_1)}.
		\]
		Setting
		\[
		v_k := (u_k-\tfrac12x_2^2),\qquad \tilde v_k:=v_k/{\eps_k},
		\]
		by assumption we have
		\[
		\phi(1,v_k)\le 3-\delta,
		\]
		which implies 
		\[
		\|\tilde v_k\|_{H^1(B_1)}\le C\quad\text{ and }\quad \|\tilde v_k\|_{L^2(\de B_1)}=1.
		\] 
		Combining \Cref{lem:lip estimate rhs} with a covering argument gives
		\[
		\tilde v_k\to v\quad\text{weakly in }H^1(B_1)\text{ and in }C^0_\loc(B_1).
		\]
		We now claim that $v$ solves \eqref{eq:thin obstacle} with
		\begin{equation}\label{eq:properties limit solution}
			\phi(1,v)\le 3-\delta,\quad v(0)=0\quad\text{and}\quad \|v\|_{L^2(\de B_1)}=1.
		\end{equation}
		Indeed, \eqref{eq:properties limit solution} holds for $v$ since it holds for $v_k$ and it is stable under $C^0_\loc$ and weak $H^1$ convergence. We now prove that it solves \eqref{eq:thin obstacle}. Note that
		\begin{equation}\label{eq:delta of u - p}
			\lap \tilde v_k = -\tfrac1{\eps_k}\ind_{\{u_k=0\}} \le0,
		\end{equation}
		so $v$ is superharmonic.
		Moreover $\tilde v_k = u_k/\eps_k\ge0$ on $\{x_2=0\}$, so also
		\[
		v\ge0\quad \text{on }\{x_2=0\}.
		\]
		In addition, as $u_k\to \tfrac12x_2^2$, given any $\delta>0$ we have $u_k>0$ on $\{|x_2|>\delta\}$ for $k$ large, so
		\begin{equation}\label{eq:consequence of convergence to thin space}
			\Delta v =0\quad\text{on }\{x_2\neq0\}.
		\end{equation}
		Finally, if $v(x_1,0)>0$ we also have $v>0$ in $B_{\bar\delta}(x_1,0)$ for some $\bar{\delta}>0$, so we deduce by uniform convergence that $u_k(x)> 0$ in $B_{\bar\delta}(x_1,0)$ for $k$ large enough. It follows from this and \eqref{eq:delta of u - p} that $\Delta v_k=0$ in $B_{\bar\delta}(x_1,0)$. Together with \eqref{eq:consequence of convergence to thin space} this yields
		\[
		\Delta v = 0\quad\text{in }B_1\setminus\{v=0,x_2=0\}.
		\]
		So $v$ solves the thin obstacle problem. According to \Cref{prop:thin obstacle}, there are 3 possible cases. We show that each one leads to a contradiction with \eqref{eq:contradiction low frequency close poly}, following the ideas of \cite[Lemma 4.1, Step 1]{SY23}.

		More precisely, we show that in each case there is a small scale $\bar r>0$ ---depending only on $v$---such that, for $k$ large enough the rescaled solution 
		\[
		\bar u_k:=\bar r^{-2}  u_k(\bar r \cdot) 
		\]
		in fact satisfies the linear decay estimate
		\[
		D^2\bar u_k(y) \ge -C\bar r^{-2}\eps_k |y|\quad\text{ for all } y\in  B_{1/2}.
		\]
		This will give the sought contradiction: since $x_k\to 0$ we can eventually take $y:=x_k/\bar r$ and find:
		\[
		D^2\tilde u_k(z_k) \ge -C \bar r^{-3}\eps_k |z_k|,
		\]
		and thus comparing it with \eqref{eq:contradiction low frequency close poly}, get  
		\[
		-k \eps_k |z_k| >-C \bar r^{-3}\eps_k |z_k|. 
		\]
		which is impossible for $k$ large enough.
		
		Heuristically, the scale $\bar r$ is chosen so that the function $v(x)$ agrees with its blow-up at $x=0$.
		
		% We also set
		% \[
		% \|\tilde v_k - v\|_{L^\infty(B_{1/2})}=:\eta_k\to 0.
		% \]
		
		\medskip\noindent\textbf{Case (a) of \Cref{prop:thin obstacle}.}  In this case there is some constant $\bar K\ge1$ such that 
		\[
		| v(x) -\big( a_+(x_2)_+ + a_- (x_2)_-\big)|\le \tfrac18\bar K |x|^2,
		\]
		where $a_++a_-\le0$ and not both $a_+,a_-$ are 0, see \Cref{figure:case a}.
		\begin{figure}
			\centering
			\begin{subfigure}{0.48\textwidth}
				\centering
				\includegraphics[page=4,width=\textwidth]{images.pdf}
			\end{subfigure}
			\hfill
			\begin{subfigure}{0.48\textwidth}
				\centering
				\includegraphics[page=1,width=\textwidth]{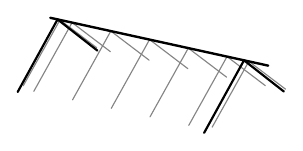}
			\end{subfigure}
			\caption{The approximating solutions $\bar r^{-2} u_k(\bar r\cdot)$ and the limit solution $v(x)$, in case (a).}
			\label{figure:case a}
		\end{figure}
		We show that there is $\bar r$ independent from $k$ and some $\beta_k\to0$ such that, for $k$ large,
		\begin{equation}\label{eq:contact set splits}
			B_{\bar r}\cap\{x_2=\beta_k\}\subset \{u_k=0\}.
		\end{equation}
		This concludes because we can split the free boundary into two regular regions in $B_{\overline r}$. Rigorously, setting
		\[
		u_k^\pm = \bar r^{-2}u_k(\bar r\cdot)\chi_{\{\pm(\bar rx_2-\beta_k)>0\}}  
		\]
		\eqref{eq:contact set splits} implies that the functions $u_k^\pm$ solve the obstacle problem in $B_1$ with $u_k^\pm(0)=0$ and
		\[
		|u_k^\pm-\tfrac12(x_2)^2_\pm|< \bar r^{-2}\eps_k\quad \text{in }B_1.
		\]
		So if $k$ is large enough we can apply \Cref{lemma:regular case} to $u_k^+,u_k^-$, and find
		\[
		- D^2 u_k^\pm(z) \le C \bar  r^{-2}\eps_k|z|,\quad \forall z\in B_{\overline{r}/2}.
		\]
		Since $x_k\in B_{\overline r /2}$ for $k$ large we get a contradiction with \eqref{eq:contradiction low frequency close poly}:
		\[
		k\eps_k |x_k|\le- D^2 u_k^\pm(x_k) \le C \bar  r^{-2}\eps_k|x_k|.
		\]
		
		Now we prove our claim \eqref{eq:contact set splits} with some barriers.
		
		First we prove that both
		\[
		a_+\le 0\quad\text{ and } \quad a_-\le 0.
		\]
		To see this it is enough to prove that $a_+>0$ implies $u_k(0)>0$ for $k$ large (the argument for $a_-$ is identical). Suppose that $a_+>0$ and choose $r=r(a_+,\bar K)$ so small that that the barrier
		\[
		U_k(x) := \tfrac12 {(1+\bar K \eps_k)}(x_2+\eps_k^2)^2-\tfrac12\bar K\eps_kx_1^2,
		\]
		satisfies, for $k$ large enough,
		\[
		{\eps_k^{-1}}(U_k-\tfrac12x_2^2)\le \tfrac{1}{2}\bar K(x_2^2-x_1^2)+C\eps_k\le \tfrac12  a_+(x_2)_+ -\bar Kr^2<v\quad\text{on }\de B_r\cap\{x_2\ge-\eps_k^2\}.
		\] 
		% Indeed with these choices, if $x_2=r$ we have
		% \begin{align*}
			% {\eps_k^{-1}}(U_k-\tfrac12x_2^2)\le\tfrac12 \bar K x_2^2 +C\eps_k< \tfrac12 r^2 +C\eps_k \le (a_+ -o(1))r\le v(x),
			% \end{align*}
		% and if $|x_1|=r,$ $-\eps_k^2\le x_2\le r$ we have
		% \begin{align*}
			% {\eps_k^{-1}}(U_k-\tfrac12x_2^2)&\le\tfrac12 x_2^2-\tfrac12 x_1^2 +C\eps_k\\
			% &\le \tfrac12 x_2^2 -\tfrac14 r^2 \le (a_+ -o(1)) (x_2)_+ -O(r^2) < v(x).
			% \end{align*}
		As $v_k\to v$ uniformly in $B_{1/2}$, this implies $U_k\le u_k$ on $\de B_r\cap\{x_2\ge-\eps_k^2\}$ for all $k$ large enough.
		On the other hand, $U_k\le0\le u_k$ on $\{x_2=-\eps_k^2\}$. Thus, for $k$ large enough
		\[
		U_k\le u_k\quad\text{on }\de(B_r\cap\{x_2\ge-\eps_k^2\}).
		\]
		As $\Delta U_k=1$, \Cref{lem:comparison principle} implies $U_k\le u_k$ on $B_r\cap\{x_2\ge-\eps_k^2\}$, and in particular $0<U_k(0)\le u_k(0)$, as claimed.
		
		So to prove \eqref{eq:contact set splits} we can now assume $a_+<0$, as not both $a_+,a_-$ vanish. We show that \eqref{eq:contact set splits} holds with $\beta_k := a_+\eps_k/2$, $k$ large enough and $\bar r=\bar r(a_+,\bar K)$ small enough. For each $\xi\in(-\bar r,\bar r)$ set
		\[
		U_{k,\xi}(x) = \tfrac12{(1-\bar K\eps_k)}(x_2+\beta_k)^2+ \tfrac12\bar K{\eps_k}(x_1-\xi)^2,
		\]
		so that on $\de B_{\bar r}((\xi,0))$ we have
		\[\begin{split}
			{\eps_k^{-1}}(U_{k,\xi}-\tfrac12x_2^2) &= \tfrac12a_+x_2 - \tfrac12\bar Kx_2^2 + \tfrac12\bar K(x_1-\xi)^2-C\eps_k\ge \tfrac12a_+(x_2)_+ -\tfrac12\bar K r^2 -C\eps_k>v(x).
		\end{split}\]
		As $\tilde v_k\to v$ uniformly in $B_{1/2}$, we have $u_k\le U_{k,\xi}$ in $\de B_{\bar r}$ for $k$ large, independently from $\xi$.
		As $U_{k,\xi}$ solve the obstacle problem as well, \Cref{lem:comparison principle} yields $u_k \le U_{k,\xi}$ in $B_{\bar r}$. Since $U_{k,\xi}(\xi,-a_+\eps_k/2)=0$, \eqref{eq:contact set splits} follows.
		
		\medskip\noindent\textbf{Case (b) of \Cref{prop:thin obstacle}.} In this case, up to a reflection we have
		\[
		v(x) = a\, \psi_{3/2}(x) + o(|x|^{3/2}),
		\]
		where $a>0$ and $\psi_{3/2}$ is given by \eqref{eq:1.5 homogeneous solution}, see \Cref{figure:case b}.
		\begin{figure}
			\centering
			\begin{subfigure}{0.48\textwidth}
				\centering
				\includegraphics[page=5,width=\textwidth]{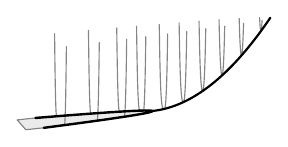}
			\end{subfigure}
			\hfill
			\begin{subfigure}{0.48\textwidth}
				\centering
				\includegraphics[page=2,width=\textwidth]{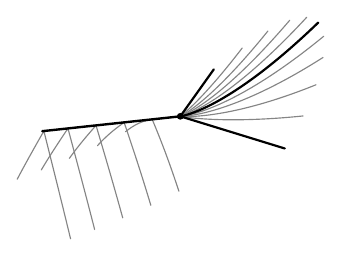}
			\end{subfigure}
			
			\caption{The approximating solutions $\bar r^{-2} u_k(\bar r\cdot)$ and the limit solution $v(x)$, in case (b).}
			\label{figure:case b}
		\end{figure}
		We first show that given $\sigma>0$ small there is $r_\sigma$ such that
		\begin{equation}\label{eq:contact set close paraboloids}
			\{x_2=0,-1+\tfrac12\sigma\le x_1\le -\tfrac12\sigma\}\subset\{r^{-2}u_k(r\cdot)=0\}\quad \text{for }r<r_\sigma
		\end{equation}
		and $k$ large enough. Indeed, given $\sigma>0$, one sees using \eqref{eq:1.5 homogeneous solution} that 
		$$\psi_{3/2}(x)\le -\lambda_\sigma|x_2|\quad \text{ in }\{-1<x_1\le -\tfrac12 \sigma,|x_2|<\sigma\}\text{ for some }\lambda_\sigma>0.$$ 
		Then take any $\xi\in(-1+\tfrac12\sigma,-\tfrac12\sigma)$ and choose some $\eta<a\lambda_\sigma/8$ and $r(\sigma,\eta)$ small so that for $r<r(\sigma,\eta)$ and $k\ge k(\sigma,\eta)$ large we have
		\[
		r^{-3/2}v(rx) \le -a\lambda_\sigma|x_2|+\tfrac18 \eta^2< (x_1-\xi)^2 -x_2^2 -\tfrac18\eta^2 \quad \text{ on } \de B_\eta(\xi).
		\]
		Thus, by uniform convergence
		\[
		r^{-2} u_k(rx)=\tfrac12x_n^2+\eps_kr^{-1/2}\le  v_k(rx)\le \tfrac12 x_n^2 +\eps_k r^{-1/2} [(x_1-\xi)^2 -x_2^2 ]:=U_{k,\xi}\quad \text{ on } \de B_\eta(\xi).
		\] 
		As $ U_{k,\xi}=1$ solves \eqref{eq:classical obstacle} as well, \Cref{lem:comparison principle} yields $r^{-2}u_k(r\cdot)\le U_{k,\xi}$ in $B_\eta(\xi,0)$. As $\xi$ was arbitrary, \eqref{eq:contact set close paraboloids} follows.
		
		We turn to the proof that $D^2u_k \ge0$ in $B_{\bar r}$ when $\bar r$ is small enough.
		
		Denoting by $P$ the global solution satisfying \eqref{eq:blow down paraboloids}, note that
		\[
		\tfrac{1}{\eps_k}\left(u_k-(a\eps_k)^4P((a\eps_k)^{-2}\cdot)\right)\to o(|x|^{3/2})\quad\text{in }C^0_\loc(B_1).
		\]
		Given $\sigma_\circ,A_\circ$ from \Cref{lem:inherit convexity}, let $c_{\sigma_\circ}>0$ be given by \eqref{eq:strict convexity paraboloids}.
		We can find $\bar r<r_{\sigma_\circ}$ small such that $o(\bar r^{3/2})<c_{\sigma_\circ}a\bar r^{3/2}/A_\circ$. In this way, setting $P_k:=(a\eps_k)^4P((a\eps_k)^{-2}\cdot)$ we have
		\[
		|\bar r^{-2}u_k(\bar r\cdot)-\bar r^{-2} P_k(\bar r\cdot)|<c_{\sigma_\circ}a\eps_k\bar r^{-1/2}/A_\circ\quad\text{in }B_1,
		\]
		while, for $k$ large enough, \eqref{eq:strict convexity paraboloids} and \eqref{eq:contact set close paraboloids} yield
		\[
		D^2(\bar r^{-2}P_k(\bar r\cdot))>c_{\sigma_\circ}a\eps_k\bar r^{-1/2}\quad\text{in }\{\dist(\cdot,\{u_k=0\})>\sigma_\circ\}\cap B_1.
		\]
		As $P$ is convex, \Cref{lem:inherit convexity} implies $D^2u_k\ge0$ in $B_{\rho_1 \bar r}$, contradicting \eqref{eq:contradiction low frequency close poly}.
		
		\medskip\noindent\textbf{Case (c) of \Cref{prop:thin obstacle}.} In this case for suitable $\alpha\ge 0$, $\beta\in \R$ (not both zero) it holds
		\[
		v(x) = \alpha x_1^2 -\alpha x_2^2 +\beta x_1 x_2 +o(|x|^2),
		\]
		as shown in \Cref{figure:case c}.
		\begin{figure}
			\centering
			\begin{subfigure}{0.48\textwidth}
				\centering
				\includegraphics[page=6,width=\textwidth]{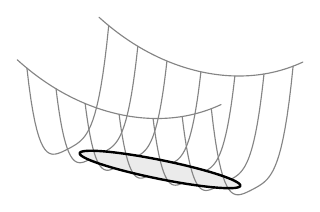}
			\end{subfigure}
			\hfill
			\begin{subfigure}{0.48\textwidth}
				\centering
				\includegraphics[page=3,width=\textwidth]{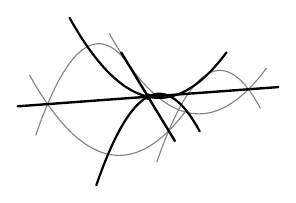}
			\end{subfigure}
			\caption{The approximating solutions $\bar r^{-2} u_k(\bar r\cdot)$ and the limit solution $v(x)$, in case (c).}
			\label{figure:case c}
		\end{figure}
		We first take $\beta=0$ performing a small rotation. Setting $\nu_k := (e_2+\beta \eps_ke_1)/|e_2+\beta \eps_ke_1|$ yields
		\[
		\tfrac1{\eps_k}(\tfrac12(x\cdot\nu_k)^2-\tfrac12x_2^2) = \tfrac12[(e_2+\nu_k)\cdot x][\tfrac1{\eps_k}(\nu_k-e_2)\cdot x)]\to \beta x_1x_2,
		\]
		thus there are rotation matrices $S_k\to \mathrm{Id}$ such that
		\[
		\hat v_k:=\tfrac1{\eps_k}(u_k\circ S_k-\tfrac12x_2^2)\to \hat v=\alpha x_1^2-\alpha x_2^2 +o(|x|^2).
		\]
		Now we prove that we must have $\alpha>0$.
		As $\phi(u_k-p,1)<3-\delta$ for all $p\in\Sigma_1$, we deduce that $\phi(\tilde v_k,1)\le 3-\delta$. By minimality of $\tfrac12x_2^2$ and the triangular inequality we have 
		$$1\le \|\hat v_k\|_{L^2(\de B_1)}\le \|\tilde v_k\|_{L^2(\de B_1)} +\|\tfrac1{\eps_k}(\tfrac12(x\cdot\nu_k)^2-\tfrac12x_2^2)\|_{L^2(\de B_1)} \le C,$$ 
		so by compactness of the trace we deduce $\hat v\not\equiv 0$ and $\phi(\hat v,1)\le 3-\delta<3$. As $\hat v$ solves the thin obstacle as well, if we had $\alpha=0$ then \Cref{prop:thin obstacle} would imply $\phi(\hat v ,1)\ge 3$, a contradiction.
		
		It is exactly at this point that we use our assumption that the frequency is blocked away from $3$.
		
		Let us now set
		\[
		p_k := (\tfrac12-\eps_k\alpha)x_2^2 + \eps_k\alpha x_1^2,\quad \tilde p_k := p_k\circ S_k^{-1},\quad 
		\]
		so that
		\[
		\tfrac1{\eps_k}(u_k-\tilde p_k)\to o(|x|^2)\quad \text{in }C^0_\loc(B_1).
		\]
		Given $A_\circ$ by \Cref{lem:inherit convexity} choose $\bar r$ such that $o(\bar r^2)<\alpha\bar r^2/A_\circ$. In this way in $B_1$ we have
		\[
		|\bar r^{-2}u_k(\bar r\cdot)-\tilde p_k| < \alpha\eps_k/A_\circ\quad\text{and}\quad D^2\tilde p_k>\alpha\eps_k,
		\]
		for $k$ large enough. So \Cref{lem:inherit convexity} yields $D^2 u_k\ge0$ in $B_{ \rho_1 \bar r}$, contradicting \eqref{eq:contradiction low frequency close poly}.
	\end{proof}
	
	\begin{proof}[Proof of \Cref{prop:low frequency}]
		Suppose by contradiction that there exist $u_k$ solving \eqref{eq:classical obstacle} with $0\in\de\{u_k>0\}$ and
		\begin{equation}\label{eq:frequency cap}
			\phi(1,u_k-p)\le 3-\delta\quad \forall p\in\Sigma_1
		\end{equation}
		but
		\[
		D_{e_ke_k}u_k(x_k)<-k\dist_{\Sigma_1}(u_k)|x_k|
		\]
		for some $x_k\in B_{1/2}$ and $e_k\in\mathbb S^1$. It follows from \Cref{lem:low frequency close poly} that
		\begin{equation}\label{eq:proof low frequency bound distance}
			\dist_{\Sigma_1}(u_k)\ge\eps_\delta\quad \text{for }k>C_\delta',
		\end{equation}
		so in particular, for $k$ large,
		\begin{equation}\label{eq:contradiction prop low frequency}
			D_{e_ke_k}u_k(x_k)<-k\eps_\delta |x_k|,\quad x_k\to0.
		\end{equation}
		By \eqref{eq:C11 regularity}, up to a subsequence we have
		\[
		u_k\to u_\infty\quad \text{in }C^1(B_{2/3}),
		\]
		where $u_\infty$ solves \eqref{eq:classical obstacle} in $B_{2/3}$ with $0\in\de \{u_\infty>0\}$.
		We split in two cases, according to Caffarelli's dichotomy. If $0\in\mathrm{Reg}(u_\infty)$ there is $\bar r$ small so that, up to a rotation,
		\[
		\left|\bar r^{-2}u_\infty(\bar r\cdot)-\tfrac12(x_2)_+^2\right|< \delta_1\quad \text{in }B_1,
		\]
		with $\delta_1$ given by \Cref{lemma:regular case}. Since the same holds for $\bar r^{-2}u_k(\bar r\cdot)$, \Cref{lemma:regular case} gives a contradiction with \eqref{eq:contradiction prop low frequency} for $k$ large enough.
		
		If $0\in\Sigma(u_\infty)$, denote by $p^*$ the blow-up of $u_\infty$ at $0$. We first note that $0\in\Sigma_0(u_\infty)$. Indeed, if $0\in\Sigma_1(u_\infty)$, \eqref{eq:proof low frequency bound distance} implies $u_\infty\not\equiv p^*$, so \Cref{lem:frequency gap} gives $\phi(1,u_\infty-p^*)\ge 3$, contradicting \eqref{eq:frequency cap}. As $p^*\in\Sigma_0$, in particular there is $\bar\eta>0$ such that
		\[
		D^2p^*\ge \bar\eta >0.
		\]
		Given $A_\circ>0$ from \Cref{lem:inherit convexity} choose $\bar r$ such that
		\[
		|\bar r^{-2}u_\infty(\bar r\cdot)-p^*|<\frac{\bar\eta}{A_\circ}\quad \text{in }B_1.
		\]
		By uniform convergence, the same holds for $\bar r^{-2}u_k(\bar r\cdot)$ with $k$ sufficiently large. So \Cref{lem:inherit convexity} yields $D^2 u_k\ge 0$ on $B_{\rho_1 \bar r}$, contradicting \eqref{eq:contradiction prop low frequency} for $k$ large enough.
	\end{proof}

	\subsection{Uniform error integrability in Monneau monotonicity}
	In this subsection we prove the following crucial proposition. We recall that we are implicitly assuming the universal bound \eqref{eq:universal upper bound}.
	\begin{proposition}\label{prop:extended monneau}
		There are universal $\delta_\circ,C_\circ>0$ such that if $u$ solves \eqref{eq:classical obstacle} in $B_2\subset \R^2$ with $0\in\partial\{u>0\}$ and for some $\rho\in(0,1)$ 
		\[
		\phi(\rho,u-\tfrac12x_2^2)\ge 3-\delta_\circ,
		\]
		then
		\begin{equation}\label{eq:nice}
			\|(u-\tfrac12 x_2^2)(\rho\cdot)\|_{L^2(\de B_1)}\le C_\circ \rho^3.
		\end{equation}
	\end{proposition}
	Note that, writing $w = u-\tfrac12x_2^2$ and $E_3(r,w) = \phi(r,w)-3$, \Cref{lem:monneau} gives
	\begin{equation}\label{eq:explainer}
		\frac{d\log H_3(r,w)}{dr} \ge \frac{2}{r}E_3(r,w).
	\end{equation}
	Thus integrating in $(\rho,1)$ one finds
	\[
	H_3(\rho,w)\le \exp\left(-2\int_r^1 E_3(r,w) \frac{dr}{r}\right) \, H(1,w).
	\]
	In order to derive \eqref{eq:nice} from this we need to prove that, as long as $r$ satisfies $E_3(r,w)\ge-\delta_\circ$:
	\[
	\int_r^1 (E_3(r,w))_- \frac{dr}{r}\le C_\circ<+\infty,
	\]
	The difficulty is that $\rho$ could be arbitrarily small, thus this estimate must be uniform in $\rho$. We show this bound using an appropriate epiperimetric inequality.
	
	\subsubsection{An epiperimetric inequality for \texorpdfstring{$E_3\le0$}{}}\label{sec:epi}
	Recall that we write
	\[
	w = u-\tfrac12x_2^2,\quad w_r=w(r\cdot),\quad \tilde w_r=\frac{w_r}{\|w_r\|_{L^2(\de B_1)}}.
	\]
	and we denote by $z_r$ the 3-homogeneous extension of $\tilde w_r$ in $B_1$.
	
	Since large frequencies are not a problem, we may as well work under the assumption that the frequency is bounded above by a large constant, which in turns gives a doubling property (cf. \Cref{lem:doubling}).
	
	We fix $r\in(0,1/2)$.
	\begin{lemma}\label{lem:apply projection epiperimetric}
		There are universal $\eps_\circ,L_\circ$ such that, if
		\[
		2\le \phi(1,w_r)\quad\text{ and }\quad \phi(2,w_r)\le 10
		\]
		and $z_r$ is the 3-homogeneous extension of $\tilde w_r|_{\de B_1}$ to $B_1$, then there is $\zeta_r\in H^1(B_1)$ with $\zeta_r = \tilde w_r=z_r$ on $\de B_1$ and satisfying $$(1-\eps_\circ)E_3(\zeta_r)\le E_3(z_r)\quad\text{ and }\quad\zeta_r(x) \ge -L_\circ |x_2| \text{ in }B_1.$$
	\end{lemma}
	We will set $\zeta_r$ as the solution of the thin obstacle problem in $B_1$, with data $z_r=\tilde w_r$ on $\de B_1$. This Lemma in particular refines the epiperimetric inequality of \cite[Proposition 6.1]{cubicEpiperimetric} in the 2D case, as the result proved there is not sufficient to deduce \Cref{prop:extended monneau}. As a general approach we adopt the one of \cite{Weiss}.
	
	We define the linear space
	\[
	\mathcal P_3:=\big\{p(x)\text{ 3-homogeneous},\quad \supp\Delta p\subset\{x_2=0\},\quad p\equiv0\text{ on }\{x_2=0\}\big\}.
	\]
	As we work in 2D, an orthonormal basis is given by $q^\odd := a_o \, x_2(x_2^2-3x_1^2),$ $q^\even := a_e\, |x_2|(x_2^2-3x_1^2)$ for suitable positive constants $a_o,$ $a_e$. We denote by $\Pi_3 $ the orthogonal projection onto $\mathcal P_3$ in $L^2(\de B_1)$. Note that
	\[
	\Delta \Pi_3 c = \scalar{c,q^\even} \Delta q^\even,\quad\lap q^\even = -6 c_e\,  x_1^2 \delta_0(x_2)\le 0,
	\]
	where $\delta_0$ is Dirac's delta function.
	
	For notational convenience, given functions $g,h\in H^1(B_1)$ we set
	\[
	W_3(g) := \int_{B_1} |\nabla g|^2-3\int_{\de B_1} g^2,\qquad W_3(g;h):= \int_{B_1}\nabla g\cdot\nabla h - 3\int_{\de B_1} gh.
	\]
	As $W_3$ is, up to a renormalization factor, equal to $E_3$, we prove an epiperimetric inequality for $W_3$.
	Using that $q\lap q=0$ and $r\de_r q=3q$ for all $q\in \mathcal{P}_3,$ one readily finds
	\begin{equation}\label{eq:weiss on P3}
		W_3(q)=0\quad\forall q\in\mathcal P_3.
	\end{equation}
	We will also use repeatedly the identity:
	\begin{equation}
		W_3(q+ g) = W_3(g) -2\int_{B_1} g \lap q,\qquad \forall q\in\mathcal  P_3,\ g\in H^1(B_1).\label{eq:weiss on P3+something}
	\end{equation}
	\begin{lemma}\label{lem:epiperimetric projection}
		For all $M>0$ there is $\eps=\eps(M)>0$ such that for any 3-homogeneous $c\in H^1(B_1)$ satisfying $$c(\cdot,0)\ge0\quad\text{ and  }\quad\Delta\Pi_3 c\le M\|\Pi_3c-c\|_{L^2(\de B_1)}\delta_0(x_2),$$ there is $\tilde\zeta\in c+H^1_0(B_1)$ satisfying 
		\[
		\tilde\zeta(\cdot,0)\ge0\quad\text{  and }\quad (1-\eps)W_3(\tilde\zeta)\le W_3(c).\]
	\end{lemma}
	\begin{proof}
		Suppose there are 3-homogeneous functions $c_k\in H^1(B_1)$ satisfying $c_k\ge0$ on $\{x_2=0\}$ and such that, defining
		\[
		q_k:= \Pi_3 c_k,\quad \delta_k:=\|c_k-q_k\|_{L^2(\de B_1)},
		\]
		we have
		\begin{equation}
			\Delta q_k\le M\delta_k\delta_0(x_2)\label{eq:lap bound artificial}
		\end{equation}
		but such that
		\begin{equation}\label{eq:almost minimizing assumption general proof}
			W_3(c_k)\le (1-\tfrac1k)W_3(v_k)\quad \forall v_k\in c_k+H^1_0(B_1), v_k\ge0\text{ on }\{x_2=0\}.
		\end{equation}
		In particular
		\[
		W_3(c_k)\le0.
		\]
		
		Setting
		\[
		w_k=\frac{c_k-q_k}{\delta_k},
		\]
		note that
		\[
		w_k \perp_{L^2(\de B_1)} \mathcal P_3,\quad \|w_k\|_{L^2(\de B_1)}=1,\quad w_k(\cdot,0)\ge0.
		\]
		% As $q_k$ are 3-homogeneous we have $x\cdot \nabla q_k-3q_k\equiv0$, so an integration by parts gives
		% \[
		%     W_3(c_k) = W_3(c_k-q_k)-2\int_{B_1}\Delta q_k(c_k-q_k).
		% \]
		Applying \eqref{eq:weiss on P3+something} to $q=q_k$ and $g=c_k-q_k$ and recalling $W_3(c_k)\le0$ yields
		\begin{equation}\label{eq:k}
			W_3(c_k-q_k) \le 2\int_{B_1}(c_k-q_k)\lap q_k.
		\end{equation}
		Recall that $\Delta q_k$ is supported on $\{x_2=0\}$ where $c_k-q_k=c_k\ge0$, thus using \eqref{eq:lap bound artificial} and the trace bound
		\[
		\|c_k-q_k\|_{L^1( B_1')}\le C(\|\nabla(c_k-q_k)\|_{L^2(B_1)}+\delta_k)
		\]
		we find
		\begin{equation}\label{eq:lap bounds}
			\int_{B_1}\Delta q_k(c_k-q_k)\le C\delta_k \int_{B_1'} c_k-q_k \le CM\delta_k\|(c_k-q_k)\|_{L^1(B_1')}\le CM\delta_k^2+CM\delta_k\|\nabla(c_k-q_k)\|_{L^2(B_1)}.
		\end{equation}
		Thus \eqref{eq:k} becomes
		\[
		\|\nabla (c_k-q_k)\|^2_{L^2(B_1)} \le 3\delta_k^2 + CM\delta_k^2 + CM\delta_k\|\nabla(c_k-q_k)\|_{L^2(B_1)}\le C'M\delta_k^2 + \tfrac12\|\nabla(c_k-q_k)\|_{L^2(B_1)}^2
		\]
		and we deduce
		\[
		\|w_k\|_{H^1(B_1)}\le C M.
		\]
		It follows that, up to a subsequence,
		\[
		w_k\weak w_\infty \quad\text{weakly in }H^1(B_1),
		\]
		so in particular
		\begin{equation}\label{eq:orthogonality assumption proof projection}
			w_\infty\perp_{L^2(\de B_1)}\mathcal P_3,\quad \|w_\infty\|_{L^2(\de B_1)}=1,\quad w_\infty(\cdot,0)\ge0.
		\end{equation}
		Recalling \eqref{eq:almost minimizing assumption general proof}, the same integration by parts leading to \eqref{eq:k} gives
		\[\begin{split}
			W_3(c_k-q_k) &- 2\int_{B_1}\Delta q_k(c_k-q_k) =W_3(c_k)\le (1-\tfrac1k) W_3(v_k)\\
			&=(1-\tfrac1k)W_3(v_k-q_k)- 2(1-\tfrac1k)\int_{B_1}\Delta q_k(v_k-q_k),
		\end{split}\]
		for all $v_k\in c_k+H^1_0(B_1)$ with $v_k(\cdot,0)\ge0$.
		Hence we get
		\[\begin{split}
			W_3(w_k)&\le (1-\tfrac1k)W_3(\varphi_k)
			-2\int_{B_1}(\delta_k^{-1}\Delta q_k)\left((1-\tfrac1k)\varphi_k-w_k\right)\quad\forall \varphi_k\in w_k+H^1_0(B_1),\,\varphi_k(\cdot,0)\ge0.
		\end{split}\]
		Given, $\phi\in H^1_0(B_1\setminus\{x_2=0\})$, testing with $\varphi_k= w_k+\phi$ gives
		\[\begin{split}
			\int_{B_1}|\nabla w_k|^2&\le (1-\tfrac1k)\int_{B_1}|\nabla (w_k+\phi)|^2 +\tfrac3k\int_{\partial B_1}w_k^2+\tfrac2k\int_{B_1}(\delta_k^{-1}\Delta q_k)w_k.
		\end{split}\]
		Now \eqref{eq:lap bounds} together with $\|w_k\|_{H^1}\le CM$ gives
		\[
		\int_{B_1}|\nabla w_k|^2\le \int_{B_1}|\nabla(w_k+\phi)|^2 + o(1).
		\]
		Letting $k\to+\infty$ weak convergence in $H^1$ yields
		\[
		\int_{B_1}|\nabla w_\infty|^2\le \int_{B_1}|\nabla (w_\infty+\phi)|^2 \quad\text{ for all }\phi\in H^1_0(B_1\setminus\{x_2=0\}),
		\]
		so, as $w_\infty$ is 3-homogeneous,
		\begin{equation}\label{eq:EL}
			0=\int_{\mathbb S^1} \nabla w_\infty\cdot \nabla \phi-9w_\infty\phi \quad\text{ for all }\phi\in H^1_0(\mathbb S^1\setminus\{x_2=0\}).
		\end{equation}
		Taking $\phi= q^\even$ and integrating by parts we get to
		\begin{align*}
			0&=\int_{\mathbb S^1} \nabla w_\infty\cdot \nabla q^\even =-\int_{\mathbb S^1} w_\infty (\lap_{\mathbb S^1}+ 9)q^\even = a_e \int_{\{x_2=0\}\cap\mathbb S^1} w_\infty x_1^2.
		\end{align*}
		As $w_\infty(\cdot,0)\ge0$ we find $w_\infty\equiv0$ on $\{x_2=0\}$. Clearly, by \eqref{eq:EL}, $w_\infty$ is harmonic in $B_1\setminus\{x_2=0\}$, hence we get $w_\infty\in\mathcal P_3$, contradicting \eqref{eq:orthogonality assumption proof projection}.
	\end{proof}
	The next lemma shows that $w$ may approach $\mathcal{P}_3$ at an angle universally bounded from below in $L^2(\de B_1)$, justifying why the existence of $M$ in the previous Lemma is indeed satisfied with $M=C_\circ$ for traces $c\in H^1(\de B_1)$ of the specific form $u-\tfrac12 x_2^2$.
	\begin{lemma}\label{lem:control projection P3}
		There is $C_\circ>0$ such that if $u$ solves \eqref{eq:classical obstacle} with $0\in\de\{u>0\}$ then $w=u-\tfrac12x_2^2$ satisfies
		\[
		\Delta\Pi_3w\le C_\circ\|\Pi_3w-w\|_{L^2(\de B_1)}\delta_0(x_2).
		\]
	\end{lemma}
	\begin{proof}
		We argue by contradiction. Assume there are $u_k$ solving \eqref{eq:classical obstacle} with $0\in\de\{u_k>0\}$ but such that, setting $w_k=u_k-\tfrac12x_2^2$, we have
		\begin{equation}\label{eq:super big projection}
			\scalar{\Pi_3 w_k,q^\even}_{L^2(\de B_1)}<-k\|\Pi_3 w_k -w_k\|_{L^2(\de B_1)}.
		\end{equation}
		Writing
		\[
		q_k := \scalar{w_k,q^\odd}_{L^2(\de B_1)}q^\odd,\quad\delta_k:=\|w_k-q_k\|_{L^2(\de B_1)}
		\]
		and setting
		\[
		v_k = \frac{w_k-q_k}{\delta_k}
		\]
		then we can write uniquely
		\begin{equation}\label{e:defvk}
			v_k = \alpha_k q^\even + v_k^\perp
		\end{equation}
		where
		\[
		v_k^\perp\perp\mathcal P_3\quad\text{and}\quad \|v_k^\perp\|_{L^2(\de B_1)}^2+\alpha_k^2=1.
		\]
		\Cref{eq:super big projection} reads as
		\[
		\alpha_k<-k\|v_k^\perp\|_{L^2(\de B_1)}\le0.
		\]
		As $|\alpha_k|\le1$, we deduce
		\[
		v_k^\perp\to0\quad\text{in }L^2(\de B_1)\quad\text{and}\quad \alpha_k\to -1.
		\]
		Consider $Q$ to be the harmonic replacement of $-q^\even$ and $h_k$ the harmonic replacement of $v_k^\perp$. Notice that by the mean value property
		\[
		v_k^\perp \to 0\text{ in $L^2(\de B_1)$ implies that }h_k\to 0\text{ uniformly  in }\overline{B_{1/2}}.
		\]
		Furthermore, $Q\ge c_0>0$ in $B_{\rho_0}$ for some universal $c_0,\rho_0>0$, this can be proved using the mean value property and the fact that $\fint_{\de B_1} q^\even <0$. 
		%(by the strong maximum principle, as $q^\even(0)=0$ and $\Delta (-q^\even)\ge0$). 
		Now $ v_k$ is superharmonic since
		\[
		\Delta v_k = \frac{1}{\delta_k}\Delta\tilde w_k = \frac1{\|w_k\|_{L^2(\de B_1)}\delta_k}\Delta (u_k-\tfrac12x_2^2) \le0,
		\]
		so, by \eqref{e:defvk} and comparison we find $v_k \ge -\alpha_k Q+h_k$ in $B_1$. Thus, in $B_{\rho_0}$:
		\[
		\begin{split}
			v_k \ge -\alpha_k c_0 -\|h_k\|_{L^\infty(B_{\rho_0})}=(1-o(1))c_0-o(1).
		\end{split}
		\]
		But this is impossible because $v_k(0)=0$ while $c_0>0$.
	\end{proof}
	\begin{proof}[Proof of \Cref{lem:apply projection epiperimetric}]
		Let $C_\circ$ be given by \Cref{lem:control projection P3}, and let $\eps_\circ$ and $\tilde\zeta_r$ be given by \Cref{lem:epiperimetric projection} applied to $z_r$ with $M=C_\circ$.
		Set $\zeta_r$ as the solution of the thin obstacle problem in $B_1$ with data $z_r=\tilde w_r$ on $\de B_1$.
		Since 
		\[
		\int_{B_1}|\nabla \zeta_r|^2\le\int_{B_1}|\nabla v|^2\quad\forall v\in \zeta_r+H^1_0(B_1)\text{ with }v(\cdot,x_2)\ge0,
		\]
		using $v=\tilde\zeta_r$ and the fact that $\zeta_r=\tilde\zeta_r=z_r$ on $\de B_1$ with $\|z_r\|_{L^2(\de B_1)}=1$ we deduce
		\[
		(1-\eps_\circ)E_3(\zeta_r)\le (1-\eps_\circ)E_3(\tilde\zeta_r)\le E_3(z_r).
		\]
		On the other hand, \Cref{lem:lip estimate rhs}, together with the doubling estimate of \Cref{lem:doubling}, implies $|\nabla z_r|\le L_\circ$ for a universal $L_\circ$. As $z_r\ge0$ on $\{x_2=0\}$ this implies
		\[
		\zeta_r =z_r\ge -L_\circ|x_2|\quad \text{on }\de B_1.
		\]
		Since $ -L_\circ|x_2|$ solves the thin obstacle as well, we conclude by comparison (see \Cref{lem:comparison thin obstacle}).
	\end{proof}

	\subsubsection{Proof of \Cref{prop:extended monneau}}
	We consider the scale
	\[
	R:= \inf\{s\in(0,1)\,: \phi(s,u-\tfrac12x_2^2)\ge 3\},
	\]
	with the understanding that $R:=1$ if that set is empty. Then using Monneau monotonicity (\Cref{lem:monneau}) and $C^{1,1}$ bounds (\eqref{eq:universal upper bound}) we find
	\[
	\|(u-\tfrac12x_2^2)(s\cdot)\|_{L^2(\de B_1)}\le s^3 \|u-\tfrac12x_2^2\|_{L^2(\de B_1)}\le C\quad\text{ for all }s\in[R,1].
	\]
	Thus if $\rho\ge R/2$ we are done, otherwise rescaling by $R/2$, it suffices to prove
	\begin{equation}
		\|\tilde w_r\|_{L^2(\de B_1)}\le C_\circ r^3 \quad \text{ for all }r\in[\rho',1],\label{eq:sbs} 
	\end{equation}    
	where
	\[
	w_r:=(\tfrac12R r)^{-2}(u-\tfrac12 x_2^2)(\tfrac12R r\cdot),\qquad \tilde w_r:=w_r/\|w_r\|_{L^2(\de B_1)},\qquad\rho':=2\rho/R<1.
	\]
	Now for $r\in[\rho',1]$ using Almgren's monotonicity and the definition of $R$ we have the frequency pinch:
	\[
	3-\delta_\circ\le\phi(r,w)\le\phi(2,w) =  \phi( R, u-\tfrac12 x_2^2)\le 3<10,
	\]
	therefore we can apply \Cref{lem:apply projection epiperimetric} to find $\zeta_r\ge -L_\circ|x_2|$ agreeing with $\tilde w_r$ on $\de B_1$ and such that
	\begin{equation}\label{eq:epi proof monneau}
		E_3(z_r) \ge (1-\eps_\circ)E_3(\zeta_r).
	\end{equation}
	Now we prove that $\tilde w_r$ was a minimizer up to an integrable error:
	\begin{equation}\label{eq:goal 1 extend monneau}
		E_3(\zeta_r)\ge E_3(r)-Cr^{1-\delta_\circ}\quad \text{for all }r\in(\rho',1),
	\end{equation}
	where we denote $E(r):=\phi(r,w)-3$.
	Since $\zeta_r= \tilde w_r$ on $\de B_1$, we have
	\[
	E_3(r) -E_3(\zeta_r)=\int_{B_1}|\nabla \tilde w_r|^2-|\nabla\zeta_r|^2= 2\int_{B_1}\lap\tilde w_r(\zeta_r-\tilde w_r)-\int_{B_1}|\nabla \zeta_r|^2\le  2\int_{B_1}\lap\tilde w_r(\zeta_r-\tilde w_r).
	\]
	Now, $\lap\tilde w_r = -{r^2}\ind_{\{u(r\cdot)=0\}}/{\|w_r\|_{L^2(\de B_1)}}$ 
	so, using $\zeta_r\ge-L_\circ |x_2|$, we get
	\[
	-\tilde w_r\lap \tilde w_r\le0,\qquad \zeta_r\lap \tilde w_r \le \frac{r^2 L_\circ|x_2|}{\|w_r\|_{L^2(\de B_1)}}\ind_{\{u(r\cdot)=0\}}.
	\]
	Using \Cref{lem:lip estimate rhs} and doubling (\Cref{lem:doubling} applies) we find
	\[
	\{u_r=0\}\cap B_1 \subset \{|x_2|\le C r^{-2} \|w_r\|_{L^2(\de B_1)}\},
	\]
	so
	\[
	\zeta_r\lap \tilde w_r\le C\ind_{\{|x_2|\le Cr^{-2}\|w_r\|_{L^2(\de B_1)}\}}.
	\]
	As $\phi(w,r)\ge3-\delta_\circ$ for all $r\in(s,1)$, \Cref{lem:monneau} yields 
	\[
	r^{-2}\|w_r\|_{L^2(\de B_1)}\le r^{1-\delta_\circ}\|w\|_{L^2(\de B_1)},
	\]
	so
	\[
	\int_{B_1}\zeta_r\Delta\tilde w_r\le C r^{-2} \|w_r\|_{L^2(\de B_1)}\le Cr^{1-\delta_\circ}\|w\|_{L^2(\de B_1)},
	\]
	and \eqref{eq:goal 1 extend monneau} follows.
	
	Therefore, combining \cref{eq:goal 1 extend monneau,eq:epi proof monneau} and \Cref{lem:derivative weiss} we find the differential inequality:
	\[
	\frac{d E_3(r)}{d r}\ge \frac2r(3\eps_\circ+E_3(r))(-E_3(r)) -Cr^{-\delta_\circ}\quad\text{for all }r\in(\rho',1).
	\]
	As $E_3(r)\ge -\delta_\circ$, choosing $\delta_\circ:=\eps_\circ$ universally small yields
	\[
	\frac{d}{dr}\big(r^{\eps_\circ}E_3(r)\big) \ge -C  r^{\eps_\circ-\delta_\circ}=-C\quad\text{ for all }r\in(\rho',1),
	\]
	which can be rewritten as
	\[
	\frac{E_3(r)}{r}\ge -\frac{C}{\eps_\circ} r^{-\eps_\circ} - \frac1{\eps_\circ}E_3(r)'\quad\text{for all }r\in(\rho',1),
	\]
	so recalling $-\delta_\circ\le E_3(r)\le 0$, \Cref{lem:monneau} gives
	\[
	\log\frac{H_3(1)}{H_3(\rho')} \ge 2\int_{\rho'}^{1} \frac{E_3(s)}{r}\,dr\ge -\frac{C}{\eps_\circ}\int_0^{1}\frac{ds}{s^{\eps_\circ}} -\frac{1}{\eps_\circ}[\phi(1,w)-\phi(0+,w)] \ge -C,
	\]
	since $\eps_\circ$ is universal, \eqref{eq:sbs} follows.
	
	\subsection{Proof of \Cref{theorem:2D 1}}\label{sec:proofmain}
	Recall the notation $\dist_{\Sigma_1}$ set in \eqref{eq:distance from Sigma1}.
	
	Given $u$ solving \eqref{eq:classical obstacle} in $B_1\subset\R^2$ with $f\equiv1$ and $0\in\de\{u>0\}$, we aim to show
	\[
	D^2 u \ge -\bar Cr\quad\text{in }B_r\qquad\text{for any }r\in(0,1/2).
	\]
	For $\delta_\circ$ given by Proposition~\ref{prop:extended monneau}, set
	\[
	r_u:= \inf\ \{r\in (0,\tfrac12)\,:\,\phi(r,u-\tfrac12(\nu\cdot x)^2)\ge 3-\delta_\circ\text{ for some }\nu\in\mathbb S^1\},
	\]
	and let $\nu_u$ realize the infimum. If the set is empty set $r_u:=\tfrac12$. 
	On one hand \Cref{prop:extended monneau} gives
	\begin{equation}\label{eq:distance Sigma1 proof main theorem}
		\dist_{\Sigma_1}(r^{-2}u(r\cdot))\le \left\|(r^{-2}u(r\cdot)-\tfrac12(\nu_u\cdot x)^2)\right\|_{L^2(\de B_1)}\le C_\circ r\quad\forall r\in[r_u,\tfrac12],
	\end{equation}
	so \Cref{lem:semiconvexity} implies
	\begin{equation}
		D^2u\ge -C_3r\quad\text{in }B_r\qquad \forall r\in[r_u,\tfrac12].\label{eq:r>r_u}    
	\end{equation}
	On the other hand the functions
	\[
	v(x) := r^{-2}_uu(r_ux)
	\]
	solve \eqref{eq:classical obstacle} in $B_1$ and satisfy
	\[
	\phi(v-p,1)\le 3-\delta_\circ\quad\forall p\in\Sigma_1\qquad\text{and}\qquad\dist_{\Sigma_1}(v)\le C_\circ r_u,
	\]
	so applying \Cref{prop:low frequency} to $v$ and rescaling we find
	% \begin{equation*}
		%     -D^2 u(r_u x) \le C_{\delta_\circ} \dist_{\Sigma_1}(v) |x|\text{ for all }x\in B_{1/2}
		% \end{equation*}
	\begin{equation}
		D^2u\ge -C_\circ C_{\delta_\circ}r\quad\text{in }B_r\qquad\text{for any }r<\tfrac12r_u.\label{eq:r<r_u}    
	\end{equation}
	Joining \eqref{eq:r>r_u} with \eqref{eq:r<r_u} we conclude.
	
	\section{General Case: Preliminary Results}\label{sec:preliminary rhs}
	\subsection{Normalization assumptions on $\bm f$}\label{eq:assumptions f}
	In this part we adapt the approach to more general right-hand sides. 
	Since we want to study the zero set of $u,$ we may always consider $u(x)/f(0)$ and assume:
	\[
	f(0)=1.
	\]
	Since we want to prove local estimates, we can do a universal rescaling and assume that the right-hand sides are flat, that is
	\[
	\|\nabla f \|_{C^\alpha(B_1)}\le c(n).
	\]
	Taking $c$ small we may ensure $\mu\ge \tfrac12$. Additionally, this allows us to use safely the adapted blow-up profiles $p_\nu$ defined in \eqref{eq:def_pnu_RHS}, and the choice of $c$ ensures that $p_\nu\ge 0$ in $B_4$. 
	
	We will refer to these right-hand sides as ``normalized''. It is clear that, after normalization, universal constants may depend only on $n$ and $\alpha$. 
	\subsection{Monotonicity formulas}
	
	We collect here the monotonicity formulas that we will need to tackle the case of a general $f$. The monotonicity of the frequency is more delicate to adapt, so we postpone the proof in \Cref{app:frequency}. Following \cite{FROS24}, given $r\in(0,1)$ we write $w_r(x):=w(rx)$ and define
	\[
	D(r,w):=\int_{B_1}|\nabla w_r|^2, \quad H(r,w):=\int_{\de B_1}w_r^2.
	\]
	We denote the truncated frequency as
	\[
	\phi^\tau(r,w):=\frac{D(r,w)+\tau r^{2\tau}}{H(r,w)+r^{2\tau}},\qquad H^\tau_\lambda(r,w):=r^{-2\lambda}(H(r,w)+r^{2\tau})
	\]
	where $\tau>0$. We will also fix the truncation at the value
	\begin{equation}
		\tau:= 3+\tfrac12\alpha,\label{eq:value of tau}
	\end{equation}
	any other value $3<\tau<3+\alpha$ would be possible. The main result of this subsection, proven in \Cref{app:frequency}, is the following.
	\begin{proposition}[Truncated Frequency Monotonicity] \label{corollary:frequency formula}
		There are universal positive constants $C_\circ$ large and $\alpha_\circ,$  $r_\circ$ small with the following property. Suppose $u$ solves \eqref{eq:obstacle} with $0\in\de\{u>0\}$ and $f$ normalized as in \Cref{eq:assumptions f}. Suppose that, writing for some $\nu\in\mathbb S^{n-1}$
		\[
		w := u-p_\nu,
		\]
		there is $s\in(0,r_\circ)$ satisfying
		\[
		\phi^\tau(s,w) + \tfrac{C_\circ}{\alpha_\circ}s^{\alpha_\circ}\ge 2.5.
		\]
		Then, for any $r\in(s,1/2),$
		\[
		\frac{d}{dr}\phi^\tau(r,w) \ge - C_\circ r^{\alpha_\circ-1}\quad\text{ and }\quad\frac{\int_{B_1}w_r\Delta w_r}{H(r,w) + r^{2\tau}}\ge -C_\circ r^{\alpha_\circ}.
		\]
	\end{proposition}
	
	The symbols $C_\circ$ and $\alpha_\circ$ will be reserved to these two specific constants.
	
	We now state and prove a few consequences of \Cref{corollary:frequency formula}. We denote the truncated frequency gap as
	\begin{equation}\label{eq:truncated frequency gap}
		E_\lambda^\tau(r,w) := \phi^\tau(r,w) + \tfrac{C_\circ}{\alpha_\circ}r^{\alpha_\circ} - \lambda,
	\end{equation}
	where $C_\circ,\alpha_\circ$ are given by \Cref{corollary:frequency formula}. In particular, \Cref{corollary:frequency formula} states that
	\[
	E_\lambda^\tau(r,w)\text{ is monotone for }r>s,\text{ provided }\phi^\tau(s,w) + \frac{C_\circ}{\alpha_\circ}s^{\alpha_\circ}\ge2.5\text{ for some }s\in(0,r_\circ).
	\]
	
	\begin{lemma}[Truncated Monneau Monotonicity]\label{lem:monneau rhs}
		There are universal constants $C,$ $\delta$ such that if
		\[
		\phi^\tau(s,w) + \frac{C_\circ}{\alpha_\circ}s^{\alpha_\circ}\ge 2.5\quad\text{for some }s\in(0,r_\circ)
		\]
		then
		\[
		\frac{d}{dr}\log(r^{-2\lambda}(H(r,w)+r^{2\tau})) \ge \frac2r E^\tau_\lambda(r,w) - Cr^{\delta-1}.
		\]
		In particular, if
		\[
		\phi^\tau(s,w) + \frac{C_\circ}{\alpha_\circ}s^{\alpha_\circ}\ge \lambda\quad\text{for some }\lambda\ge 2.5\text{ and some }s\in (0,r_\circ),
		\]
		then
		\[
		\log [r^{-2\lambda}(H(r,w)+r^{2\tau})]+ \tfrac C\delta r^\delta\quad\text{ is non-decreasing for all }r\in(s,r_\circ).
		\]
	\end{lemma}
	\begin{proof}
		As shown, for instance, in \cite{FROS24} we have for any $w$ smooth enough
		\[
		\frac{d}{dr}\log(r^{-2\lambda}(H(r,w)+r^{2\tau})) = \frac2r(\phi^\tau(r,w)-\lambda) + \frac2r\frac{\int_{B_1}w_r\Delta w_r}{H(r,w) + r^{2\tau}},
		\]
		so we conclude by \Cref{corollary:frequency formula}.
	\end{proof}
	
	We will also use the following doubling estimate, which is a consequence of \cite[Lemma 4.1 (a)]{FROS24}.
	\begin{lemma}\label{lem:doubling rhs}
		There is a universal $C>0$ such that, if for some $s\in(0,r_\circ)$ we have
		\[
		\phi^\tau(s,w) + \frac{C_\circ}{\alpha_\circ}s^{\alpha_\circ} \ge2.5\quad\text{and}\quad \phi^\tau(2s,w) + \frac{C_\circ}{\alpha_\circ}(2s)^{\alpha_\circ}\le10,
		\]
		then
		\[
		\|w_{2s}\|^2_{L^2(\de B_1)} +(2s)^{2\tau}\le C(\|w_s\|^2_{L^2(\de B_1)}+s^{2\tau}).
		\]
	\end{lemma}
	
	Finally, we derive in this truncated framework the derivative of the frequency gap.
	\begin{lemma}\label{lem:derivative frequency gap RHS}
		Given $\lambda>0$, there are universal constants $C,$ $\beta$ depending further on $\lambda$ with the following property. Suppose
		\[
		\phi^\tau(s,w) + \frac{C_\circ}{\alpha_\circ}s^{\alpha_\circ} \ge2.5\text{ for some }s\in(0,r_\circ)\quad\text{and}\quad E^\tau_\lambda(R,w)\le0\text{ for some }R>s.
		\]
		Then, denoting by $z_r$ the $\lambda$-homogeneous extension of $w_r$
		\[
		\frac{d}{dr} E^\tau_\lambda(r,w) \ge \tfrac{n-2+2\lambda}{r}(E^\tau_\lambda(r,z_r)-E^\tau_\lambda(r,w)) - \tfrac2r [E^\tau_\lambda(r,w)]^2 - Cr^{\beta-1}\quad\text{for any }s<r<R.
		\]
	\end{lemma}
	\begin{proof}
		First, under the stated assumptions we have
		\[
		-\lambda\le E^\tau_\lambda(r,w)\le 0\quad\text{for any }r\in(s,R),
		\]
		where on one side we use the lower bound $\phi^\tau\ge0$, while the other inequality uses the monotonicity of $E^\tau_\lambda$.
		
		Setting
		\[
		W_\lambda^\tau:= (\phi^\tau-\lambda) r^{-2\lambda}(H(r,w)+r^{2\tau}),\quad
		W_\lambda := r^{-2\lambda}(D(r,w)-\lambda H(r,w))
		\]
		we can write
		\[
		W_\lambda^\tau = W_\lambda + (\tau-\lambda)r^{2(\tau-\lambda)},
		\]
		so using \cite[Equation (2.1)]{CSV20} we get
		\[
		\frac{d}{dr} W_\lambda^\tau \ge \tfrac{n-2+2\lambda}{r}r^{-2\lambda}(D(z_r) - D(r,w))
		= \tfrac{n-2+2\lambda}r(E_\lambda^\tau(z_r)-E_\lambda^\tau(r,w))[r^{-2\lambda}(H(r,w)+r^{2\tau})].
		\]
		Thus
		\[
		\frac{d}{dr} E^\tau_\lambda\ge\frac{d}{dr} (\phi^\tau-\lambda) \ge \frac{n-2+2\lambda}{r}(E^\tau_\lambda(z_r)-E^\tau_\lambda(r,w))
		- (\phi^\tau(r,w)-\lambda) \frac{d}{dr}\log[r^{-2\lambda}(H(r,w)+r^{2\tau})].
		\]
		Using
		\[
		\frac{d}{dr}\log[r^{-2\lambda}(H(r,w)+r^{2\tau})]\ge \frac{2}{r}E^\tau_\lambda - Cr^{\eps-1},\quad -\lambda\le \phi^\tau-\lambda\le E_\lambda^\tau \le0
		\]
		we get to
		\[
		[-(\phi^\tau-\lambda)] \frac{d}{dr}\log[r^{-2\lambda}(H(r,w)+r^{2\tau})] \ge \frac{2}{r} [-(\phi^\tau-\lambda)]E_\lambda^\tau - Cr^{\eps-1}.
		\]
		We conclude computing
		\[
		-\tfrac{2}{r} (\phi^\tau-\lambda)E_\lambda^\tau = -\tfrac{2}{r} ((E_\lambda^\tau)^2- Cr^{\alpha_\circ} E_\lambda^\tau) \ge -\tfrac2r (E_\lambda^\tau)^2 - C r^{\alpha-1}.\qedhere
		\]
	\end{proof}
	
	% We also extend \Cref{lem:frequency gap} in this case.
	% Recall that $0\in\Sigma_{n-1}(u)$ if and only if $u(x) = \tfrac{f(0)}{2}(x\cdot e )^2 + o(r^2)$ in $B_r$ for some $e\in\mathbb S^{n-1}$.
	% \begin{lemma}\label{lem:frequency gap RHS} \textbf{NO NEED AT THE MOMENT}
		%     Let $u$ be a solution of \eqref{eq:obstacle} in $B_1\subset\R^n$ with $0\in\Sigma_{n-1}(u)$ and blow-up $p^*_\nu$. Then either $u\equiv p^*_\nu$ or $\phi^\tau(r,u-p^*)+C_\circ r^{\alpha_\circ}\ge3$ for all $r\in(0,1)$.
		% \end{lemma}
	% \begin{proof} 
		%     If $u\not\equiv p^*$ then $\phi^\tau(r,u-p^*)+C_\circ r^{\alpha_\circ}$ is well-defined and monotone by \Cref{corollary:frequency formula}, so it exists $\lambda^* = \phi^\tau(0^+,u-p^*)$. It is proven in \cite[Lemmas A.7 and A.8]{CinftyRectifiable} that $\lambda^*>2$ and, when $\lambda^*<3$, is the homogeneity of some homogeneous solution of the thin obstacle problem \eqref{eq:thin obstacle}. It follows from \cite{FranceschiniSavin25} that $\lambda^*\ge3$.
		% \end{proof}

	\subsection{Semiconvexity decay with  frequency $\bm{<3}$}
	
	This sub-section is overall quite similar to the 2D case with mostly notational changes and the appearance, in the statements, of a scale $r$, due to the presence of the right-hand side.
	
	\begin{proposition}\label{prop:low frequency f}
		Given $\delta\in(0,\tfrac\alpha5)$ there is $M=M(n,\alpha,\delta)>0$ large and $\rho=\rho(n,\alpha,\delta)>0$ small such that the following holds.
		
		Assume $u$ solves \eqref{eq:obstacle} in $B_1\subset\R^n$ for some normalized right-hand side $f$,  $0\in\de\{u>0\}$, and for some $r\in(0,\rho)$ it holds
		\begin{equation}
			\label{eq:assfreq<3}
			\phi^{\tau}(r,u-p_\nu)+ \frac{C_\circ}{\alpha_\circ}r^{\alpha_\circ}\le 3-\delta\quad\text{ for all }\nu\in\S^{n-1}.
		\end{equation}
		Then if we set
		\[
		\eps:=r^{-2}\min_{\nu\in\S^{n-1}}H(r,u-p_{\nu})^{1/2},
		\]
		we have \begin{equation}\label{eq:goodconclusion}
			\Delta_{e^\perp} u(x)\ge - M \eps\frac{|x|}{r}\quad \text{ for all } x\in B_{r/2} \text{ and all unit vectors }e.
		\end{equation}
	\end{proposition}
	
	We explain the proof giving momentarily for granted \Cref{lem:low frequency close poly f} below, which is the most delicate point of this section. Nevertheless the high-level strategy is the same in both proofs, \Cref{lem:low frequency close poly f} requiring a next-order analysis compared to \Cref{prop:low frequency f}.
	
	Before embarking in the proof, we record a simple observation, useful when working with an upper bound on the truncated frequency.
	
	\begin{remark}\label{rmk:not-toosmall}
		If $\phi^\tau(r,w)\le \tau'<\tau,$ then $H(r,w)\neq 0$,
		\[
		\phi(r,w)\le \phi^\tau(r,w)\quad \text{ and }\quad  H(r,w)\ge \tfrac{\tau-\tau'}{\tau} r^{2\tau}.
		\]
		Indeed rearranging the definition
		\begin{align*}
			D(r,w)=\phi^\tau(r,w) H(r,w)+\underbrace{(\phi^\tau(r,w)-\tau)}_{\le 0}r^{2\tau}\le\phi^\tau(r,w) H(r,w).
		\end{align*}
		By the same rewriting, since $D(r,w)\ge0$,
		\begin{align*}
			r^{2\tau}\le \frac{\phi^\tau(r,w)}{\tau -\phi^\tau(r,w)} H(r,w)\le \frac{\tau}{\tau-\tau'} H(r,w). 
		\end{align*}
		since the function $t\mapsto t/(\tau-t)$ is increasing for $t<\tau$.
	\end{remark}
	
	\begin{proof}[Proof of \Cref{prop:low frequency f} given \Cref{lem:low frequency close poly f}]
		We argue by contradiction, considering a sequence $u_k$ solving \eqref{eq:obstacle} for some normalized $f_k$, such that,
		\[
		\eps_k:=r_k^{-2}\min_{\nu\in\S^{n-1}}H(r_k,u_k-p_{\nu})^{1/2},\qquad r_k\to0,
		\]
		but, setting $\tilde u_k:=r_k^{-2} u_k(r_k\cdot)$, we have
		\begin{equation}
			\lap_{\xi_k^\perp}\tilde u_k(x_k)<-k \eps_k |x_k| \quad \text{ for some } x_k\in B_{1/2}\text{ and unit vectors }\xi_k.\label{eq:contr prop low freq}
		\end{equation}
		By \eqref{eq:semiconvexity_weak}, we must have $x_k\neq0$ and by \Cref{rmk:not-toosmall} and \eqref{eq:assfreq<3} we must have $\eps_k>0$.
		
		If $\eps_k\to 0$ then we can apply \Cref{lem:low frequency close poly f} and conclude that \eqref{eq:goodconclusion} holds. Then up to subsequences we may assume that
		\begin{equation}\label{eq:apply close poly f}
			\eps_k\ge \bar \eps>0\qquad\text{ for all } k.
		\end{equation}
		By $C^{1,1}$ bounds $\lap_{\xi_k^\perp} \tilde u_k\ge -|D^2 \tilde u_k |\ge -C$, so \eqref{eq:contr prop low freq} forces
		\begin{equation}\label{eq:x_k-to-zero}
			x_k\to0.
		\end{equation}
		By $C^{1,1}$ bounds and normalization, up to a subsequence we have
		\[
		\tilde u_k\to u_\infty\quad \text{in }C^1(B_{3/2}), \qquad f_k(r_k\cdot)\to {1}\quad\text{in }C^1(B_{3/2}),
		\]
		where $u_\infty$ solves \eqref{eq:obstacle} in $B_{2/3}$ with $0\in\de \{u_\infty>0\}$ and constant right-hand side $1$.
		
		We split in three cases, according to Caffarelli's dichotomy.
		\begin{enumerate}[label = \alph*)]
			\item If $0\in\mathrm{Reg}(u_\infty)$ there is $\bar r>0$ small so that, up to a rotation,
			\[
			\left|\bar r^{-2}u_\infty(\bar r\cdot)-\tfrac12{} (x_2)_+^2\right|< \delta_1\quad \text{in }B_1,
			\]
			with $\delta_1$ given by \Cref{lemma:regular case}. Since the same holds for $\bar r^{-2}\tilde u_k(\bar r\cdot)$ and $k$ large, \Cref{lemma:regular case} gives
			\[
			\lap_{\xi_k^\perp} \tilde u_k(y)\ge -C\delta_1 |y|/\bar r\quad \text{ for all }y\in B_{\bar r/4}.
			\]
			Thanks to \eqref{eq:x_k-to-zero} and \eqref{eq:apply close poly f}, for $k$ large enough we can choose $y:=x_k$, contradicting \eqref{eq:contr prop low freq}.
			\item If $0\in\Sigma_{n-1}(u_\infty)$, we denote by $p_\infty=\tfrac12{} (x\cdot\bar \nu)^2$ the blow-up of $u_\infty$ at $0$, for some unique $\bar \nu\in\S^{n-1}$. For $\nu\neq\bar \nu$ and $k$ large we have $H(r_k,\tilde u_k -p_\nu)\neq0$ and we use \Cref{rmk:not-toosmall} and \eqref{eq:assfreq<3} to find 
			\[
			\phi(1,\tilde u_k-r_k^{-2}p_{\nu}(r_k\cdot))=\phi(r_k, u_k-p_{\nu})\le \phi^\tau(r_k, u_k-p_{\nu})\le 3-\delta,
			\]
			where $p_\nu$ is the Ansatz associated with the right-hand side $f_{k,1}$ and the direction $\nu$.
			
			Passing to the limit we find
			\[
			\phi(1,u_\infty-\tfrac12 {} (x\cdot \nu)^2)\le 3-\delta,\qquad \text{ for all }\nu\neq \bar \nu,
			\]
			but, by \Cref{lem:frequency gap}, $\phi(1,u_\infty-\tfrac12 {} (x\cdot\bar \nu)^2)\ge 3$ which is impossible letting $\nu\to\bar \nu$. 
			\item If $0\in\Sigma_k(u_\infty)$ for some $k\in\{0,1,\ldots,n-2\}$, denote by $p_\infty=\tfrac12 {} x\cdot A_\infty x$ the blow-up of $u_\infty$ at $0$. 
			Since $A_\infty\ge0 $ has rank at least two, there is $\bar \eta>0$ and two orthonormal vectors $\theta_1,$ $ \theta_2$ such that
			\[
			A_\infty \ge \bar\eta( \theta_1\otimes \theta_1+\theta_2\otimes \theta_2)
			\]
			from which follows that, if $\xi^\perp$ is an hyperplane with $\xi\in\S^{n-1}$, then
			\[
			\lap_{\xi^\perp} p_\infty ={} \tr_{\xi^\perp} A_\infty \ge \bar\eta\tr_{\xi^\perp}( \theta_1\otimes \theta_1+\theta_2\otimes \theta_2) =2\bar \eta - \bar\eta\big ((\xi\cdot \theta_1)^2 +(\xi\cdot\theta_2)^2\big)\ge\bar\eta. 
			\]
			Given $A_\circ>0$ from \Cref{lem:inherit convexity} choose $\bar r$ such that
			\[
			|\bar r^{-2}u_\infty(\bar r\cdot)-p_\infty|<\frac{\bar\eta}{A_\circ}\quad \text{in }B_1.
			\]
			By uniform convergence, the same holds for $\bar r^{-2}\tilde u_k(\bar r\cdot)$ with $k$ sufficiently large. We can now apply \Cref{lem:inherit convexity} with $\bar r^{-2}\tilde u_k(\bar r\cdot)$ in place of $u$, $f_k(\bar r r_k\cdot)$ in place of $f$, $p_\infty$ in place of $U$, ${}$ in place of $F$ and along the hyperplane $\xi_k^\perp$ from \eqref{eq:contr prop low freq}. Assumptions $(i)$ to $(iv)$ are satisfied for $\eps:=A_\circ\bar\eta$ and $k$ large, so we get
			\[
			\lap_{\xi^\perp_k} \tilde u_k \ge 0\text{ in } B_{\rho_1\bar r},
			\]
			which contradicts \eqref{eq:contr prop low freq} as $x_k\in B_{\rho_1\bar r}$ eventually.
			\qedhere
		\end{enumerate}    
	\end{proof}
	
	We now adapt \Cref{lem:low frequency close poly}.
	\begin{lemma}\label{lem:low frequency close poly f}
		Given $\delta\in(0,\tfrac\alpha5)$ there is $M=M(n,\alpha,\delta)>0$ large and $\eta=\eta(n,\alpha,\delta)>0$ small such that the following holds. 
		
		If $u$ solves \eqref{eq:obstacle} in $B_1\subset\R^n$ for some normalized $f$, $0\in\de\{u>0\}$ and for some $r\in(0,r_\circ)$ we have that
		\begin{enumerate}[label=\roman*)]
			\item $\phi^{\tau}(r,u-p_\nu)+ \frac{C_\circ}{\alpha_\circ}r^{\alpha_\circ}\le 3-\delta$ for all $\nu\in \S^{n-1}$;
			\item $\eps:=r^{-2}\min_{\nu\in\S^{n-1}}H(r,u-p_{\nu})^{1/2} \le\eta$;
		\end{enumerate}
		Then for all unit vectors $\xi$ it holds
		\begin{equation}
			\lap_{\xi^\perp} u(x)\ge - M \eps\, \frac{|x|}{r}\quad \quad \forall x\in B_{r/2}.\label{eq:kl}
		\end{equation}
		The universal constants $r_\circ$, $\alpha_\circ$ and $C_\circ$ are the ones of \Cref{corollary:frequency formula}.
	\end{lemma}
	\begin{proof}
		We denote by $C$ and $c$ a large and a small universal constant, we also systematically split coordinates as $y=(y',y_n)\in\R^{n-1}\times\R$. 
		
		We argue by contradiction. Take $u_k$ solving \eqref{eq:obstacle} in $B_1$ for some normalized $f_k$ with $0\in\de\{u_k>0\}$, such that $(i)$ holds. 
		Up to a rotation, we may assume that the minimal direction in (ii) is $\bar \nu_k=e_n$ and set
		\[
		p_k(x):=p_{\bar \nu_k}(x),\quad \lap p_k =f_{k,1},\quad \tilde p_k:=r_k^{-2} p_k(r_k\cdot),\quad \tilde u_k(x):= r_k^{-2} u_k(r_k x)\quad \text{ for }x\in B_{1/r_\circ},
		\]
		so that
		\[
		\eps_k:= \|\tilde u_k- \tilde p_k \|_{L^2(\de B_1)}\to 0.
		\]
		Notice also that, since the $f_k$ have been normalized, we must have
		\[
		\tilde p_k(x) \to \tfrac12 x_n^2,
		\]
		in any norm.
		The contradiction assumption is that, for some $z_k\in B_{1/2}\setminus\{0\}$ and unit vectors $\xi_k$, \begin{equation}\label{eq:contradiction low frequency close poly RHS}
			\lap_{\xi_k^\perp}\tilde u_k(z_k)<-k \eps_k |z_k|.
		\end{equation}
		
		By \Cref{rmk:not-toosmall} and the choice of $\tau=3+\alpha/2$ we have
		\[
		0<r_k^{\tau-2}\le C\eps_k\to 0\quad\text{ and }\quad \phi(1,\tilde u_k-\tilde p_k)\le 3-\delta.
		\]
		In particular we deduce that $r_k^{1+\alpha}\le C\eps_k^{({1+\alpha})/({\tau-2})}=C\eps_k\eps_k^{{\alpha}/({\alpha+2})}$, so errors of the form  $O(r_k^{1+\alpha})$ are $o(\eps_k)$. This fact will be used repeatedly in the computations below.
		
		We observe now that 
		\begin{equation}
			\label{eq:x_ktozero}z_k\to 0,
		\end{equation}
		indeed by the semiconvexity bound of \Cref{lem:semiconvexity} applied to $\tilde u_k$ and $U:=\tilde p_k$ we find for any unit vector $\xi$: 
		\begin{equation}\begin{split}\label{eq:semiconvexity freq<3}
				-D_{\xi\xi}\tilde u_k(z_k)&\le C\big(\|\tilde u_k -\tilde p_k\|_{L^2(\de B_1)} +\|f_k(r_k\cdot)-f_{k,1}(r_k\cdot)\|_{C^{1/4}(B_1)} +\sup_{B_1}(D^2\tilde p_k)_-\big)\\&\le  C\eps_k+Cr_k^{1+\alpha}\le C\eps_k.
			\end{split}
		\end{equation}
		Where we used that, by \eqref{eq:p_k almost convex} and scaling:
		\[
		D^2\tilde p_k\ge -Cr_k^2\quad \text{ and }\quad \|f_k(r_k\cdot)-f_{k,1}(r_k\cdot)\|_{C^{1/4}(B_1)}\le Cr_k^{1+\alpha}.
		\]
		Summing over an orthonormal basis of $\xi_k^\perp$ and using \eqref{eq:contradiction low frequency close poly RHS} we get
		\[
		k \eps_k |z_k|\le -\lap_{\xi_k^\perp}\tilde u_k \le C\eps_k,
		\]
		which indeed implies that $z_k\to0$.
		
		Now we go to the next-order, set 
		\[
		v_k:=r_k^{-2}(u_k-p_k)(r_k\cdot),\qquad \tilde v_k:=\frac{v_k}{ \eps_k}.
		\]
		Thanks to \Cref{lem:lip estimate rhs}, we have
		\[
		\|\nabla  v_k\|_{L^\infty(B_{1/2})}\le C(\eps_k + r_k^{1+\alpha})\le C \eps_k,
		\]
		thus $\tilde v_k$ converges locally uniformly in $B_1$ to some locally Lipschitz function $v$. Furthermore, by \Cref{rmk:not-toosmall}, we have an upper bound on its Almgren's frequency:
		\[
		\phi(1,\tilde v_k)=\phi(r_k,\tilde u_k -\tilde p_k)\le \phi^\tau(r_k,\tilde u_k -\tilde p_k)\le 3-\delta.
		\]
		Letting $k\to\infty$ and using the compactness of the trace $H(B_1)\to L^2(\de B_1)$ we find:
		\[
		\|v\|_{L^2(\de B_1)}=1,\qquad \phi(1,v)\le 3-\delta,\qquad v(0)=0.
		\]
		
		As in the case $f\equiv 1$, the limit $v$ solves the thin obstacle problem \eqref{eq:thin obstacle}. Indeed $v$ is superharmonic:
		\begin{equation}\label{eq:delta of u - p RHS}
			\lap v_k =\eps_k^{-1}{f_{k,2}(r_k\cdot )}\ind_{\{\tilde u_k>0\}} -\eps_k^{-1}{f_{k,1
				}(r_k\cdot )}\ind_{\{\tilde u_k=0\}} \le C r_k^{1+\alpha}/ \eps_k\to 0.
		\end{equation}
		Moreover $v_k = u_k\ge0$ on $\{x_n=0\}$, so
		\[
		v\ge0 \text{ on }\{x_n=0\}
		.
		\]
		In addition, as $\tilde u_k\to \tfrac12x_n^2$ locally uniformly in $B_1$, given any $t>0$ we have $\tilde u_k>0$ on $\{|x_n|>t\}$ for $k$ large, so
		\begin{equation}\label{eq:consequence of convergence to thin space RHS}
			\Delta v =0\quad\text{on }\{x_n\neq0\}.
		\end{equation}
		Finally, if $v(z',0)>0$ we also have $v>0$ in $B_{\bar\delta}(z',0)$ for some $t>0$, so we deduce by uniform convergence that $u_k(x)> 0$ in $B_{t}(z',0)$ for $k$ large enough. It follows from this and \eqref{eq:delta of u - p RHS} that $\lap v_k=0$ in $B_{t}(z',0)$. Together with \eqref{eq:consequence of convergence to thin space RHS} this yields
		\[
		\Delta v = 0\quad\text{in }B_1\setminus\{v=0,\  x_n=0\}.
		\]
		According to \Cref{prop:thin obstacle}, there are 3 possible cases: we show that each one leads to a contradiction with \eqref{eq:contradiction low frequency close poly RHS}.
		
		More precisely, we show that in each case there is a small scale $\bar r>0$ ---depending only on $v$ (but uniform in $k$, crucially)---such that the rescaled solution 
		\[
		\bar u_k:=\bar r^{-2} \tilde u_k(\bar r \cdot) 
		\]
		in fact satisfies the linear decay estimate
		\[
		\lap_{\xi^\perp}\bar u_k(y) \ge -C\bar r^{-2}\eps_k |y|\quad\text{ for all } y\in  B_{1/2}\text{ and for all unit vectors }\xi.
		\]
		This will give the sought contradiction: since $z_k\to 0$ we can take $y:=z_k/\bar r$ and $\xi=\xi_k$ to find:
		\[
		\lap_{\xi_k^\perp}\tilde u_k(z_k) \ge -C \bar r^{-3}\eps_k |z_k|,
		\]
		and thus comparing it with \eqref{eq:contradiction low frequency close poly RHS}, get  
		\[
		-k \eps_k |z_k| >-C \bar r^{-3}\eps_k |z_k|. 
		\]
		which is impossible for $k$ large enough.
		
		Heuristically, the scale $\bar r$ is chosen so that the function $v(x)$ agrees with its blow-up at $x=0$.
		
		\textbf{Case (a) of \Cref{prop:thin obstacle}.}  In this case there is a constant $\bar K\ge1$ such that 
		\[
		|v(x) - a_+(x_n)_+ + a_- (x_n)_-|\le \tfrac{1}{4n}\bar K|x|^2\quad\text{for all }x\in B_1,
		\]
		where $a_++a_-\le0$ and not both $a_+,a_-$ are 0.
		We show that there is $\bar r>0$ independent from $k$ and some $\beta_k\to0$ such that, for $k$ large,
		\begin{equation}\label{eq:contact set splits RHS}
			B_{\bar r}\cap\{x_n=\beta_k\}\subset \{\tilde u_k=0\}.
		\end{equation}
		This concludes because we can split the free boundary into two regular regions in $B_{\overline r}$. Rigorously, setting
		\[
		\hat u_k^\pm := \bar r^{-2}\tilde u_k(\bar r\cdot)\ind_{\{ x_n \gtrless \beta_k/\bar r\}},  
		\]
		\eqref{eq:contact set splits RHS} implies that the functions $\hat u_k^\pm$ solve the obstacle problem in $B_1$ with $\hat u_k^\pm(0)=0$ and
		\[
		|\hat u_k^\pm-\tfrac12(x_2)^2_\pm|< \bar r^{-2}(\eps_k+r_k)\quad \text{in }B_1.
		\]
		So if $k$ is large enough we can apply \Cref{lemma:regular case} to $\hat u_k^\pm$ and find
		\[
		-C \bar r^{-2}\eps_k|y|  \le D^2  \hat u_k( y) \quad \forall y \in B_{1/2},
		\]
		which leads a to a contradiction with $y:=z_k/\bar r$ (remember that $z_k\to0$) and \eqref{eq:contradiction low frequency close poly RHS}.
		
		Now we prove claim \eqref{eq:contact set splits RHS} using suitable barriers.
		
		First we observe that
		\[
		a_+\le 0\quad\text{ and } \quad a_-\le 0.
		\]
		In fact we prove that $a_+>0$ implies $\tilde u_k(0)>0$ for $k$ large (the argument for $a_-$ is identical). For $\sigma>0$ to be chosen uniform in $k$, set $Q_\sigma := B_\sigma'\times [-\sigma,\sigma]$ and define the barrier
		\[
		U_k(x) := {(1+8\bar K\eps_k)}\tilde p_k(x+\eps_k^2 e_n)-4\bar K{\eps_k} \tilde q_k,
		\]
		where 
		\begin{equation}
			\label{eq:def q tilde}\tilde q_k(x):= \frac{1}{n-1}\sum_{i=1}^{n-1}r_k^{-2}p_{e_i}(r_k x),\qquad \lap \tilde q_k = f_{k,1}(r_k\cdot).
		\end{equation}
		In the following computations it is helpful to remember that, as the $f_k$ are normalized,
		\[
		\tilde p_k(x) =\tfrac{1}{2}(1+O(r_k))x_n^2,\qquad \tilde q_k(x) =\tfrac{1}{2(n-1)}(1+O(r_k))|x'|^2.
		\]
		We claim that, if $a_+>0,$ we can choose $\sigma=\sigma(a_\pm,\bar K )>0$ small such that, for $k$ large enough,
		\begin{equation}
			{\eps_k^{-1}}(U_k-\tilde p_k) <v\quad\text{on }\de Q_\sigma\cap\{x_2\ge-\eps_k^2\}.\label{eq:a+>0}
		\end{equation}
		Indeed, if $x_n=\sigma$ and $|x'|<\sigma$, we have
		\begin{align*}
			{\eps_k^{-1}}(U_k-\tilde p_k)(x)&\le\frac{\tilde p_k(x+\eps_k^2 e_n)-\tilde p_k(x)}{\eps_k}+8\bar K\tilde p_k(x+\eps_k^2 e_n)-4\bar K\tilde q_k(x)\\
			&\le C\eps_k + C\bar K\sigma^2 < \tfrac12a_+\sigma\le a_+\sigma -\bar K\sigma^2< v(x).
		\end{align*}
		If $|x'|=\sigma,$ $-\eps_k^2\le x_n\le \sigma$ and $\sigma$ is small we have 
		\begin{align*}
			{\eps_k^{-1}}(U_k-\tilde p_k )&\le\frac{\tilde p_k(x+\eps_k^2 e_n)-\tilde p_k(x)}{\eps_k}+8\bar K\tilde p_k(x+\eps_k^2 e_n)- 4\bar K\tilde q_k(x)\\
			&\le C\eps_k+ 5\bar K(x_n+\eps_k^2)^2-\tfrac{2}{n}\bar K \sigma^2 \le 6\bar K x_n^2 -\tfrac{1}{n}\bar K \sigma^2 < \tfrac12 a_+ (x_n)_+-\tfrac{1}{4n}\bar K\sigma^2 < v(x).
		\end{align*}
		Thus \eqref{eq:a+>0} is proved.
		We emphasize that $\sigma$ depends on $\bar K$ and $\pm a$, so it is by no means universal, nevertheless it is uniform in $k$.
		As $v_k\to v$ uniformly on $Q_\sigma$, this implies $U_k\le u_k$ on $\de Q_\sigma\cap\{x_2\ge-\eps_k^2\}$ for all $k$ large enough.
		On the other hand, $U_k\le0\le u_k$ on $\{x_n=-\eps_k^2\}$. Thus, for $k$ large enough
		\[
		U_k\le u_k\quad\text{on }\de(Q_\sigma\cap\{x_n\ge-\eps_k^2\}).
		\]
		Now $U_k$ is a subsolution:
		\begin{align*}
			\lap U_k &= (1+8\bar K\eps_k)f_1(r_k(x+\eps_k^2e_n))-4\bar K \eps_kf_1(r_k x)\\
			&\ge f_1(r_k x)+4\bar K \eps_kf_1(r_k x)-C(1+\bar K)r_k \eps_k^2\\
			&\ge f(r_kx) -Cr_k^{1+\alpha} +4\bar K\eps_k-C (1+\bar K)r_k\eps_k^2\\
			&\ge \lap \tilde u_k +{\eps_k}[4\bar K-Cr_k\eps_k-C(1+\bar K)r_k^{1+\alpha}/\eps_k]>\lap \tilde u_k,
		\end{align*}
		so by comparison $U_k\le \tilde u_k$ on $Q_\sigma\cap\{x_2\ge-\eps_k^2\}$; in particular $0<U_k(0)\le \tilde u_k(0)$, which is a contradiction.
		
		We proved that $a_+\le0$. A symmetric argument gives $a_-\le0$. Since they cannot both vanish we may assume $a_+<0$. We prove similarly \eqref{eq:contact set splits RHS} with $$\beta_k := a_+\eps_k/2.$$ 
		For $\bar r>0$ to be chosen we set $Q_{\bar r}:=B_{\bar r}'\times [-\bar r,\bar r]$ and for $\xi'\in B_{\bar r}'\times\{0\}$ define the upper barrier
		\[
		U_{k,\xi'}(x) := {(1-8\bar K\eps_k)}\, \tilde p_k(x+\beta_ke_n )+ 4\bar K{\eps_k}\tilde q_k(x-\xi'),\qquad U_{k,\xi'}(\xi',-\beta_k)=0.
		\]
		For all $x\in \de Q_{\bar r}$ we have
		\[\begin{split}
			{\eps_k^{-1}}(U_{k,\xi'}-\tilde p_k)(x) &=\tfrac1{\eps_k}(\tilde p_k(x+\beta_ke_n )-\tilde p_k(x)) - 8\bar K\tilde p_k(x+\beta_ke_n )+4\bar K \tilde q_k(x-\xi')\\
			&\ge \tfrac12a_+\de_n\tilde p_k(x) - C\eps_k - 5\bar K (x_n+\beta_k)^2 +\tfrac{2}{n}\bar K |x'-\xi'|^2\\
			&\ge\tfrac{1}{2}a_+x_n - 5\bar K x_n^2 +\tfrac{2}{n}\bar K |x'-\xi'|^2 -Cr_k -C\eps_k \\
			&\ge v(x) + \bar Kx_n^2 + \tfrac{1}{n}\bar K |x'-\xi'|^2 -Cr_k -C\eps_k >v(x).
		\end{split}\]
		Where we used that, since $a_-\le0,$ for all $|x_n|\le\bar r\le \tfrac{1}{100} \bar K/|a_+|$ it holds: 
		\[
		a_-(x_n)_-+a_+(x_n)_+ \le\tfrac12 a_+ x_n -10\bar K x_n^2.
		\]
		As $v_k\to v$ uniformly in $Q_{2\bar r}$, we have $\tilde u_k< U_{k,\xi'}$ in $\de Q_{\bar r}$ for $k$ large.
		Now $U_{k,\xi'}$ is a supersolution:
		\begin{align*}
			\lap U_{k,\xi'}(x)&= (1-8\bar K \eps_k)f_{k,1}(r_k(x+\beta_ke_n))+{4\bar K\eps_k}{} f_{k,1}(r_k(x-\xi'))\\
			&\le (1-8\bar K\eps_k)f_{k,1}(r_k x) +4\bar K{\eps_k}{}f_{k,1}(r_kx)+C(|a_+|+|\xi'|)r_k\eps_k\\
			&\le f_k(r_kx)-{4\bar K\eps_k}{}f_{k,1}(r_kx)+C(|a_+|+|\xi|)r_k\eps_k+Cr_k^{1+\alpha}\\
			&\le f_k(r_kx)-\eps_k[4\bar K - C'r_k -C r_k^{1+\alpha}/\eps_k]<f_k(r_kx).
		\end{align*}
		Since $U_{k,\xi'}\ge 0$,  by comparison
		\Cref{lem:comparison principle} we find $\tilde u_k \le U_{k,\xi'}$ in $Q_{\bar r}$. Since $U_{k,\xi'}(\xi,-\beta_k)=0$ and $\xi\in B_{\bar r}'$ was arbitrary, \eqref{eq:contact set splits RHS} follows.

		\textbf{Case (b) of \Cref{prop:thin obstacle}.} In this case, up to a further rotation that fixes $e_n$ we have
		\[
		|v(x) - a\psi(x) |\le \omega(|x|) |x|^{3/2}\quad \forall x\in B_1,
		\]
		where $a>0$, $\omega$ is some modulus of continuity and $\psi(x):=\psi_{3/2}(x_1,x_n)$ where $\psi_{3/2}$ is given by \eqref{eq:1.5 homogeneous solution}.
		
		We rescale by a small parameter $\bar r$ independent on $k$ and denote:
		\[
		\bar u_k:=\bar r^{-2} \tilde u_k(\bar r\cdot),\quad \bar p_k:= \bar r^{-2} \tilde p_k(\bar r\cdot),\quad \bar v_k:=\bar r^{-2} \tilde v_k(\bar r\cdot),
		\]
		and $\eta_k:=\|\tilde v_k- v\|_{L^\infty(B_{1/2})}\to 0.$
		
		It is then clear that we can write, defining $\omega_k(\bar r):= \omega(\bar r)+\eta_k \bar r^{-3/2},$ the identity 
		\begin{equation}
			\bar u_k(x) = \bar p_k(x) + \bar r ^{-1/2}\eps_k \big[\psi(x) + \omega_k(\bar r) b_k(x)\big],\qquad |b_k(x)|\le 1\quad \forall x\in B_2.\label{eq: u close to psi}
		\end{equation}
		
		We first show that there is a modulus of continuity $\bar \sigma(\bar r)$, such that
		\begin{equation}\label{eq:contact set close paraboloids RHS}
			N_{\bar \sigma}:=B_1 \cap\{ x_1\le -\bar \sigma(\bar r),\ x_n=0\} \subset\{\bar u_k =0\},
		\end{equation}
		for all $k$ larger than $\bar k=\bar k(\bar\sigma)$. 
		
		By \eqref{eq:1.5 homogeneous solution}, for any $\sigma>0$ there is $c_\sigma>0$ such that $\psi(x)\le -c_\sigma|x_n|$ in $N_{\sigma/4}$. So for any $\xi\in N_{\sigma}$, on the sphere $\de B_{\sigma/2}(\xi)$ we have:
		\[
		\bar r^{1/2}\eps_k^{-1}\big(\bar u_k-\bar p_k\big)(x) \le -c_\sigma |x_n|+\omega_k(\bar r) \le -8x_n^2 + (\sigma/2)^2= -7x_n^2+|x'-\xi|^2.
		\]
		
		Thus, for $\xi\in N_\sigma$, we consider the barrier
		\[
		U_{k,\xi} := (1-8r^{-1/2}\eps_k) \bar p_k+4 r^{-1/2}\eps_k \bar q_k(x-\xi)\ge 0,\qquad \bar q_k:=\bar r ^{-2}\tilde q_k(\bar r \cdot),
		\]
		where $\tilde q_k$ is the same as in \eqref{eq:def q tilde}. Now
		\[\begin{split}
			{\bar r^{1/2} \eps_k^{-1}}(U_{k,\xi}-\bar p_k)(x) &\ge - 5 x_n^2+ 2|x'-\xi|^2>-7x_n^2+|x'-\xi|^2\ge  \bar r^{1/2}\eps_k^{-1}\big(\bar u_k-\bar p_k\big)(x).
		\end{split}\]
		We also have, with the same computation of case 1 with $\bar K=1$ and $\beta=0$:
		\[
		\lap U_{k,\xi} \le f_k(r_k\bar r\cdot),
		\]
		so \Cref{lem:comparison principle} yields $\bar u_k \le U_{k,\xi}$ in $B_{\sigma/2}(\xi)$, and \eqref{eq:contact set close paraboloids RHS} follows since $U_{k,\xi}(\xi)=0$.
		
		Now we claim that, for $k$ large, the function
		\[
		h(x):=\lap_{\xi_k^\perp}\bar u_k /\eps_k\quad \forall x\in\{\bar u_k>0\},
		\]
		satisfies the assumptions of \Cref{lem:caffarelli rhs holder} with $\bar u_k$, which gives $\lap_{\xi_k^\perp} \tilde u_k\ge 0$ in $B_{\rho_1\bar r}$ and contradicts \eqref{eq:contradiction low frequency close poly}, as explained above. 
		Let us check the assumptions. If $\{\xi_1,\ldots,\xi_{n-1}\}$ is an orthonormal basis of $\xi_k^\perp$, then
		\[
		\lap h_k(x) = \divergence F_k\quad\text{ with }F_k:= \frac{r_k}{\eps_k}\bar r\sum_{j=1}^{n-1} \xi_j \de_{\xi_j}f_k(r_k\bar r\cdot),
		\]
		so $\|F_k\|_{C^\alpha(B_1)}\le C\|f_k(r_k\cdot)\|_{C^{1,\alpha}(B_1)}/\eps_k \le 1$ for $k$ large.
		
		Furthermore by the weaker semiconvexity estimate of \eqref{eq:semiconvexity_weak},  $h_k\ge0 $ on $\de\{\bar u_k>0\}$.
		
		By rescaling \eqref{eq:semiconvexity freq<3} we get
		\[
		h_k(x) \ge -\eps_k^{-1}\sum_{j=1}^{n-1} D_{\xi_j\xi_j}\bar u_k(x) \ge -C_1.
		\]
		
		Finally we need to check that $h_k$ is very large in the region:
		\[
		\{\dist(\cdot,\{\bar u_k=0\})>\sigma_\circ\}\cap B_1,
		\]
		where $\sigma_\circ$ is the universal constant of \Cref{lem:caffarelli rhs holder}.
		We fix once and for all $\sigma:=\sigma_\circ/8$ and we may assume $\bar r$ so small that $\bar\sigma(\bar r)<\sigma/2$. Then by \eqref{eq:contact set close paraboloids RHS} we have
		\[
		\{\dist(\cdot,\{\bar u_k=0\})>\sigma_\circ\}\cap B_1 \subset \{\dist(\cdot,N_\sigma)>2\sigma\}\cap B_1 =:\Omega_\sigma,
		\]
		so it is enough to show that $h_k$ is large in the region $\Omega_\sigma$.
		
		Now for all $z\in \Omega_\sigma$ we have that $\psi(x)$ is harmonic in $B_{\sigma}(z)$, so we may use standard elliptic estimates to bound $D^2b_k$ (defined by \eqref{eq: u close to psi})
		\begin{align*}
			c(\sigma)\|D^2 b_k\|_{L^\infty(B_{\sigma/2}(z))}&\le \|\lap b_k\|_{C^\alpha(B_\sigma(z))} +\|b_k\|_{L^\infty(B_{\sigma}(z))}\\
			&\le \frac{\sqrt{\bar r}}{\omega_k(\bar r)}\|\lap \bar v_k\|_{C^\alpha(B_\sigma(z))}  +1 \le C \frac{\sqrt{\bar r}\, r_k^{1+\alpha}}{ \omega_k(\bar r)\, \eps_k}{}+1
		\end{align*}
		Since $z\in\Omega_\sigma$ was arbitrary, using \eqref{eq: u close to psi} we get for all $x\in \Omega_\sigma$:
		\[
		D^2\bar u(x) \ge D^2(\bar p_k +\frac{a\eps_k}{\sqrt{\bar r}}\psi) -C\frac{\eps_k}{\sqrt{\bar r}}\big(\omega_k(\bar r)+ \bar r ^{1/2} r_k^{1+\alpha}/\eps_k\big) \ge D^2(\tfrac12 x_n^2 +\frac{a\eps_k}{\sqrt{\bar r}}\psi) -\frac{\eps_k}{\sqrt{\bar r}}\gamma_k(\bar r) 
		\]
		where we used that $\bar p_k$ is almost convex \eqref{eq:p_k almost convex}, and set for brevity
		\[
		\gamma_k(\bar r):= C\big(\omega(\bar r) + \eta_k \bar r ^{-3/2} + \bar r^{1/2} r_k^{1+\alpha}\eps_k^{-1} +r_k^2\eps_k^{-1}\big),
		\]
		note that, first choosing $\bar r$ and then choosing $k$ large, one can arrange $\gamma_k(\bar r)$ to be arbitrarly small.
		
		Now we claim that for some universal $c_\circ>0$ we have
		\begin{equation}\label{eq:hessian psi elliptic}
			D^2_{\xi\xi}(\tfrac12 x_n^2 +\frac{a\eps_k}{\sqrt{\bar r}}\psi)(x) \ge c_\circ \frac{a\eps_k}{\sqrt{\bar r}}(\xi_1^2+\xi_n^2)\ge 0 \quad\forall x\in\Omega_\sigma,\quad \forall \xi\in\R^n.
		\end{equation}
		Since by dimension count $\xi_k^\perp$ must intersect the plane spanned by $\{e_1,e_n\}$ we may assume that $\xi_1$ lies in this plane and deduce
		\[
		h_k(x)= \eps_k^{-1}\sum_{j=1}^{n-1}D^2_{\xi_j\xi_j} \bar u_k(x) \ge \eps_k^{-1}D^2_{\xi_1\xi_1}\bar u_k(x)-\frac{\gamma_k(\bar r)}{\sqrt{\bar r}} \ge \frac{1}{\sqrt{\bar r}}(c_\circ -\gamma_k(\bar r)) \ge c_\circ \frac{a}{2\sqrt{\bar r}}
		\quad\forall x\in\Omega_\sigma.
		\]
		Choosing $\bar r\lesssim (c_\circ a/A_\circ)^2$ where $A_\circ$ is the constant of \Cref{lem:caffarelli rhs holder} we are done.
		
		It remains to verify that \eqref{eq:hessian psi elliptic} holds. Notice that the function $\tfrac12 x_n^2 +t\psi(x) $ depends only on $x_1$ and $x_n$ so denoting by $\rho:=(x_1^2+x_n^2)^{1/2}$ and $\cos\theta:=x_n/\rho$ we compute in the $\{e_1,e_n\}$ basis:
		\[
		D^2(\tfrac12 x_n^2 + t \psi_{3/2}(x_1,x_n)) = \begin{pmatrix}
			t \frac{3}{4} \rho^{-1/2} \cos(\theta/2) & t \frac{3}{4} \rho^{-1/2} \sin(\theta/2) \\
			t \frac{3}{4} \rho^{-1/2} \sin(\theta/2) & 1 - t \frac{3}{4} \rho^{-1/2} \cos(\theta/2)
		\end{pmatrix}.
		\]
		Checking determinant and trace, one sees that this matrix is positive definite for $t={a\eps_k}/{\sqrt{\bar r}}$ small, because, crucially, in the domain $\Omega_\sigma$ it holds:
		\[
		\sigma\le \rho\le 1\quad \text{ and }\quad  |\theta|<{\pi}{}-c(n)\sigma.
		\]
		The first condition ensures that the $(n,n)$ entry is $\ge 1/2$, while the second keeps the $\cos$ in the $(1,1)$ entry uniformly positive.
		
		\textbf{Case (c) of \Cref{prop:thin obstacle}.} In this case for some $(n-1)$ symmetric matrix $A\ge 0$ and some $b\in \R^{n-1}$ (not both zero) it holds
		\[
		|v(x) -\big( x'\cdot A x -\tr(A)x_n^2 + (b\cdot x')x_n \big)|\le \omega(|x|)|x|^2, \text{ for }x\in B_1,
		\]
		where $\omega$ is some modulus of continuity.
		
		We first show that we may assume $b=0$ by perturbing slightly the choice of $\bar \nu_k=e_n$. We consider the rotated Ansatz
		$$\nu_k := (e_n+ \eps_k b)/|e_n+ \eps_k b|,\quad \tilde P_k:=r_k^{-2}p_{\nu_k}(r_k\cdot),$$
		so that (in the $C^1(B_2)$ norm) we have
		\[
		{\eps_k^{-1}}(\tilde P_k-\tilde p_k) = \tfrac{1}2[(\nu_k+e_n)\cdot x][\tfrac1{\eps_k}(\nu_k-e_n)\cdot x)]+O(r_k+\eps_k)\to x_n (b\cdot x').
		\]
		Now we consider, rather than $\tilde v_k$, the function:
		\[
		w_k:= \tilde u_k -\tilde P_k,\qquad \tilde w_k:=\frac{w_k}{\eps_k},
		\]
		and remark that, by minimality of $\bar \nu_k$, and the triangular inequality:
		\[
		\eps_k \le \|w_k\|_{L^2(\de B_1)} \le \|v_k\|_{L^2(\de B_1)}+\|\tilde P_k-\tilde p_k\|_{L^2(\de B_1)}\le\eps_k+ C\eps_k\le C\eps_k.
		\]
		Thus we have $\|\tilde w_k\|_{L^2(\de B_1)}\in[1,C]$. 
		Hence we find
		\[
		\tilde w_k=\tilde v_k +\frac{\tilde p_k-\tilde P_k}{\eps_k}\to v-\bar x_n(b\cdot x')=:w,
		\]
		where the convergence is locally uniform. By construction:
		\[
		|w(x) - \big( x'\cdot A x -\tr(A)x_n^2\big)|\le   \omega(|x|)|x|^2\quad \forall x\in B_1.
		\]
		Now, using again assumption $(i)$ with $\nu=\nu_k$ and \Cref{rmk:not-toosmall} we get
		\[
		\phi(1,\tilde w_k)=\phi(1,\tilde u_k -\tilde P_k)=\phi(r_k, u_k - P_k)\le \phi^\tau(r_k, u_k - P_k) \le 3-\delta.
		\]
		Thus passing to the limit and using compactness of the trace we find
		\[
		\phi(1,w)\le 3-\delta,\qquad \|w\|_{L^2(\de B_1)}\in[1,C],\qquad w(0)=0,
		\]
		Thus $w$ is nonzero and consequently $A$ cannot be zero, otherwise $\phi(1,w)\ge 3$. So, up to a rotation that fixes $e_n$, we have
		\[
		A\ge  \alpha\, e_1\otimes e_1\text{ for some } \alpha >0. 
		\]
		We set
		\[
		U_k:=\tilde P_k + \eps_k(x'\cdot Ax' -\tr(A)x_n^2),\text{ for which }\lap U_k=f_{k,1}(r_k\cdot).
		\]
		
		Now we choose $\bar r>0$ small and $k$ large such that we can apply \Cref{lem:inherit convexity} to the functions
		\[
		\bar u_k:=\bar r^{-2}\tilde u_k(\bar r\cdot),\quad \bar U_k:=\bar r^{-2} U_k(\bar r\cdot).
		\]  
		Indeed we have for $x\in B_1$ that
		\[
		|(\bar u_k -\bar U_k)(x)|\le  o(\eps_k)+\eps_k|\bar r^{-2} \tilde w_k (\bar r\cdot)-(x'\cdot Ax' -\tr(A)x_n^2)|\le o(\eps_k)+\eps_k\omega(\bar r)\le \alpha \eps_k/(4A_\circ),
		\]
		and $|D^2(f_{k}(r_k\bar r\cdot)|\le Cr_k^2\le \alpha \eps_k/(4A_\circ)$.
		On the other hand using \eqref{eq:p_k almost convex} one has
		\begin{equation*}
			D^2\bar U_k \ge -Cr_k^2 +(\tfrac12-\alpha\eps_k) (e_n\otimes e_n)+\alpha\eps_k (e_1\otimes e_1)\quad \text{ in }B_1.
		\end{equation*}
		so using again that $\{e_1,e_n\}$ and $e_k^\perp$ share at least one line we get
		\begin{equation*}
			\lap_{e_k^\perp}\bar U_k \ge{\alpha}{ \eps_k} (1- Cr_k^2/\eps_k)\ge \alpha \eps_k /2.
		\end{equation*}
		So assumptions of \Cref{lem:inherit convexity} hold with $\eps=\alpha\eps_k/(4A_\circ)$ and we get that $\lap_{e_k^\perp}\bar u_k\ge0$ in $B_{\rho_1\bar r}$, contradicting \eqref{eq:contradiction low frequency close poly RHS} as in the other cases.
	\end{proof}

	\subsection{Extending Monneau monotonicity formula}
	In this Section we extend \Cref{prop:extended monneau} to all dimensions and non-constant $f$. 
	
	The main difference, compared to the 2D case, is that the epiperimetric inequality we prove is logarithmic. Crucially, it is still powerful enough to make the frequency gap integrable, so the final statement is the same. 
	
	Recall the definition of the truncated frequency gap \eqref{eq:truncated frequency gap}, the choice $\tau=3+\alpha/2$ and the adapted ansatz $p_\nu$ defined in \eqref{eq:def_pnu_RHS}. We are also implicitly assuming the universal bound \eqref{eq:universal upper bound}.
	\begin{proposition}\label{prop:monneau ND RHS}
		There are universal positive constants $r_\circ,$ $\delta_\circ,$ $ C_\circ$ such that the following holds: suppose that $u$ solves \eqref{eq:obstacle} with $0\in\de\{u>0\}$ and write $w=u-p_\nu$, where $\nu\in\mathbb S^{n-1}$.
		If
		\[
		E_3^\tau(s,w)\geq-\delta_\circ\quad\text{for some }s\in(0,r_\circ/2),
		\]
		then
		\[
		\|w_s\|_{L^2(\de B_1)}\leq C_\circ s^3.
		\]
	\end{proposition}
	We will write
	\[
	w_r:= (u-p_{\nu})(r\cdot),
	\]
	and $z_r$ will denote the 3-homogeneous extension of $w_r$ in $B_1$.
	
	\subsubsection{An epiperimetric inequality}
	The main ingredient is an epiperimetric inequality, in the spirit of \Cref{lem:apply projection epiperimetric}. In higher dimensions, however, we need to weaken it to a logarithmic epiperimetric inequality. The nontrivial proof is relatively short, and proceeds by contradiction taking advantage of a clever construction of Savin and Yu \cite{SavinYuinteger}.
	
	\begin{lemma}\label{lem:epi ND}
		There are universal positive constants $r_\circ,$ $L_\circ,$ $c_\circ$ and a dimensional constant\,\footnote{The proof gives $\gamma=\frac{n-1}{n+3}$.} $\gamma\in(0,1)$ such that the following holds. Suppose $r<r_\circ$ and
		\[
		0\le E_{2.5}^\tau(r,w),\quad E_{10}^\tau(2r,w)\le0.
		\]
		Let $z_r$ be the $3$-homogeneous extension of $w_r$ and let $\zeta_r$ be the solution of the thin obstacle problem in $B_1$ with boundary value $z_r=w_r$ on $\de B_1$. Then
		\[
		E_3^\tau(r,\zeta_r)\le E_3^\tau(r,z_r)(1+c_\circ|E_3^\tau(r,z_r)|^\gamma)+r^\alpha,
		\]
		and
		\[
		\zeta_r
		\ge
		-L_\circ
		\bigl(H(r,w)+r^{6+2\alpha}\bigr)^{1/2}|x_n|
		\qquad\text{in }B_1.
		\]
	\end{lemma}
	We recall some notation. We define the linear space
	\[
	\mathcal P_3:=\big\{p(x)\text{ 3-homogeneous},\quad \supp\Delta p\subset\{x_n=0\},\quad p\equiv0\text{ on }\{x_n=0\}\big\},
	\]
	and the orthogonal projection
	\[
	\Pi_3\colon L^2(\de B_1)\to \mathcal P_3.
	\]
	It is useful to consider the decomposition of $\mathcal P_3$ in functions that are respectively even and odd in the $x_n$ variable, that is $\mathcal P_3 = \mathcal P_3^\odd \oplus \mathcal P_3^\even$, and their corresponding orthogonal projections $\Pi_3^\odd,\Pi_3^\even$, and the cone of 3-homogeneous solutions of the thin obstacle problem
	\[
	\mathcal P_3^+ :=\{p\in \mathcal P_3, \ \lap p\le0\}.
	\]
	Recall the Weiss energy notation
	\[
	W_3(g) :=\int_{B_1}|\nabla g|^2 - 3\int_{\de B_1} g^2,\qquad W_3(g;h) := \int_{B_1}\nabla g\cdot\nabla h-3\int_{\de B_1}gh.
	\]
	
	The equivalent of \Cref{lem:epiperimetric projection} with a right-hand side is the following.
	\begin{lemma}\label{lem:projection epi RHS}
		Given $M>0$ there is $\eps=\eps(n,M)>0$ such that the following holds. Let $\sigma\ge0$ and let $c\in H^1(B_1)$ be a $3$-homogeneous function with $c(\cdot,0)\ge0$ and
		\[
		\Delta\Pi_3c \le C(\|c-\Pi_3c\|_{L^2(\de B_1)}+\sigma)
		\delta_{\{x_n=0\}},
		\]
		then there is $v\in c+H^1_0(B_1)$ with $v(\cdot,0)\ge0$ such that
		\[
		(1-\eps)W_3(v)\leq W_3(c)+\sigma^2.
		\]
	\end{lemma}
	\begin{proof}
		By contradiction, suppose there are $3$-homogeneous functions $c_k\in H^1(B_1)$ and $\sigma_k\ge0$ such that $c_k\ge0$ on $\{x_n=0\}$ and
		\[
		\Delta\Pi_3c_k \le M(\|\Pi_3c_k-c_k\|_{L^2(\de B_1)}+\sigma_k)\delta_0(x_n)
		\]
		but, for any $v_k\in c_k+H^1_0(B_1)$ with $v_k(\cdot,0)\ge0$,
		\begin{equation}\label{eq:almost minimizer ND}
			W_3(c_k)\le (1-\tfrac1k)W_3(v_k)-\sigma_k^2 .
		\end{equation}
		
		Setting
		\[
		q_k:=\Pi_3c_k,\qquad
		\delta_k^2:=
		\|c_k-q_k\|_{L^2(\de B_1)}^2+\sigma_k^2,
		\]
		the only difference between the proof of \Cref{lem:epiperimetric projection} is to show $\sigma_k = o(\delta_k)$. This can be proven since $W_3(c_k)\le-k\sigma_k^2$ (which follows by taking $v_k = c_k$ in \eqref{eq:almost minimizer ND}), so in this case \eqref{eq:k} reads as
		\[
		k\sigma_k^2+W_3(c_k-q_k)\le2\int_{B_1}\Delta q_k(c_k-q_k),
		\]
		and the same bound on the right hand side leads to 
		\[
		(k-CM)\sigma_k^2+\tfrac12\|\nabla(c_k-q_k)\|_{L^2(B_1)}^2 \le CM\|c_k-q_k\|_{L^2(\de B_1)}^2,
		\]
		as we wanted, which implies $\sigma_k = o(\delta_k)$ and a non-zero limit up to a subsequence $w_\infty$.
		The optimality of $w_\infty$ follows with the same proof, since
		\[
		W_3(c_k)\le (1-\tfrac1k)W_3(v_k)-\sigma_k^2\le (1-\tfrac1k)W_3(v_k),
		\]
		while the final test function can be chosen as $|x_n|(|x'|^2-\tfrac{n-1}3x_n^2)$.
	\end{proof}
	
	However, in higher dimensions we need to understand a more delicate case, covered by the following result.
	
	\begin{lemma}\label{lem:log epi}
		There are dimensional $\eps_0,$ $c_0>0$ and $\gamma\in(0,1)$ with the following property. 
		
		Let $p$ be an even cubic solution of the thin obstacle problem with $\|p\|_{L^2(\de B_1)}=1$, and let $c\in H^1(B_1)$ be 3-homogeneous and satisfy
		\[
		c\ge0\text{ on }\{x_n=0\}\quad\text{and}\quad \|c-p\|_{H^1(B_1)}<\eps_0.
		\]
		Then there is $\hat\zeta \in c +H^1_0(B_1)$ such that $\hat\zeta(\cdot,0)\ge0$ and
		\[
		W_3(\hat\zeta)\le (1+c_0|W_3(c)|^\gamma)W_3(c).
		\]
	\end{lemma}
	As a preliminary, we recall a useful construction of Savin and Yu \cite{SavinYuinteger}. We extend the definition of $W_3$ to functions defined only on $\mathbb S^{n-1}$ by setting
	\[
	W_3(g(\theta)) := W_3(r^3g(\theta)) = \frac{1}{n+4}\int_{\mathbb S^{n-1}}|\nabla_\theta g|^2 - \lambda_3 g^2;\quad \lambda_3:=3(n+1).
	\]
	We define $W_3(g;h)$ in a similar way.
	
	Fix $\sigma>0$ and write $L_\sigma^\pm := \{x_n = \pm\sigma\}\cap\mathbb S^{n-1}$, while we denote $\Omega_\sigma := \{|x_n|<\sigma\}\cap\mathbb S^{n-1}$. If $\sigma>0$ is dimensionally small the energy $W_3$ is coercive on $H^1_0(\Omega_\sigma)$. We will denote the equator by $E:= \{x_n=0\}\cap\mathbb S^{n-1}$.
	
	Given $q\in\mathcal P_3$ we denote by $\bar q$ the solution of the thin obstacle problem inside $\Omega_\sigma$ with operator $\Delta_{\mathbb S^{n-1}}+\lambda_3$, equal to $q$ outside $\Omega_\sigma$. This is possible thanks to the choice of $\sigma$. We will identify it with its 3-homogeneous extension. Note that
	\[
	(\Delta_{\mathbb S^{n-1}}+\lambda_3)\bar q = \psi^+_{q}\delta_{L^+_\sigma} + \psi^-_q\delta_{L^-_\sigma} + g\delta_E,
	\]
	where $g\le0$ and $\bar q \cdot g \equiv0$ on $E$, and set
	\[
	\kappa_q := \int_{L^+_\sigma} \psi_q^+.
	\]
	Given $q\in\mathcal P_3$, we associate the unique matrix $A_q\in \mathrm{Sym}(n-1)$ such that $q = h + |x_n| (\tfrac{\tr A}{3}x_n^2-x'\cdot Ax)$, where $h$ is harmonic.
	\begin{lemma}\label{lem:auxiliary properties}
		Given $q$ and the construction above we have the following:
		\begin{enumerate}[label = (\roman*)]
			\item $\bar q\ge q$, $\psi_q^+\equiv\psi_q^-\ge0$ and $c\kappa_q\le\psi_q^\pm\le C\kappa_q$ for some positive dimensional constants $c$, $C$;
			\item Let $p_0\in\mathcal P_3$ be an even solution of \eqref{eq:thin obstacle} with $\|p_0\|_{L^2(\de B_1)}=1$. For each compact subset $K$ of $\{\partial_np_0(\cdot,0)<0\}$ there is $\delta_K>0$ such that if $\|p_0-\bar p\|_{L^2(\de B_1)}<\delta_K$, for some $p\in\mathcal P_3$, then $\bar p\equiv 0$ on $K$;
			\item For $\|q\|_{L^2(\de B_1)}\le1$ we have
			\begin{equation}\label{eq:crucial weiss bounds}
				\|(A_q)_-\|_\infty^{\alpha-1}\le C\kappa_q,\quad -W_3(\bar q)\le C\kappa_q^\beta,
			\end{equation}
			where $C(n)>0$, $\alpha=\frac{n+3}{2}>2$ and $\beta=\frac{n+3}{n+1}\in(1,2)$ 
		\end{enumerate}
	\end{lemma}
	\begin{proof}
		Notice that for any $q$ cubic harmonic polynomial vanishing on $\{x_n=0\}$ we have $\overline{p+q} = \bar p+q$. So, (i) is \cite[Lemma 2.2]{SavinYuinteger}, while (ii) is \cite[Lemma B.3]{FROS24}.
		
		To show (iii), suppose that $\delta:= \|A_-\|_\infty= -A_{11}>0$. Then $e_1$ is a critical point for the quadratic form $x'\cdot A x'$, so
		\[
		x'\cdot A x'\le -c\delta\quad\text{in }B_{c\sqrt\delta}(e_1)\cap E.
		\]
		Fix now a cutoff function $0\le\varphi\le1$, $\varphi\in C^\infty_c(B_1)$, $\varphi\equiv 1$ on $B_{1/2}$. Given $t>0$ sufficiently small, consider
		\[
		\tilde\varphi(\theta) :=  \delta^{3/2}\varphi((\theta-e_1)/c\sqrt\delta).
		\]
		Note that
		\[
		W_3(q+t\tilde\varphi) = 2tW_3(q;\tilde\varphi) + t^2 W_3(\tilde\varphi).
		\]
		As $\tilde\varphi$ is supported inside $\Omega_\sigma$, we have
		\[
		W_3(q;\tilde\varphi) = -\int_{E}\partial_nq^\even \tilde\varphi = \int_{E\cap B_{c\sqrt\delta}(e_1)}x'\cdot Ax' \tilde\varphi\sim -\delta^{(n-2)/2+1+3/2},
		\]
		on the other hand $|\nabla\tilde\varphi|\lesssim \delta^{-1/2+3/2}$, so
		\[
		W_3(\tilde\varphi) \lesssim \delta^{(n-1)/2+2}.
		\]
		We deduce
		\begin{equation}\label{eq:bound weiss modified poly}
			W_3(\bar q) \le W_3(q+t\tilde\varphi) = 2tW_3(q;\tilde\varphi) + t^2 W_3(\tilde\varphi) \le -c\delta^{(n+3)/2},
		\end{equation}
		provided $t>0$ is universally small.
		
		Let now $p$ be the projection of $q$ onto $\mathcal P_3^+$, so that $\|p-q\|_{L^2(\de B_1)} \sim \delta$. Note that, since $\bar q$ solves the thin obstacle problem and $q\equiv0$ on $E$, we have $W_3(\bar q;\bar q-q)=0$. Then
		\[
		-W_3(\bar q) = W_3(\bar q-q) = -W_3(q;\bar q-q) = - W_3(p;\bar q-q) - W_3(q-p;\bar q-q).
		\]
		As $\Delta p\le0$ and $\bar q-q\ge0$ we have $-W_3(p;\bar q-q) \le0$, while using $q-p\equiv0$ on $E$ yields
		\[
		-W_3(q-p;\bar q-q) = \int_{L^\pm_\sigma} \psi_q^\pm (q-p) \le C\delta \kappa_q.
		\]
		Combining this with \eqref{eq:bound weiss modified poly} gives \eqref{eq:crucial weiss bounds}.
	\end{proof}
	
	\begin{proof}[Proof of \Cref{lem:log epi}]
		Choose $\gamma\in(0,1)$ such that $\beta(1+\gamma)=2$, where $\beta$ is given by \eqref{eq:crucial weiss bounds}. Suppose by contradiction that there are even solutions of the thin obstacle problem $p_k\in\mathcal P_3$ with $\|p_k\|_{L^2(\de B_1)}=1$ and $c_k$ 3-homogeneous with
		\[
		c_k(\cdot,0)\ge0\quad\text{and}\quad\|c_k-p_k\|_{H^1(B_1)}<1/k
		\]
		but
		\begin{equation}\label{eq:contradiction log epi}
			W_3(c_k)(1+\tfrac1k |W_3(c_k)|^\gamma)\le W_3(v)\quad\text{for each }v\in c_k + H^1_0(B_1), v(\cdot,0)\ge0.
		\end{equation}
		Note in particular that
		\[
		W_3(c_k)\le0.
		\]
		
		Let
		\[
		q_k := \mathrm{arg}\min \{\max\{\|c_k-\bar q\|_{H^1(B_1)}, \kappa_q\},\, q\in\mathcal P_3\},\quad \delta_k:=\max\{\|c_k-\bar q_k\|_{H^1(B_1)},\kappa_{q_k}\}\to0,
		\]
		and consider the functions
		\[
		w_k := \frac{c_k-\bar q_k}{\delta_k}.
		\]
		As $w_k$ is bounded in $H^1(B_1)$, up to a subsequence
		\[
		w_k\weak w_\infty\quad\text{weakly in }H^1(B_1).
		\]
		
		We prove that $w_\infty\in H^1(\de B_1)$ satisfies
		\begin{equation}\label{eq:ELE log epi}
			w_\infty \ge0\text{ on }\{x_n=0\}\qquad\text{and}\qquad (\Delta_{\mathbb S^{n-1}} +\lambda_3)w_\infty = - \mu_\infty\quad\text{on }\mathbb S^{n-1}\setminus E.
		\end{equation}
		where $\mu_\infty\ge0$ is supported on $L_\sigma^\pm$ and $\|\mu_\infty\| = 2\lim_k \frac{\kappa_{q_k}}{\delta_k}$.
		
		For the first property, it suffices to note that $p_k,q_k\to p_\infty$, which also implies $\bar q_k\to p_\infty\in \mathcal P_3^+$ with $\|p_\infty\|=1$. So, \Cref{lem:auxiliary properties} implies, for any compact subset $K$ of $\{\partial_np_\infty<0\}$,
		\[
		\bar q_k=0\quad \text{on }K\qquad\text{for }k\text{ large enough}.
		\]
		Thus, inside $\{\partial_n p_\infty<0\}$ we have $w_\infty\ge0$. Since $\mathcal H^{n-2}(\{\partial_np_\infty=0\})=0$, the claim follows.
		
		For the second property, note that
		\[
		W_3(c_k) = W_3(\bar q_k) + 2W_3(\bar q_k;c_k-\bar q_k) + W_3(c_k-\bar q_k).
		\]
		Since $|W_3(c_k-\bar q_k)| \le C\|c_k-\bar q_k\|_{H^1(B_1)}^2$ and
		\[
		W_3(\bar q_k;c_k-\bar q_k) = -c_n\int_{B_1} \Delta \bar q_k (c_k-\bar q_k) \ge -C\kappa_{q_k}\|c_k-\bar q_k\|_{H^1(B_1)},
		\]
		using \Cref{lem:auxiliary properties} we find, for $k$ large enough (recall $\beta<2$)
		\[
		|W_3(c_k)| = -W_3(c_k) \le -W_3(\bar q_k) + C\delta_k^2\le C\delta_k^{\beta}.
		\]
		As $\gamma\in(0,1)$ was chosen so that $\beta (1+\gamma) = 2$ (possible as $\beta>1$), \eqref{eq:contradiction log epi} gives
		\begin{equation}\label{eq:almost minimality log epi}
			W_3(c_k) \le W_3(v) + o(\delta_k^2)\quad\text{for each }v\in c_k+H^1_0(B_1)\cap\{v(\cdot,0)\ge0\}.
		\end{equation}
		Fix $\phi\in H^1_0(B_1\cap\{x_n>0\})$ and test \eqref{eq:almost minimality log epi} with $c_k + \delta_k\phi = \bar q_k + \delta_k(w_k + \phi)$ to find
		\[
		W_3(w_k) + \tfrac2{\delta_k}W_3(w_k;\bar q_k) \le W_3(w_k+\phi) + \tfrac2{\delta_k}W_3(w_k+\phi;\bar q_k) + o(1).
		\]
		Integrating by parts and recalling that $\bar q_k$ is 3-homogeneous and $\phi = 0$ on $E$ leads to
		\[
		W_3(w_k) - \int_{L_\sigma^\pm}\tfrac{2}{\delta_k}\psi_k^\pm w_k \le W_3(w_k+\phi) -\int_{L^\pm_\sigma}\tfrac{2}{\delta_k}\psi^\pm_k(w_k+\phi) + o(1)
		\]
		Since $\|\psi_k\|_{L^\infty}\le C\delta_k$ by \Cref{lem:auxiliary properties}, up to a subsequence we have $\psi_k^\pm/\delta_k\weak^* \psi_\infty^\pm$ in $L^\infty(L_\sigma^\pm)$. Letting $k\to+\infty$ and defining $\mu_\infty = \psi^+_\infty \delta_{L^+_\sigma} + \psi_\infty^-\delta_{L^-_\sigma}$ one finds
		\[
		\int_{B_1} |\nabla w_\infty|^2 - 2\int_{L_\sigma^\pm}  w_\infty \,d\mu_\infty \le \int_{B_1}|\nabla (w_\infty+\phi)|^2- 2\int_{L_\sigma^\pm} (w_\infty+\phi)\,d\mu_\infty.
		\]
		for any $\phi\in H^1_0(B_1\setminus\{x_n=0\})$, where we used $w_k(\cdot,\pm\sigma)\to w_\infty(\cdot,\pm\sigma)$ strongly in $L^1(L^\pm_\sigma)$ by compactness of the trace. As $w_\infty$ is $3$-homogeneous, \eqref{eq:ELE log epi} follows.
		
		Now testing \eqref{eq:ELE log epi} with $P(x',x_n) = |x_n|(|x'|^2-\tfrac{n-1}3x_n^2)$ yields
		\[
		\int_{L^\pm_\sigma}P d\mu_\infty = -\int_{\{x_n=0\}} w_\infty\partial_n P.
		\]
		Recalling $\mu_\infty\ge0,w_\infty\ge0$, the two terms have different signs. Since $\partial_nP(\cdot,0)>0$ and $P(\cdot,\sigma)>0$, this forces
		\[
		w_\infty\equiv0\quad\text{on }\{x_n=0\}\qquad\text{and}\qquad \mu_\infty \equiv0,
		\]
		so in particular
		\begin{equation}\label{eq:ELE limit}
			\kappa_{q_k} = o(\delta_k)\quad \text{and}\quad w_\infty\in \mathcal P_3.
		\end{equation}
		
		We now prove strong $H^1$ convergence along a subsequence satisfying \eqref{eq:ELE limit}.
		To do so, we test \eqref{eq:almost minimality log epi} with $v_k:=\bar q_k+ \delta_k h_k$ finding, after some simplifications,
		\begin{equation}
			W_3(w_k) \le W_3(h_k)+\tfrac{2}{\delta_k} W_3(\bar q_k; h_k-w_k)+o(1),\label{eq:dunno}
		\end{equation}
		provided $h_k\in H^1_0(B_1)+ w_k$ satisfies $h_k\ge -\bar q_k/\delta_k$ on $\{x_n=0\}$. Since $w_\infty$ vanishes on the thin space, we may choose $h_k$ as
		\[
		h_k:= \xi w_\infty + (1-\xi)w_k,\qquad h_k(\cdot,0)\ge (1-\xi)w_k\ge -(1-\xi)\bar q_k/\delta_k \ge - \bar q_k/\delta_k.
		\]
		where $\xi:= (\min\{1,2(1-|x|)\})_+$ is a radial bump function. With this choice we can bound the mixed term using \eqref{eq:ELE limit}:
		\begin{align*}
			W_3(\bar q_k; h_k-w_k)=\int_{B_1} (w_k-w_\infty)\lap \bar q_k \xi\, =\frac{1}{\delta_k}\int_{\{x_n=0\}}c_k \lap\bar q_k \xi + O(\kappa_{q_k})\le o(\delta_k).
		\end{align*}
		Now $w_k$ and $h_k$ agree on $\de B_1$, so putting things together \eqref{eq:dunno} becomes
		\[
		D(w_k) \le D(\xi w_\infty + (1-\xi)w_k) +o(1),
		\]
		which splitting $w_k=\xi w_k+(1-\xi)w_k$ rearranges into:
		\[
		D(\xi w_k) + 2D(\xi w_k; (1-\xi)w_k) \le D(\xi w_\infty) + 2D(\xi w_\infty; (1-\xi)w_k) + o(1).
		\]
		Taking the upper limit and using weak convergence:
		\[
		\limsup_k \int_{B_1} |\nabla (\xi w_k)|^2 + 2\nabla(\xi w_k)\cdot\nabla((1-\xi)w_k) \le \int_{B_1} |\nabla(\xi w_\infty)|^2 +2\nabla(\xi w_\infty)\cdot \nabla ((1-\xi) w_\infty).
		\]
		Note that weak $H^1$ convergence implies
		\[
		\lim_k\int_{B_1}2(1-\xi)w_k\nabla w_k\cdot\nabla\xi -w_k^2|\nabla\xi|^2 = \int_{B_1}2(1-\xi)w_\infty\nabla w_\infty\cdot\nabla\xi -w_\infty^2|\nabla\xi|^2,
		\]
		so subtracting this gives
		\[
		\limsup_k\int_{B_1} (1-(1-\xi)^2)|\nabla w_k|^2 \le \int_{B_1}(1-(1-\xi)^2) |\nabla w_\infty|^2.
		\]
		As $w_k$ are 3-homogeneous and $\xi$ is radial, we can rewrite it as
		\[
		\limsup_k\int_{\de B_1} |\nabla_\theta w_k|^2 + 9w_k^2 \le \int_{\de B_1} |\nabla_\theta w_\infty|^2 + 9w_\infty^2.
		\]
		As the second term converges due to strong $L^2$ convergence, we conclude.
		
		We now conclude. Set
		\[
		q_k':=q_k+\delta_k w_\infty
		\]
		and note that, since $w_\infty\not\equiv0$ thanks to strong $H^1$ convergence and \eqref{eq:ELE limit}, we contradict the minimality of $\bar q_k$ by proving
		\begin{equation}\label{eq:replacement linearization}
			\|\bar q_k'-\bar q_k-\delta_k w_\infty\|_{H^1(B_1)}=o(\delta_k),
			\qquad
			\kappa_{q_k'}=o(\delta_k).
		\end{equation}
		To show this we argue similarly to \cite[Lemma 3.5]{SavinYuinteger}. Using the free boundary condition for $\bar q_k,\bar q_k'$ and $w_\infty(\cdot,0)\equiv0$ we compute
		\begin{equation}\label{eq:exploit coercivity}\begin{split}
				W_3(\bar q_k'-\bar q_k-\delta_k w_\infty)
				&= W_3(\bar q_k'-\bar q_k-\delta_k w_\infty;\bar q_k') - W_3(\bar q_k'-\bar q_k-\delta_k w_\infty;\bar q_k)\\
				&\qquad-\delta_k W_3(w_\infty;\bar q_k'-\bar q_k-\delta_k w_\infty)\\
				&\le -\delta_k W_3(w_\infty;\bar q_k'-\bar q_k-\delta_k w_\infty).
		\end{split}\end{equation}
		Since $\bar q_k'-\bar q_k-\delta_k w_\infty\in H^1_0(\Omega_\sigma)$ and $W_3$ is coercive there,
		setting
		\[
		\eta_k:=\frac{\bar q_k'-\bar q_k-\delta_k w_\infty}{\delta_k}
		\]
		we deduce $\|\eta_k\|_{H^1}\le C$, so up to a subsequence $\eta_k\weak \eta_\infty$ weakly in $H^1(\Omega_\sigma)$.
		By construction $\eta_\infty$ solves $(\Delta_{\mathbb S^{n-1}}+\lambda_3)\eta_\infty=0$ on $\Omega_\sigma\setminus E$; on the other hand,
		as $\bar q_k,\bar q_k'\to p_\infty$,
		by \Cref{lem:auxiliary properties} $\eta_\infty(\cdot,0)=0$. Therefore, $W_3(\eta_\infty) = -\int_{\Omega_\sigma} \eta_\infty (\Delta_{\mathbb S^{n-1}}+\lambda_3)\eta_\infty = 0$, and coercivity of $W_3$ implies
		$\eta_\infty\equiv0$. By \eqref{eq:exploit coercivity},
		\[
		c\|\eta_k\|_{H^1}^2\le W_3(\eta_k)\le-W_3(w_\infty,\eta_k)=o(1),
		\]
		which is the first relation in \eqref{eq:replacement linearization}.
		
		Finally, recall that $\eta_k$ solves $(\Delta_{\mathbb S^{n-1}}+\lambda_3)\eta_k = (\psi_{q_k'}^\pm-\psi^\pm_{q_k})\delta_{\pm\sigma}(x_n)$ on $\mathbb S^{n-1}\setminus E$. Testing with a cutoff function $\varphi$ satisfying $\varphi(\cdot,\pm\sigma)\equiv1$ and supported away from $E$ gives
		\[
		|\kappa_{q_k'}-\kappa_{q_k}| = \left|\int_{L^\pm_\sigma}(\psi^\pm_{q_k'}-\psi^\pm_{q_k})\varphi\right|
		\le C\|\bar q_k'-\bar q_k-\delta_k w_\infty\|_{H^1(B_1)}
		=o(\delta_k).
		\]
		Since $\kappa_{q_k}=o(\delta_k)$, this proves
		\eqref{eq:replacement linearization} and the lemma follows.
	\end{proof}
	
	We conclude the auxiliary Lemmas by showing that we may always apply, with universal parameters, either \Cref{lem:log epi} or \Cref{lem:projection epi RHS}.
	\begin{lemma}\label{lem:projection dichotomy}
		For every $\bar\eps>0$ there are universal constants $\bar M =\bar M(n,\bar\eps)$, $\bar r=\bar r(n,\bar\eps)$ such that for any $r<\bar r$, if $u$ solves \eqref{eq:obstacle} with $f$ normalized and $0\in\de\{u>0\}$, then writing $w:= u-p_\nu$:
		\begin{enumerate}[label=(\alph*)]
			\item either $\Delta\Pi_3 w_r\le \bar M (\|w_r-\Pi_3w_r\|_{L^2(\de B_1)}+r^{3+\alpha}) \delta_0(x_n)$ in $B_1$,
			\item or there are $p_3^{\mathrm{even}}\in\mathcal P_3^{\mathrm{even}}$ solving the thin obstacle problem and $p_3^{\mathrm{odd}}\in\mathcal P_3^{\mathrm{odd}}$ such that
			\[
			\|w_r-p_3^{\mathrm{odd}}-p_3^{\mathrm{even}}\|_{L^2(\de B_1)}\le\bar\eps\|p_3^{\mathrm{even}}\|_{L^2(\de B_1)}.
			\]
		\end{enumerate}
		If, in addition $W_3(z_r)\leq0$, alternative (b) implies
		\begin{equation}\label{eq:projection control H^1}
			\|z_r-p_3^{\odd}-p_3^{\even}\|_{H^1(B_1)} \le C\bar\eps\|p_3^{\even}\|_{L^2(\de B_1)},
		\end{equation}
		where $z_r$ is the $3$-homogeneous extension of $w_r|_{\de B_1}$.
	\end{lemma}
	\begin{proof}
		Note that, assuming $W_3(z_r)\le0$, one finds
		\[
		W_3(z_r-p_3^\odd-p_3^\even) = W_3(z_r) + W_3(p_3^\odd+p_3^\even) -2W_3(z_r;p_3^\even+p_3^\odd) \le W_3(z_r)\le0,
		\]
		where we used $W_3(q)=0$ for any $q\in\mathcal P_3$ and
		\[
		-W_3(z_r;p_3^\even+p_3^\odd) = \int_{\{x_n=0\}\cap B_1} z_r\Delta p_3^\even \le0 
		\]
		since $z_r\ge0$ on $\{x_n=0\}$ while $\Delta p_3^\even\le0$. So, \eqref{eq:projection control H^1} follows as $z_r$ is 3-homogeneous.
		
		We now prove the first dichotomy by contradiction. Suppose there are $r_k\to0$ and $u_k$ solving \eqref{eq:obstacle} with normalized $f_k$ but such that, writing
		\[
		w_k(x):=(u_k-p_{\nu,k})(r_kx),
		\]
		we have
		\begin{equation}\label{eq:super big projection RHS}
			\|(A_{\Pi_3 w_k})_-\|_\infty\ge k(\|\Pi_3 w_k -w_k\|_{L^2(\de B_1)}+r_k^{3+\alpha})
		\end{equation}
		and, for any $p_3^\odd\in\mathcal P_3^\odd$ and any $p_3^\even\in\mathcal P_3^\even$ solving the thin obstacle problem,
		\begin{equation}\label{eq:distant from signorini cone}
			\|w_k-p_3^\even-p_3^\odd\|_{L^2(\de B_1)}>\bar\eps\|p_3^\even\|_{L^2(\de B_1)}.
		\end{equation}
		Writing
		\[
		q_k^\odd := \Pi_3^\odd w_k,\quad\delta_k:=\|w_k-q_k^\odd\|_{L^2(\de B_1)}+r_k^{3+\alpha}
		\]
		and setting
		\[
		v_k := \frac{w_k-q_k}{\delta_k}
		\]
		then we can write uniquely
		\[
		v_k = q_k^\even + v_k^\perp
		\]
		where $q_k^\even\in \mathcal P_3^\even$ and
		\[
		v_k^\perp \perp_{L^2(\de B_1)}\mathcal P_3 \quad \text{and}\quad \|v_k^\perp\|_{L^2(\de B_1)}^2+\|q_k^\even\|_{L^2(\de B_1)}^2=1-r_k^{6+2\alpha}/\delta_k^2.
		\]
		The first inequality in \eqref{eq:super big projection RHS} gives
		\[
		\|(A_{q_k^\even})_-\|_\infty\ge k(\|v_k^\perp\|_{L^2(\de B_1)}+r_k^{3+\alpha}/\delta_k)
		\]
		Since $\|(A_{q_k^\even})_-\|_\infty\le C\|q_k^\even\|_{L^2(\de B_1)}\le C$, this implies
		\[
		r_k^{3+\alpha} = o(\delta_k)\quad\text{and}\quad v_k^\perp\to0\quad\text{in }L^2(\de B_1),
		\]
		so in particular $\|q_k^\even\|_{L^2(\de B_1)} = 1 - o(1)$.
		On the other hand,  \eqref{eq:distant from signorini cone} applied with $p_3^\odd = q_k^\odd$ and $p_3^\even = \delta_k q_3^\even$ gives, for any 3-homogeneous even solution of the thin obstacle problem $q_3^\even$,
		\[
		\|v_k-  q_3^\even\|_{L^2(\de B_1)}>\bar\eps\|q^\even_3\|_{L^2(\de B_1)}.
		\]
		Notice that $q_k^\even$ are given and $q_3^\even$ are arbitrary. 
		As $\|q_k^\even\|_{L^2(\de B_1)} \to 1$, this implies for $k$ large
		\begin{equation}
			\dist(q_k^\even,\mathcal P_3^+)\ge\tfrac12 \bar\eps>0.\label{eq:no top}
		\end{equation}
		Let $q_\infty^\even \not\in\mathcal P_3^{\even,+}$ be an accumulation point of $q_k^\even$, and denote $Q_k^\even$ (resp., $Q_\infty^\even$) the solution of the Thin Obstacle problem in $B_1$ with boundary data $q_k^\even$ (resp., $q_\infty^\even$) and let $h_k$ be the harmonic replacement of $v_k^\perp$. By elliptic regularization,
		\[
		v_k^\perp \to 0\text{ in $L^2(\de B_1)$ implies that }h_k\to 0\text{ uniformly  in }\overline{B_{1/2}}.
		\]
		We claim that $Q_\infty^\even(0)>0$: indeed, note that $Q_\infty^\even-q_\infty^\even\ge0$ is harmonic on $B_1\cap\{x_n>0\}$ and $Q_\infty^\even\not\equiv q_\infty^\even$, as $q_\infty^\even$ does not solve the Thin Obstacle problem by \eqref{eq:no top}. Thus, if $Q_\infty^\even(0) = q_\infty^\even(0)=0$ then Hopf lemma would imply $\partial_n(Q_\infty^\even-q_\infty^\even)(0)>0$. As $\partial_nq_\infty^\even(0) = 0$ while $\partial_nQ_\infty^\even(0)\le0$ by the free boundary condition, this is a contradiction.
		
		Now $v_k$ is nonnegative on $\{x_n=0\}$ and (almost) superharmonic since \eqref{eq:pde u-p_nu} gives
		\[
		\Delta v_k = \frac{1}{\delta_k}\Delta w_k \le O(r_k^{3+\alpha}/\delta_k) \le o(1),
		\]
		so by comparison we find 
		$$0 = v_k(0) \ge Q_k(0)+h_k(0) - o(1)\ge Q_\infty^\even(0)-o(1)>0$$
		for $k$ large, a contradiction.
	\end{proof}

	We can finally prove our epiperimetric inequality:

	\begin{proof}[Proof of \Cref{lem:epi ND}]
		Thanks to the Lipschitz estimates of \eqref{eq:lip estimates} and the doubling estimate of \Cref{lem:doubling rhs}, the lower bound follows as in the proof of \Cref{lem:apply projection epiperimetric}.
		
		To prove the upper bound on the energy, by minimality we have $E_3^\tau(r,\zeta_r)\le E_3^\tau(r,\tilde\zeta_r)$ for any $\tilde\zeta_r = z_r$ on $\de B_1$ and $\tilde\zeta_r(\cdot,0)\ge0$. We may assume $E_3^\tau(r,z_r)<0$, as otherwise we may take $\tilde\zeta_r = z_r$. We will use a suitable $\tilde\zeta_r$, given either by \Cref{lem:log epi} or \Cref{lem:projection epi RHS}. Thanks to the identity
		\begin{equation}\label{eq:W E identity ND RHS}
			W_3(g) = (H(r,w)+r^{2\tau})(E^\tau_3(r,g)-\tfrac{C_\circ}{\alpha_\circ}r^{\alpha_\circ})-(\tau-3)r^{2\tau},
		\end{equation}
		in particular $W_3(z_r)<0$ (as $\tau>3$).
		
		Let $\eps_0$ be given by \Cref{lem:log epi}, let
		$M_0:=\bar M(n,\eps_0)$ be given by \Cref{lem:projection dichotomy}, and let $\eps_\circ=\eps_\circ(n, M_0)$ be given by \Cref{lem:projection epi RHS} with $M=M_0$. 
		
		Then by \Cref{lem:projection dichotomy} we have two cases.
		
		In the first case we apply \Cref{lem:projection epi RHS} to $z_r$ with $\sigma=r^{3+\alpha}$ to find $\tilde\zeta_r$ satisfying
		\[
		(1-\eps_\circ)W_3(\tilde\zeta_r)
		\le W_3(z_r) +r^{6+2\alpha}.
		\]
		Using \eqref{eq:W E identity ND RHS} together with $H(r,w)+r^{2\tau}\ge r^{6+\alpha}$ gives
		\[
		(1-\eps_\circ)E_3^\tau(r,\tilde\zeta_r)
		\le E^\tau_3(r,z_r) +r^{\alpha}.
		\]
		Since $-3\le E_3^\tau(r,z_r)\le0$ we have $\eps_\circ\ge c|E_3^\tau(r,z_r)|^{1+\gamma}$. Using that, by minimality, $E_3^\tau(r,\zeta_r)\le E_3^\tau(r,z_r)$, we get to
		\[
		E_3^\tau(r,\zeta_r)\le (1+\eps_\circ)E_3^\tau(r,z_r) + r^\alpha \le E_3^\tau(r,z_r)(1+c_\circ|E_3^\tau(r,z_r)|^{\gamma})+r^{\alpha}.
		\]
		
		In the second case \Cref{lem:projection dichotomy} gives us $p_3^\even$ and $p^\odd_3$. We apply \Cref{lem:log epi} to 
		$$c:=(z_r-p_3^{\odd})/({\|p_3^{\even}\|_{L^2(\de B_1)}})\in H^1(\de B_1)$$ 
		to find $\hat\zeta_r$ such that, setting $\tilde\zeta_r:= p_3^\odd+\|p_3^\even\|_{L^2(\de B_1)}\hat\zeta_r$, and using \eqref{eq:weiss on P3+something} repeatedly, the log-epiperimetric inequality holds:
		\[
		W_3(\tilde\zeta_r)
		\le W_3(z_r)(1+c_0
		\|p_3^{\even}\|_{L^2(\de B_1)}^{-2\gamma}
		|W_3(z_r)|^{\gamma}).
		\]
		Note that $\|p_3^\even\|^2_{L^2(\de B_1)}\le C(H(r,w)+r^{2\tau})$, thanks to \eqref{eq:projection control H^1}. Since \eqref{eq:W E identity ND RHS} yields $W_3(z_r)\le (H(r,w)+r^{2\tau})E_3^\tau(r,z_r)\le0$, this implies
		\[
		1+c_0\|p_3^\even\|_{L^2(\de B_1)}^{-2\gamma}|W_3(z_r)|^\gamma \ge 1+c_\circ|E_3^\tau(r,z_r)|^\gamma.
		\]
		Thus, using this and \eqref{eq:W E identity ND RHS} together with $E_3^\tau(r,z_r)\le0$ gives
		\[
		E_3^\tau(r,\tilde\zeta_r)\le E_3^\tau(r,z_r)(1+c_0\|p_3^{\even}\|_{L^2(\de B_1)}^{-2\gamma}
		|W_3(z_r)|^{\gamma})\le E_3^\tau(r,z_r)(1+c_\circ|E_3^\tau(r,z_r)|^\gamma),
		\]
		as we wanted.
	\end{proof}

	\subsection{Proof of \Cref{prop:monneau ND RHS}}
	
	Without loss of generality we can assume $\nu= e_n$.  If $r_\circ$ is given by \Cref{corollary:frequency formula}, thanks to the universal bounds \eqref{eq:universal upper bound} it suffices to consider $s\in(0,r_\circ)$, otherwise the statement holds with a universal constant.
	
	Consider the radius
	\[
	R:=\inf\left(\left\{r\in(0,r_\circ):
	\phi^\tau(r,w)+\frac{C_\circ}{\alpha_\circ}r^{\alpha_\circ}
	\ge3\right\}\cup\{r_\circ\}\right).
	\]
	Then \Cref{lem:monneau rhs} gives, for any $r>R$,
	\[
	\frac d{dr}\log\left[r^{-6}\bigl(H(r,w)+r^{2\tau}\bigr)\right]
	\ge-Cr^{\beta-1}.
	\]
	As the right-hand side is uniformly integrable, this implies
	\[
	\|w_r\|_{L^2(\de B_1)}\le C_\circ r^3
	\qquad\text{for any }r>R.
	\]
	Thus, it suffices to prove that
	\begin{equation}\label{eq:goal monneau ND}
		\bigl(H(s,w)+s^{2\tau}\bigr)^{1/2}
		\le C(s/R)^3
		\bigl(H(R,w)+R^{2\tau}\bigr)^{1/2}
	\end{equation}
	whenever $s<R/2$ and $E_{3-\delta_\circ}^\tau(s,w)\ge0$.
	
	In this case \Cref{lem:monneau rhs} implies
	\begin{equation}\label{eq:apply monneau ND}
		\|w_r\|_{L^2(\de B_1)}\le Cr^{3-\delta_\circ}
		\qquad\text{for any }r\in(s,R),
	\end{equation}
	while the Lipschitz and doubling estimates yield
	\begin{equation}\label{eq:apply lip ND}
		B_1\cap\{u_r=0\}
		\subset\left\{|x_n|\le Cr^{-2}
		\bigl(\|w_r\|_{L^2(\de B_1)}+r^{3+\alpha}\bigr)\right\}.
	\end{equation}
	
	Set $E(r):=E_3^\tau(r,w)$.  Then
	\[
	-\delta_\circ\le E(r)\le0\qquad\text{for }r\in(s,R),
	\]
	so \Cref{lem:epi ND} applies and implies that if $\zeta_r$ solves \eqref{eq:thin obstacle} in $B_1$ with boundary datum $z_r$ on $\de B_1$, then
	\begin{equation}\label{eq:properties zeta_r}
		\zeta_r\ge -L_\circ(H(r,w)+r^{6+2\alpha})^{1/2}|x_n|\quad\text{and}\quad E_3^\tau(r,\zeta_r)\le E_3^\tau(r,z_r)(1+c_\circ|E_3^\tau(r,z_r)|^\gamma)+r^\alpha.
	\end{equation}
	
	We now show that for some universal $C,\beta>0$
	\begin{equation}\label{eq:replacement comparison ND}
		E_3^\tau(r,\zeta_r)\ge E(r)-Cr^\beta.
	\end{equation}
	Indeed, using $\zeta_r=w_r$ on $\de B_1$, it suffices to estimate
	\[
	E_3^\tau(r,\zeta_r)-E_3^\tau(r,w)\ge\frac{2}{H(r,w)+r^{2\tau}}\int_{B_1}(\zeta_r-w_r)\Delta w_r.
	\]
	Using \eqref{eq:pde u-p_nu},
	\[
	\int_{B_1}(\zeta_r-w_r)\Delta w_r\ge \int_{B_1\cap\{u_r=0\}}\zeta_r r^2f_r-C\int_{B_1}r^{3+\alpha}|\zeta_r-w_r|.
	\]
	The first term is bounded as in the proof of
	\eqref{eq:goal 1 extend monneau}, using \eqref{eq:apply monneau ND}, \eqref{eq:apply lip ND}, and the lower bound of $\zeta_r$ in \eqref{eq:properties zeta_r}.
	To bound the second term, note first that $w_r$ is a competitor for $\zeta_r$ and $E(r)\le0$, therefore
	\[
	\int_{B_1}|\nabla(\zeta_r-w_r)|^2\le4\int_{B_1}|\nabla w_r|^2\le12(H(r,w)+r^{2\tau}).
	\]
	Thus Poincar\'e's inequality gives
	\[
	-C\int_{B_1}r^{3+\alpha}|\zeta_r-w_r|
	\ge-Cr^{3+\alpha}\bigl(H(r,w)+r^{2\tau}\bigr)^{1/2},
	\]
	and \eqref{eq:replacement comparison ND} follows recalling that
	$H(r,w)+r^{2\tau}\ge r^{2\tau}=r^{6+\alpha}$.
	
	We now combine \cref{eq:properties zeta_r,eq:replacement comparison ND} and \Cref{lem:derivative frequency gap RHS} to deduce
	\begin{equation}\label{eq:ODE frequency gap ND}
		E'(r)\ge\frac c r(-E(r))^{1+\gamma}-Cr^{\sigma-1}\qquad\text{for }r\in(s,R/2).
	\end{equation}
	If $E_3^\tau(r,z_r)\ge0$, \Cref{lem:derivative frequency gap RHS} already gives, for $\delta_\circ$ small enough,
	\[
	E'(r)\ge\frac{n+4+2E(r)} r(-E(r))-Cr^{\beta-1}\ge \frac{c}{r}(-E(r))^{1+\gamma}-Cr^{\beta-1}.
	\]
	On the other hand, when $E_3^\tau(r,z_r)\le0$, \eqref{eq:properties zeta_r} gives
	\[
	E_3^\tau(r,z_r)\ge E_3^\tau(r,\zeta_r)+c_\circ|E_3^\tau(r,z_r)|^{1+\gamma} - r^\alpha.
	\]
	Since $-\delta_\circ\le E_3^\tau(r,\zeta_r)\le E_3^\tau(r,z_r)\le0$, up to changing $c_\circ$ this implies
	\begin{align}
		E_3^\tau(r,z_r)\ge E_3^\tau(r,\zeta_r)+c_\circ|E_3^\tau(r,\zeta_r)|^{1+\gamma}- r^\alpha.\label{eq:swap z  zeta}
	\end{align}
	Indeed, by minimality of $\zeta_r$ we have $0\ge E_3^\tau(r,z_r)\ge E_3^\tau(r,\zeta_r)$. If $E_3^\tau(r,z_r)$ and $E_3^\tau(r,\zeta_r)$ are comparable (say $|E_3^\tau(r,z_r)|\ge \tfrac12|E_3^\tau(r,\zeta_r)|$), then also the remainders are comparable, that is 
	\[ c_\circ|E_3^\tau(r,z_r)|^{1+\gamma}\ge c_\circ 2^{-{1+\gamma}}|E_3^\tau(r,\zeta_r)|^{1+\gamma}.\]
	In the other case we have that $|E_3^\tau(r,z_r)|\le \tfrac12|E_3^\tau(r,\zeta_r)|$, so we bound directly:
	\begin{align*}
		E_3^\tau(r,z_r)\ge-\tfrac12 |E_3^\tau(r,\zeta_r)| = E_3^\tau(r,\zeta_r) +\tfrac12|E_3^\tau(r,\zeta_r)|\ge E_3^\tau(r,\zeta_r) + \tfrac{3^{-\gamma}}2|E_3^\tau(r,\zeta_r)|^{1+\gamma},
	\end{align*}
	where we used $0\ge E_3^\tau(r,\zeta_r)\ge-3$. Thus \eqref{eq:swap z  zeta} is proved.
	
	Since $w_r$ is a competitor for $\zeta_r$, we also have
	$E_3^\tau(r,\zeta_r)\le E(r)\le0$, so that $|E_3^\tau(r,\zeta_r)|^{1+\gamma} \ge (-E(r))^{1+\gamma}$. So, using \eqref{eq:replacement comparison ND} we obtain
	\[
	E_3^\tau(r,z_r)\ge E(r)+c(-E(r))^{1+\gamma}-Cr^\sigma.
	\]
	Inserting this into \Cref{lem:derivative frequency gap RHS}, and choosing $\delta_\circ$ small so that the term $2E(r)^2$ is absorbed by $|E|^{1+\gamma}$, leads to \eqref{eq:ODE frequency gap ND}.
	
	Setting
	\[
	\widetilde E(r):=E(r)+Cr^\sigma/\sigma,
	\]
	whenever $\tilde E(r)<0$ then \eqref{eq:ODE frequency gap ND} gives $\widetilde E'(r) \ge\frac cr(-\widetilde E(r))^{1+\gamma}$, which can be written as
	\[
	\frac{(\widetilde E(r))_-}{r} \le-\frac1{c(1-\gamma)}\left[(\widetilde E(r))_-^{1-\gamma}\right]'.
	\]
	To deduce \eqref{eq:goal monneau ND}, thanks to \Cref{lem:monneau rhs} we have
	\[
	\log\Big[\frac{s^{-6}(H(s,w)+s^{2\tau})}{R^{-6}(H(R,w)+R^{2\tau})}\Big]\le C+2\int_s^{R}\frac{(E(r))_-}{r}\,dr.
	\]
	In the regime $E(r)\le0$ we have
	\[
	\int_s^{R}\frac{(E(r))_-}{r}\,dr\le \int_s^R \Big(\frac{(\widetilde E(r))_-}{r} + \tfrac C\sigma r^{\sigma-1}\Big)\,dr\le C,
	\]
	where we used $\gamma<1$ and $\widetilde E(r)\ge -\delta_\circ$, and \eqref{eq:goal monneau ND} follows.

	\section{Proof in the General Case}\label{sec:theorem rhs}
	We first prove the linear $(n-1)$-semiconvexity decay of \Cref{theorem:linear decay semiconvexity}, from which we will deduce the curvature bounds of \Cref{theorem:viscosity bound}.
	\begin{proof}[Proof of \Cref{theorem:linear decay semiconvexity}]
		Let $\delta_\circ$ be given by \Cref{prop:monneau ND RHS} and let $r_1$ be a universal radius smaller than $\rho(\delta_\circ)$ (given by \Cref{prop:low frequency f}) and $r_\circ$, given by \Cref{proposition:frequency formula}. Given $u$ solving \eqref{eq:obstacle} in $B_1\subset\R^n$ with $0\in\de\{u>0\}$, we aim to prove that, for any unit vector $e$,
		\begin{equation}\label{eq:goal proof semiconvexity}
			\Delta_{e^\perp} u\ge -\bar C r\quad\text{in }B_r\qquad\text{for each }r\in(0,1/2),
		\end{equation}
		where $\bar C$ is universal. Setting
		\[
		\rho_u:= \inf\{r\in(0,r_1)\,:\,\phi^\tau(r,u-p_\nu)+C_\circ r^{\alpha_\circ}/\alpha_\circ\ge 3-\delta_\circ\text{ for some }\nu\in\mathbb S^{n-1}\},
		\]
		or $\rho_u:=r_1$ if that set is empty, we prove \eqref{eq:goal proof semiconvexity} separately the cases $r>\rho_u$ and $r<\rho_u$.
		
		By definition
		\begin{equation}\label{eq:oki}
			\phi^\tau(\rho_u,u-p_\nu)+C_\circ \rho_u^{\alpha_\circ}/\alpha_\circ\le 3-\delta_\circ\quad \forall \nu\in \S^{n-1}.
		\end{equation}
		Choosing $\bar\nu$ attaining the infimum, \Cref{prop:monneau ND RHS} gives
		\[
		s^{-6}H(s,u-p_{\overline{\nu}}) \le \tfrac12 C_1\quad \forall s\in(\rho_u,1],
		\]
		with $C_1\ge 1$ universal or $C_1=Cr_1^{-6}$, if $\rho_u=r_1$. 
		Then by the semiconvexity estimate \Cref{lem:semiconvexity} applied at scale $s$ with $v=p_{\overline{\nu}}$ we get
		\begin{equation}
			D^2 u \ge -CC_1 s-C s^2 \ge -C_2 s \quad \text{ in } B_s, \quad \forall s\in(\rho_u,1/2],\label{eq:5}
		\end{equation}
		so in particular for any unit vector $e$,
		\[
		\Delta_{e^\perp} u\ge -C_2 s\quad\text{in }B_s,\qquad\forall s\in(\rho_u,1/2].
		\]
		If $\rho_u=0$ the proof ends here, otherwise we consider
		\[
		\eps_1:=\rho_u^{-2}\min_{\nu\in\S^1}H(\rho_u,u-p_\nu)^{1/2}\le  C_1 \rho_u.
		\]
		Because of \eqref{eq:oki} we can apply \Cref{prop:low frequency f} with $r:=\rho_u<\rho(\delta_\circ)$ and find
		\[
		-\Delta_{e^\perp} u(x)\le M(\delta_\circ) \eps_1 |x|/\rho_u \le M(\delta_\circ) C_1 |x|\quad\text{ for all }x\in B_{\rho_u/2}.
		\]
		This, together with \eqref{eq:5}, concludes the proof.
	\end{proof}
	
	Before proving \Cref{theorem:viscosity bound}, we recall a simple fact: if locally $E = \{x_n<\varphi(x')\}$ and $\varphi(0)=|\nabla\varphi(0)|=0$ then $H_{\de E}(0) = -\Delta\varphi(0)$.
	So, fix $z\in \de\{u>0\}\cap B_{1/2}$ with the property that
	\begin{equation}\label{eq:barriercurv}
		\exists E\text{ open with }\de E\in C^2\text{ and }\rho>0 \text{ s.t. } z\in \de E\text{ and } \{u>0\}\cap B_\rho(z) \subset E,
	\end{equation}
	when $z$ is a regular point the inclusion \eqref{eq:barriercurv} and the ellipticity of the mean curvature operator give
	\begin{equation}\label{eq:ellipticity}
		H_{\de E}(z) \le H_{\de\{u>0\}}(z).
	\end{equation}
	\begin{proof}[Proof of \Cref{theorem:viscosity bound}]
		Let $z\in \de\{u>0\}\cap B_{1/2}$ be any point such that \eqref{eq:barriercurv} holds for some $\rho$ and $E$. As $E$ is regular, $z\in\de E$ and $\{u=0\}\supset B_\rho(z)\setminus E$, the contact set $\{u=0\}$ has positive density at $z$. Caffarelli's dichotomy implies that $z$ is a regular free boundary point for $u$.
		
		Since $z$ is regular there are $e_z\in\de B_1$ and $r_z>0$ such that
		\[
		\| r^{-2}u(z+rx)-\tfrac{f(z)}2(x\cdot e_z)_+^2\|_{C^1(B_2)} \le \delta_1<\tfrac{1}{100},\quad \forall r\in(0,4r_z).
		\]
		If $\delta_1$ is chosen sufficiently small, the same holds at any free boundary point close to $z$.
		Thanks to \Cref{lem:bounds regular points} it suffices to bound the mean curvature at points $y\in \{u>0\}$ satisfying $y\to z$. For such points, set $d_y:=\dist(y,\de\{u>0\})$, so that on one hand \Cref{theorem:linear decay semiconvexity} gives $\Delta_{e^\perp}u(y)\ge -C_0 d$ for any unit vector $e$. On the other hand, let $\hat y\in B_{r_z}(z)\cap\de\{u>0\}$ be a closest free boundary point to $y$, and set $\hat u(x):=d^{-2}u(z+d x)$ and $\hat \nu:= (y-\hat y)/d\in \de B_1.$ By the choices of $r_z$, and by $C^1$ closeness, we have
		\[
		d^{-1}|\nabla u(y)| = |\nabla\hat u (\hat\nu)| \ge (\hat \nu\cdot e_{\hat y})_+ -\tfrac{1}{100} =\tfrac{99}{100}.
		\]
		Indeed the minimal choice of $z$ implies that $\hat \nu$ is perpendicular to the free boundary at $\hat y$ and pointing outward, as so is $e_{\hat y}$, thus they must be the same vector.
		
		We conclude recalling that we defined the mean curvature of a superlevel set $\{u>t\}$ of a smooth function by the formula 
		\[
		H_{\de\{u>0\}}(z) = -\divergence\left(\frac{\nabla u}{|\nabla u|}\right) = \frac{-\Delta_{\nabla u^\perp} u(y)}{|\nabla u(y)|} \le \frac{C_0 d}{\tfrac12 d}=2C_0.
		\]
		\Cref{eq:viscositybound} now follows thanks to \eqref{eq:ellipticity}.
	\end{proof}

	\appendix
	\section{Proof of \texorpdfstring{\Cref{corollary:frequency formula}}{}}
	\label{app:frequency}
	
	In this Appendix we adapt some techniques of \cite{FROS24,CinftyRectifiable} to prove the following frequency formula. A similar result was proven in \cite{CinftyRectifiable} assuming $0$ is a singular point.
	\begin{proposition}\label{proposition:frequency formula}
		Let $u$ solve \eqref{eq:obstacle} in $B_1$ with $f$ satisfying \eqref{eq:assumptions f} and $0\in\de \{u>0\}$. Let $\nu\in\mathbb S^{n-1}$. There exist universal constants $\tilde C,\tilde\alpha, C_\circ,\alpha_\circ$ with the following property: let
		\[
		w:= u- p_\nu
		\]
		and suppose that for some $s\in(0,1/2)$ we have
		\[
		W_2(s,w)\ge -\tilde C s\quad\text{and}\quad\phi^{2+1/8}(s,w)+ \tilde C s^{\tilde\alpha}\ge 2+\tfrac1{10}.
		\]
		Then
		\[
		\frac{d}{dr}\phi^\tau(r,w) \ge -C_\circ r^{\alpha_\circ-1}\qquad\text{for any }r\in(s,1/2)
		\]
		and
		\[
		\frac{\int_{B_1}w_r\Delta w_r}{H(r,w) + r^{2\tau}}\ge -C_\circ r^{\alpha_\circ},
		\]
		where we recall that $\tau = 3+\alpha/2$.
	\end{proposition}
	
	\Cref{corollary:frequency formula} is an immediate consequence of \Cref{proposition:frequency formula} and the following observation.
	\begin{remark}
		There is a universal $r_\circ$ such that if $\phi^\tau(s,w) + \tfrac{C_\circ}{\alpha_\circ}s^{\alpha_\circ}\ge 2.5$ for some $s\in(0,r_\circ)$ then
		\[
		W_2(s,w)\ge -\tilde C s\quad\text{and}\quad\phi^{2+1/8}(s,w)+ \tilde C s^{\tilde\alpha}\ge 2+\tfrac1{10}.
		\]
		Indeed, we can rewrite the assumption as
		\[
		D(s,w) \ge (2.5-\tfrac{C_\circ}{\alpha_\circ}s^{\alpha_\circ}) H(s,w) - (\alpha/2+0.5+\tfrac{C_\circ}{\alpha_\circ}s^{\alpha_\circ})s^{2\tau},
		\]
		so choosing $r_\circ$ small we have
		\[
		W_2(s,w) = s^{-4}(D(s,w)-2H(s,w)) \ge -1.5s^{2+\alpha}> - \tilde C s.
		\]
		Similarly, for $r_\circ$ small we have
		\[
		D(s,w) + (2+\tfrac18)s^{2(2+\tfrac18)}- (2+\tfrac18)(H(s,w)+s^{2(2+\tfrac18)}) = D(s,w)-(2+\tfrac18)H(s,w) \ge -2s^{2\tau},
		\]
		so dividing by $H(s,w)+s^{2(2+\tfrac18)}\ge s^{2(2+\tfrac18)}$ gives
		\[
		\phi^{2+\tfrac18}(s,w) \ge (2+\tfrac18) - 2s^{2(\tau-2-1/8)}\ge 2+\tfrac1{10}
		\]
		for $r_\circ$ small enough.
	\end{remark}
	
	The proof first relies on an almost monotonicity of a truncated frequency with a small gain $\tau-2>0$. Then improves the admissible truncations through an iteration.
	
	Since $0\in\de\{u>0\}$ then \eqref{eq:C11 regularity} gives
	\begin{equation}\label{eq:universal bounds rhs}
		|u_r| + |\nabla u_r|\le C r^2.
	\end{equation}
	
	We use the following.
	\begin{lemma}[see {\cite[Lemma 2.3]{FROS24}} and {\cite[proof of Corollary 4.3]{CinftyRectifiable}}]\label{lem:frequency ode}
		For every $C^{1,1}$ function $w$,
		\begin{equation}\label{eq:frequency ode}
			\frac{d}{dr}\phi^\tau(r,w) \ge \frac2r\left(\frac{r^{2-n}\int_{B_r}w\Delta w}{H(r,w)+r^{2\tau}}\right)^2 + \frac{2}{r} \frac{\int_{B_1}\big(\phi^\tau(r,w)w_r-x\cdot\nabla w_r\big)\lap w_r}{H(r,w)+r^{2\tau}} .
		\end{equation}
		Moreover, for every $\lambda\in\R$,
		\begin{equation}\label{eq:monneau ode}
			\frac{d}{dr}\log\big(r^{-2\lambda}(H(r,w)+r^{2\tau})\big) = \frac2r (\phi^\tau(r,w)-\lambda) + \frac2r \frac{\int_{B_1}w_r\lap w_r}{H(r,w)+r^{2\tau}}.
		\end{equation}
	\end{lemma}
	A simple modification of \cite[Lemma A.3]{CinftyRectifiable} yields the following.
	\begin{lemma}\label{lem:monotone weiss}
		There is a universal $C>0$ such that, writing $w = u-p_\nu$, if $0\in\de\{u>0\}$ then
		\[
		\frac{d}{dr}W_2(r,w) \ge -C.
		\]
	\end{lemma}
	\begin{proof}
		The proof is the same as \cite[Lemma A.3]{CinftyRectifiable}. The only modification is \eqref{eq:pde u-p_nu} (so that computations can be carried out with $\delta=1+\alpha$ in their proof), and
		\[
		(2(p_\nu)_r-x\cdot\nabla (p_\nu)_r) \ge -Cr^3.\qedhere
		\]
	\end{proof}
	
	Thanks to \Cref{lem:monotone weiss}, instead of assuming that $0$ is a singular point we can assume a lower bound on the Weiss energy in \cite[Lemma A.6]{CinftyRectifiable}.
	
	\begin{lemma}\label{lem:base step frequency}
		Suppose $W_2(s,w)+C_\circ s\ge 0$ for some $s\in(0,1/2)$. Then for almost every $r\in(s,1/2)$ and each $\tau\in(2,2+1/4)$,
		\[
		2-C_\tau r^{5-2\tau}\le\phi^\tau(r,w)\le C_\tau,\qquad \frac{d}{dr}\phi^{\tau}(r,w) \ge -C_\tau r^{8-4\tau}
		\]
		and
		\[
		\frac{\int_{B_1}w_r\Delta w_r}{H(r,w)+r^{2\tau}}\ge -C_\tau r^{5+\alpha-2\tau}
		\]
		for $C_\tau>0$ depending further on $\tau$.
	\end{lemma}
	\begin{remark}
		We emphasize that here $p_\nu$ is arbitrary. The assumption $W_2(s,w)+C_\circ s\ge 0$ is telling us that, effectively up to scale $s$, the solution $u$ looks singular, but possibly with a quadratic blow-up which is completely different from $p_\nu$.
	\end{remark}
	\begin{proof}
		The proof is the same as \cite[Lemma A.6]{CinftyRectifiable}, with an additional (but higher order) error coming from the fact that $p_\nu$ is not 2-homogeneous.
		
		First, note that if $C_\circ$ is large enough, the condition $W_2(r,w)\ge -C_\circ r$ holds for any $r\in (s,1)$. This can be proved integrating the bound in \Cref{lem:monotone weiss}. As a consequence, following exactly \cite[Lemma A.6]{CinftyRectifiable}, for any $r\in(s,1)$ and any $\tau\ge2$ we have
		\[
		\phi^\tau(r,w) - 2 \ge -C_\circ r^{5-2\tau}.
		\]
		
		For the second inequality, using \Cref{eq:frequency ode}, it suffices to prove the bound
		\[
		\frac{2}{r(H(r,w)+r^{2\tau})} \int_{B_1}(\phi^\tau(r,w)w_r-x\cdot\nabla w_r)\Delta w_r \ge -Cr^{8-4\tau}.
		\]
		Thanks to \cref{eq:pde u-p_nu,eq:universal bounds rhs}
		\[
		|\Delta w_r + r^2f_r\chi_{\{u_r=0\}}| \le Cr^{3+\alpha},\quad |\phi^\tau(r,w)w_r-x\cdot\nabla w_r| \le C(\phi^\tau(r,w)+1)r^2
		\]
		while a direct computation gives
		\[
		\phi^\tau(r,w)w_r-x\cdot\nabla w_r = (2-\phi^\tau(r,w))p_r + O(r^3)\quad\text{in }\{u_r=0\},
		\]
		so, computing as in \cite[Lemma A.6]{CinftyRectifiable} with $\delta=1+\alpha$ one finds
		\[
		\int_{B_1}(\phi^\tau(r,w)w_r-x\cdot\nabla w_r)\Delta w_r \ge -Cr^{9-2\tau} - Cr^5 - Cr^{5+\alpha}(\phi^\tau(r,w)+1).
		\]
		Thus,
		\[
		\frac{d}{dr}\phi^\tau(r,w) \ge -Cr^{8-4\tau} - Cr^{4-2\tau} - Cr^{4+\alpha-2\tau}(\phi^\tau(r,w)+1)\ge -Cr^{8-4\tau}(1+\phi^{\tau}(r,w)).
		\]
		It follows that $\log(1+\phi^\tau(r,w))$ is almost monotone when $\tau\in(2,2+1/4)$, which implies
		\[
		\phi^\tau(r,w)\le C_\tau\quad\text{and}\quad \frac{d}{dr}\phi^\tau(r,w) \ge -C_\tau r^{8-4\tau},
		\]
		as we wanted.
		
		Similarly, the last bound follows using \eqref{eq:pde u-p_nu}.
	\end{proof}
	
	We note the following simple but crucial consequence.
	\begin{lemma}\label{lem:apply monneau}
		Suppose for some $s\in(0,1)$ we have
		\[
		W_2(s,w) \ge -C_\circ s\quad\text{and}\quad\phi^{2+1/8}(s,w) \ge 2+1/10-C_\circ s^{\alpha_\circ},
		\]
		where $C_\circ,\alpha_\circ$ are given by \Cref{lem:monotone weiss,lem:base step frequency}. Then
		\[
		\|w_r\|_{L^2(\de B_1)}\le Cr^{2+1/10}\quad \text{for any }r\in(s,1).
		\]
	\end{lemma}
	\begin{proof}
		Integrating the lower bound in \Cref{lem:base step frequency} gives
		\[
		\phi^{2+1/8}(r,w)-(2+1/10)\ge-C_\circ r^{\alpha_\circ}\quad\text{for any }r\in(s,1).
		\]
		Thus \cref{eq:monneau ode} and \Cref{lem:base step frequency} yield, for a. e. $r\in(s,1)$,
		\[\begin{split}
			\frac{d}{dr}\log(r^{-2(2+1/10)}(H(r,w)+r^{2(2+1/8)})) &\ge \frac{2}{r}(\phi^{2+1/8}(r,w)-(2+1/10)) - Cr^{-1/4}\\
			&\ge -Cr^{\alpha_\circ-1}.
		\end{split}\]
		Integrating between $r$ and $1$ gives
		\[
		r^{-2(2+1/10)}(\|w_r\|^2_{L^2(\de B_1)}+r^{2(2+1/8)})\le C
		\]
		as we wanted.
	\end{proof}
	
	We can finally prove \Cref{proposition:frequency formula}.
	\begin{proof}[Proof of \Cref{proposition:frequency formula}]
		This proof relies on an iteration argument, where the base step is given by \Cref{lem:base step frequency}.
		
		\medskip\noindent\textbf{Step 1.} For any $\tau>0$ there is a universal $C_\tau>0$ such that
		\begin{equation}\label{eq:intermediate step freuqneyc formula}\begin{split}
				\frac{d}{dr}\phi^\tau(r,w) \ge \frac{2}{r}\frac{(r^{2-n}\int_{B_r}w\Delta w)^2}{(H(r,w)+r^{2\tau})^2}&- C r^{2+\alpha-\tau}(\phi^\tau(r,w)+1)\frac{\|w_r\|_{L^2(\de B_2)}+r^{3+\alpha}}{(H(r,w)+r^{2\tau})^{1/2}}\\
				&- Cr^{1/10-1}\frac{(\|w_r\|_{L^2(\de B_2)} + r^{3+\alpha})^2}{H(r,w) + r^{2\tau}}
		\end{split}\end{equation}
		and
		\begin{equation}\label{eq:intermediate step frequency formula 2}
			\frac{r^{2-n}\int_{B_r}w\Delta w}{H(r,w)+r^{2\tau}} \ge -Cr^{3+\alpha-\tau}\frac{\|w_r\|_{L^2(\de B_2)}+r^{3+\alpha}}{(H(r,w)+r^{2\tau})^{1/2}}.
		\end{equation}
		
		By \cref{eq:frequency ode}, it suffices to bound from below
		\[
		\frac{2}{r(H(r,w)+r^{2\tau})}\int_{B_1}(\phi^\tau(r,w)w_r-x\cdot\nabla w_r)\Delta w_r.
		\]
		First, integrating \eqref{eq:pde u-p_nu} in $B_1$ and using \Cref{lem:local bounds rhs} gives
		\[
		\int_{B_1} w_r\Delta w_r \ge -Cr^{3+\alpha}(\|w_r\|_{L^2(\de B_2)}+r^{3+\alpha}).
		\]
		To bound the other term we use \eqref{eq:pde u-p_nu} to write
		\[
		\int_{B_1}(x\cdot\nabla w_r)\Delta w_r = \overbrace{\int_{B_1}O(r^{3+\alpha})(x\cdot\nabla w_r)}^A + \overbrace{\int_{B_1\cap\{u_r=0\}}(x\cdot\nabla (p_\nu)_r) r^2f(r\cdot)}^B.
		\]
		As \eqref{eq:lip estimates} reads $|\nabla w_r|\le C(\|w_r\|_{L^2(\de B_2)}+r^{3+\alpha})$, we have
		\[
		A \le Cr^{3+\alpha}(\|w_r\|_{L^2(\de B_2)}+r^{3+\alpha}).
		\]
		On the other hand, a computation gives
		\[
		| x\cdot\nabla (p_\nu)_r| \le 3(p_\nu)_r\le Cr^2 |x\cdot\nu|^2,
		\]
		while \eqref{eq:bound contact set} and \Cref{lem:apply monneau} give $B_1\cap\{u_r=0\} \subset \{|x\cdot\nu| \le Cr^{-2}(\|w_r\|_{L^2(\de B_2)}+r^{3+\alpha})\}$ and $\|w_r\|_{L^2(\de B_2)} \le C r^{2+\tfrac1{10}}$. So,
		\[\begin{split}
			B &\le Cr^4\int_{\{|x\cdot\nu|\le Cr^{-2}(\|w_r\|_{L^2(\de B_2)}+r^{3+\alpha})\}} |x\cdot\nu|^2\\
			&\le C \frac{\|w_r\|_{L^2(\de B_2)} + r^{3+\alpha}}{r^2} (\|w_r\|_{L^2(\de B_2)}+r^{3+\alpha})^2\\
			&\le Cr^{1/10}(\|w_r\|_{L^2(\de B_2)}+r^{3+\alpha})^2,
		\end{split}\]
		and the claim follows.
		
		\medskip\noindent\textbf{Step 2.} For all $2+1/10<\tau<3+\alpha$ we have
		\begin{equation}\label{eq:goal proof frequency formula}
			\phi^\tau(r,w) \le C_\tau\quad\text{and}\quad \|w_r\|_{L^2(\de B_2)}\le C_\tau(H(r,w)+r^{2\tau})^{1/2}.
		\end{equation}
		
		The case $\tau = 2 + 1/10$ follows by \Cref{lem:base step frequency}, together with \Cref{lem:apply monneau}. So, following \cite[Lemma 4.3, Step 2]{FROS24} it suffices to show that if the claim holds for some $\tau$, then it holds also for $\tau+\beta$ for any $\beta<\min\{1/20,(3+\alpha-\tau)/4\}$. Step 2 then follows by iterating this claim finitely many times.
		
		Following \cite[Lemma 4.3, Step 2]{FROS24} we find
		\[
		\phi^{\tau+\beta}(r,w)\le 2C_\tau r^{-2\beta}\qquad\text{and}\qquad \frac{\|w_r\|_{L^2(\de B_2)}}{(H(r,w)+r^{2(\tau+\beta)})^{1/2}} \le \frac{1}{r^{\beta}} C_\tau.
		\]
		Then Step 1 (with $\tau$ replaced by $\tau+\beta$) yields
		\[
		\frac{d}{dr}\phi^{\tau+\beta}(r,w)\ge \frac2r\left(\frac{r^{2-n}\int_{B_r}w\Delta w}{H(r,w)+r^{2\tau}}\right)^2 - Cr^{2+\alpha-\tau-4\beta}-Cr^{1/10-1-2\beta}.
		\]
		This implies that $\phi^{\tau+\beta}$ is almost monotone, provided $\beta< \min\{1/20,(3+\alpha-\tau)/4\}$, which yields
		\[
		\phi^{\tau+\beta}(r,w) \le C_{\tau+\beta}.
		\]
		In addition, \cite[Lemma 4.1 (a)]{FROS24} implies
		\[
		\|w_r\|_{L^2(\de B_2)}^2 \le C(H(r,w)+r^{2\tau+2\beta}),
		\]
		so \eqref{eq:goal proof frequency formula} follows. The result now follows using these bounds into \cref{eq:intermediate step freuqneyc formula,eq:intermediate step frequency formula 2}.
	\end{proof}
	
	\section{Proof of \texorpdfstring{\Cref{lem:caffarelli rhs holder}}{}}\label{app:caffarelli}
	
	\begin{lemma}[Local Glaeser Inequality]\label{lem:gleaser}
		Let $w\in C^{1,1}(B_1)$ be a nonnegative function and let 
		$$Z:=\{w=0\}\cap B_{1/2}.$$ 
		Assume that
		\[
		D^2_{ee} w(x)\le M\quad \text{ for all }e\in\S^{n-1} \text{ and  for a.e. }x\in B_1.
		\]
		Then
		\[
		|\nabla w(x)|^2 \le 2M w(x) \quad \text{ whenever } \dist(x,Z)\le\tfrac14. 
		\]
	\end{lemma}
	\begin{proof}
		Let $x$ satisfy $\delta := \operatorname{dist}(x, Z) \le \frac{1}{4}$. If $\delta=0$ the assertion is trivial, so we may assume $w(x)>0$. Choose $z \in Z$ such that $|x - z| = \delta$. Since $w \ge 0$ and $w(z) = 0$, $\nabla w(z) = 0$.
		By Taylor's theorem along the segment $[z, x]$, the condition $D^2 w \le M $ gives
		\[
		w(x) \le \tfrac12{M}\delta^2.
		\]
		
		If $\nabla w(x) = 0$, the conclusion holds trivially. Otherwise, let $e := -{\nabla w(x)}/{|\nabla w(x)|}$. For $t > 0$, Taylor's theorem along the ray $x + te$ yields
		\[
		0 \le w(x + te) \le w(x) - t|\nabla w(x)| + \tfrac12{M}{}t^2.
		\]
		
		Evaluating this inequality at $t_0 := \sqrt{{2w(x)}/{M}}\le \delta \le \frac{1}{4}$ gives
		\[
		t_0 |\nabla w(x)| \le w(x) + \tfrac12{M}{}t_0^2 = 2w(x),
		\]
		where we used that $x+t_0 e\in B_{1/2}+B_\delta +B_{\delta} \subset B_1$.
		Dividing by $t_0$ and squaring both sides yields the result.
	\end{proof}
	
	\begin{lemma}[Caffarelli with $C^{1,\alpha}$ RHS]
		There are positive constants $A_1$ large and $\sigma_1,$ $\rho_1$ small with the following property.
		
		Let $u$ solve \eqref{eq:obstacle} in $B_1$ for some $f\ge\mu>0$.
		Let $h\colon \{u>0\}\cap B_1\to\R$ satisfy
		\[\begin{cases}
			h\ge0   &\text{on }\de\{u>0\},\\
			\Delta h\le \divergence F   &\text{in }\{u>0\},\\
			h\ge -1         &\text{in }\{u>0\},\\
			h\ge A_1        &\text{in }\{\dist(\cdot,\de\{u>0\})>\sigma_1\},
		\end{cases}\]
		for some $F\in C^\alpha(B_1;\R^n),$ with $\alpha>0$.
		Then
		\[
		h\ge 0\quad\text{ in }\{u>0\}\cap B_{\rho_1}.
		\]
		The constants $A_1$, $\sigma_1$ and $\rho_1$ depend on
		\begin{equation}\label{eq:dependencies_caffarelli proof}
			n,\quad\mu,\quad\alpha,\quad \|F\|_{C^\alpha(B_1)},\quad \|f\|_{C^\alpha(B_1)}.
		\end{equation}
	\end{lemma}
	\begin{proof}
		We denote by $C$ and $c$ constants depending on \eqref{eq:dependencies_caffarelli proof} and set
		\[
		d(x):=\dist(x,\de\{u>0\}).
		\]
		Let $\rho,\sigma$ and $A$ to be chosen later.  
		Pick any $x_0\in \{u>0\}\cap B_\rho$ and assume $d_0:=d(x_0)\le \sigma\rho$, otherwise $h(x_0)\ge A>0$. This means that there is $z_0\in\de\{u>0\}$ such that
		\[
		|z_0-x_0|=d_0\le \sigma \rho.
		\]
		The idea is to construct a suitable lower barrier for $h$ in $\{u>0\}\cap B_{2\rho}(x_0)$ of the form
		\begin{equation*}
			g(x):=a\, u(x)^{\frac{1+\alpha}{2}} +b\, q(x)+\Phi(x), 
		\end{equation*}
		where 
		\begin{itemize}
			\item $a,b$ are constants to be chosen positive;
			\item $\Phi\in C^{1,\alpha}(B_{2\rho}(x_0)))$ is a corrector such that
			\[
			\lap \Phi =\divergence F,\quad |\Phi(x)|\le C_1|x-z_0|^{1+\alpha}\quad \forall x\in B_{2\rho(x_0)}.
			\]
			To construct $\Phi$ is enough to solve the Dirichlet problem in $B_{2\rho(x_0)}$ and then subtract the Taylor expansion at $z_0$.
			\item The paraboloid-type function $q$ is given by
			\[
			q(x):=-(|x-x_0|^2+d_0^2)^p+\kappa_p d^{2p}_0,
			\]
			where we set for brevity
			\[
			p:=\frac{1+\alpha}{2},\quad p\in(0,1),\quad \kappa_p:=\frac{1+2^p}{2}\in(1,2^p).\]
		\end{itemize}
		Notice that 
		$$g(x_0)\ge\big( b(\kappa_p-1)-C_1\big)d_0^{2p},$$
		so this barrier will show $h(x_0)\ge 0$ provided
		\begin{equation}
			b\ge {C_1}/(\kappa_p-1){}\label{eq:c0}\tag{c0}.
		\end{equation}
		
		\textbf{Step 1.} We show that $\lap g \ge \lap h$ in $B_{\rho}(x_0)\cap\{u>0\}$.
		
		A direct computation gives that, in $B_{2\rho}(x_0)$,
		\begin{align*}
			\lap q &=-2p(r^2+d_0^2)^{p-2}[nd_0^2 +(n+\alpha-1)r^2]\\
			&\ge -C_0 (r^2+d_0^2)^{p-1},
		\end{align*}
		where $r:=|x-x_0|$.
		
		To counter this term we estimate from below $\lap(u^p)$, which is
		\[\lap (u^p) = p u^{p-1} \Big(f(x) - (1-\alpha)\frac{|\nabla u|^2}{2u}\Big).\]
		Notice that, for all $x\in B_{2\rho}(x_0)$,
		\[
		\dist\big(x,\de\{u>0\}\cap B_{2\rho}(x_0)\big)\le |x-z_0|\le r+d_0\le 3\rho
		\]
		Thus by \Cref{eq:C11 regularity} and recalling that $p<1$ we find
		\begin{equation}
			u(x)^{p-1}\ge c(r^2+d_0^2)^{p-1}.\label{eq:up-1}
		\end{equation}
		Applying \Cref{lem:gleaser} to $w:=(20\rho)^{-2}u(x_0+20\rho\cdot)$ we find
		\[
		|\nabla u|^2 \le 2M u\text{ in } B_{2\rho}(x_0),
		\]
		with $M$ given by
		\[
		M:=\sup_{x\in B_{20\rho}(x_0),\ e\in \S^{n-1}}D_{ee} u(x).
		\]
		Using the equation and the weak semiconvexity estimate \eqref{eq:semiconvexity_weak} we find
		\[
		M\le  (n-1)\omega(21\rho)+\sup_{B_{20\rho}(x_0)}f.
		\]
		This proves that, if 
		\begin{equation}\label{eq:c1}
			\rho<c_1    \tag{c1}
		\end{equation} 
		universal,
		\begin{align*}
			\Big(f(x) - (1-\alpha)&\frac{|\nabla u|^2}{2u}\Big)\ge
			\Big(f(x) - (1-\alpha)M\Big)\\
			&\ge \alpha\mu  -(1-\alpha) \osc_{B_{20\rho}(x_0)}f -(n-1)\omega(21\rho)\ge \tfrac\alpha2 \mu.
		\end{align*}
		Then combining this with \eqref{eq:up-1}  we find for some $c_0>0$,
		\[
		\lap(u^p)\ge c_0(r^2+d_0^2)^{p-1}\quad \forall x\in B_{2\rho(x_0)}.
		\]
		Putting things together:
		\[
		\lap g \ge (ac_0-bC_0)(r^2+d_0^2)^{p-1}+\divergence F\ge \lap h \quad \forall x\in B_{2\rho(x_0)},
		\]
		provided
		\begin{equation}
			a>\frac{C_0}{c_0}b.\label{eq:c2}\tag{c2}
		\end{equation}
		\textbf{Step 2.} We show that $g\le h$ on $\de(B_{\rho}(x_0)\cap \{u>0\})$.
		
		We start considering $x\in B_{\rho}(x_0)\cap \de \{u>0\}.$ Here $q$ is quite negative so it helps us, more precisely $|x-x_0|\ge d_0$ so
		\[
		q(x)\le -\big(1-\frac{\kappa_p}{2^p}\big)(|x-x_0|^2+d_0^2)^{p} +\kappa_p\underbrace{\big[d_0^{2p}-\frac{1}{2^p}(|x-x_0|^2+d_0^2)^{p}\big]}_{\le 0}
		\]
		so we find for some $c_2>0$,
		\begin{equation*}
			g(x)\le C_1 |x-z_0|^{2p} - b c_2(|x-x_0|^2+d_0^2)^{p}<0\le h(x),
		\end{equation*}
		as soon as
		\begin{equation}
			\label{eq:c3} b\ge C_1/c_2. \tag{c3}
		\end{equation}
		
		We go on and consider $x\in  \de B_{\rho}(x_0)\cap \{0<d<\sigma\rho\}.$ Here by quadratic contact (see \eqref{eq:C11 regularity}) we have
		\[
		u(x)^p\le C_2\sigma^{2p}\rho^{2p}.
		\]
		If $\sigma$ is small, then $q$ helps us again, recalling that $d_0\le \sigma\rho$:
		\[
		q(x)\le -(\rho^2+d_0^2)^p+\kappa_p d_0^{2p}\le -(1-\kappa_p \sigma^{2p})\rho^{2p} \le -\tfrac12\rho^{2p},
		\]
		provided 
		\begin{equation}\label{eq:c4}
			\sigma\le (2\kappa_p)^{-1/2p}.\tag{c4}
		\end{equation}
		Thus we have
		\[
		g(x)\le  C_2 a\sigma^{2p}\rho^{2p} - \tfrac b2\rho^{2p} +C_1 |x-z_0|^{2p}
		\le  \big(C_2 a\sigma^{2p}- \tfrac 12 b+9^pC_1 \big)\rho^{2p}\le -1\le h(x)
		\]
		if 
		\begin{equation}\label{eq:c5}
			b\ge 2\rho^{-2p}+20C_1+2C_2 a\sigma^{2p}. \tag{c5}
		\end{equation}
		
		Finally, we consider $x\in  \de B_{\rho}(x_0)\cap \{d>\sigma\rho\}.$ 
		Here we need rough upper bounds since $h(x)\ge A$, so we just bound
		\begin{equation}
			g(x)\le a C\rho^{2p}+ b \rho^{2p} +9^pC_1\rho^{2p}=:A\le h(x).\label{eq:c7}
		\end{equation}
		To complete the proof we need to show that there are choices of $\rho,$ $\sigma$, $a$ and $b$, compatible with \eqref{eq:c0}, \eqref{eq:c1}, \eqref{eq:c2}, \eqref{eq:c3}, \eqref{eq:c4} and \eqref{eq:c5}.
		
		First we choose $\rho_1$ such that \eqref{eq:c1} holds; then we choose $b_1$ in such a way that \cref{eq:c0,eq:c3} hold and additionally:
		\begin{equation}\label{eq:c6}
			b_1\ge 2\rho^{-2p}_1+20C_1+2C_2. \tag{c5'}
		\end{equation}
		Then we choose $a_1$ in such a way that \eqref{eq:c2} holds; then we choose $\sigma_1$ in such a way that \eqref{eq:c4} holds and additionally:
		\[
		a_1\sigma^{2p}_1\le 1.
		\]
		This and \eqref{eq:c6} imply that $b_1$ satisfies also \eqref{eq:c5}. $A_1$ is determined then by \eqref{eq:c7}. (In the statement we called $\sigma_1$ what is $\sigma_1\rho_1$ in the proof).
	\end{proof}
	
	\section{A counterexample to the exterior ball property}\label{ex:no ball property}
	We modify some examples of free boundaries in \cite{Schaeffer} to construct solutions in $B_1\subset\R^2$ with $C^\infty$ obstacle, but such that $\{u=0\}$ doesn't have the exterior $r$-ball property for any $r>0$ (as defined in \Cref{rmk:exterior ball}). We will use the following generic regularity statement.
	\begin{theorem}[\cite{FROS24}]\label{theorem:generic regularity}
		Let $\Omega\subset\R^2$ be a bounded domain and let $u$ solve \eqref{eq:obstacle} in $\Omega$. Then for each $\eps>0,$ there is $u_\eps$ solving \eqref{eq:obstacle} in $\Omega,$ with $\de\{u_\eps>0\}$ regular and $u-\eps\le u_\eps\le u$. In particular, regular free boundaries are dense in the local Hausdorff distance.
	\end{theorem}
	By \cite[Theorem 2.1]{Schaeffer} there is a positive function $f\in C^\infty(B_1)$ and $u\colon B_1\to[0,+\infty)$ solving
	\[\Delta u = f\,\chi_{\{u>0\}}\qquad\text{in }B_1\subset\R^2
	\] such that:
	\begin{enumerate}[label=(\roman*)]
		\item $u$ is even in $x_2$;
		\item $\{u=0\}\cap\{x_2=0\} = \cup_{m\ge1} I_m \cup\{0\}$ where $I_m=[-2^{-2m},-2^{-2m-1}]$;
		\item $\{u=0\}\cap B_r\subset\{|x_2| = o(r)\}$.
	\end{enumerate}
	
	Note that this is not yet a counterexample, because the end points $(-2^{-2m},0)$, $(-2^{-2m-1},0)$ are \textit{singular}, which do enjoy the exterior ball property with a uniform radius. So, we perturb this example to have a regular free boundary.
	
	Given $x\in \R^2$ and $r>0$ we write $Q_r(x) := x+[-r,r]^2$. Choose $x_m,$ $r_m$ so that for $Q_m:=Q_{r_m}(x_m)$ it holds:
	\begin{enumerate}[label=(\roman*)]
		\item $I_m\times\{0\}\subset Q_m$;
		\item $Q_m\cap Q_j=\emptyset$ for $m\neq j$;
		\item for $m$ large enough $\{u=0\}\cap Q_m\subset Q_{r_m(1-1/50)}(x_m)$.
	\end{enumerate}
	
	Set $\eps_m := \exp(-1/r_m)$ and choose a solution $v_m$ of the obstacle problem in $Q_m$ with RHS $f$ such that:
	\begin{enumerate}[label=(\roman*)]
		\item $v_m$ is even in $x_2$;
		\item $|v_m-u|<\eps_m$ on $\de Q_m$;
		\item $0<v_m<u$ on $\de Q_m$;
		\item $\de\{v_m>0\}\subset Q_{r_m(1-1/100)}(x_m)$ is a regular free boundary,
	\end{enumerate}
	which is possible by \Cref{theorem:generic regularity}.
	
	Fix a smooth nonnegative cutoff function $\eta\in C^\infty_c(Q_{1-1/1000})$ and satisfying $\eta\equiv1$ on $Q_{1-1/200}$ and $0\le\eta\le1$.
	We consider the function
	\[
	v(x) = u(x) + \sum_{m\ge1} \eta\left(\tfrac{x-x_m}{r_m}\right)(v_m(x)-u(x)),
	\]
	which is well defined since the series is locally finite, and we prove that it solves the obstacle in $B_1$ with positive $C^\infty$ RHS. Note that this is true in $B_1\setminus\{0\}$. Note also that this coincides with $f$ in $B_1\setminus\big(\cup_mQ_m\setminus Q_{r_m(1-1/100)}(x_m)\big)$. Since
	\[
	\Delta (v_m - u) = 0\quad\text{and}\quad |v_m - u|<\eps_m\qquad\text{in }Q_m\setminus Q_{r_m(1-1/100)}(x_m)
	\]
	we have
	\[
	\|v_m - u\|_{C^k(Q_m\setminus Q_{r_m(1-1/100)}(x_m))}< C_k r_m^{-k}\eps_m \to 0,
	\]
	as desired.
	
	Now note that $\{v=0\}$ has infinitely many regular connected components $\mathcal C_m$ accumulating at $0$. This implies that $\dist(\mathcal C_m,\mathcal C_{m+1})\to0$. Let $x_m\in\mathcal C_m,$ $y_m\in\mathcal C_{m+1}$ be such that $|x_m-y_m| = \dist(\mathcal C_m,\mathcal C_{m+1})$. Then $x_m-y_m$ is orthogonal to $\mathcal C_m$ at $x_m$, so any ball touching $\mathcal C_m$ at $x_m$ will have the center lying in the segment $\{x_m + t(y_m-x_m) : t\in[0,1]\}$ (this is because $\mathcal C_m$ is smooth). As $u(y_m)=0$, the largest ball touching $\mathcal C_m$ at $x_m$ must be centered at $\frac12(x_m+y_m)$ with radius $\frac12|x_m-y_m|\to 0$.

	\printbibliography
\end{document}